\pdfoutput=1
\documentclass[letterpaper, oneside, reqno, 10pt]{amsart}
\usepackage[margin=1in]{geometry}

\usepackage{amsmath,amssymb,mathtools}
\usepackage[foot]{amsaddr}

\usepackage{placeins}

\usepackage{amssymb} 
\usepackage{stmaryrd}
\SetSymbolFont{stmry}{bold}{U}{stmry}{m}{n}

\usepackage{graphicx}

\usepackage{setspace}
\usepackage[colorlinks=true, pdfstartview=FitV, linkcolor=blue, citecolor=Green, urlcolor=WildStrawberry, linktoc=page]{hyperref} %

\usepackage{array}
\usepackage{ragged2e}

\usepackage{bm}

\usepackage{dsfont}
\usepackage[T1]{fontenc}
\usepackage[utf8]{inputenc}

\DeclareFontFamily{OT1}{rsfs}{}
\DeclareFontShape{OT1}{rsfs}{n}{it}{<-> rsfs10}{}
\DeclareMathAlphabet{\mathscr}{OT1}{rsfs}{n}{it}
\usepackage{mathrsfs}
\usepackage{MnSymbol}

\usepackage{silence}
\usepackage[final,nopatch=footnote]{microtype} %

\usepackage{enumitem}
\usepackage[dvipsnames]{xcolor}

\usepackage[normalem]{ulem}

\usepackage{booktabs} %

\usepackage{comment}

\usepackage{braket}

\usepackage{etoolbox}
\newtoggle{focs}
\toggletrue{focs}
\newcommand{\iffocs}[2]{\iftoggle{focs}{#1}{#2}}

\definecolor{darkgreen}{rgb}{0,0.5,0}
\definecolor{darkblue}{rgb}{0,0,0.7}
\definecolor{darkred}{rgb}{0.9,0.1,0.1}

\newcommand{\js}[1]{{\color{ForestGreen}{[JS: #1]}}}

\newcommand{\holden}[1]{{\color{purple}{[HL: #1]}}}

\renewcommand{\js}[1]{}
\renewcommand{\holden}[1]{}

\newlength{\bibitemsep}
\let\oldthebibliography\thebibliography
\renewcommand\thebibliography[1]{%
  \oldthebibliography{#1}%
  \setlength{\parskip}{\bibitemsep}%
  \setlength{\itemsep}{-7pt}%
}

\newtheoremstyle{break}%
{}{}%
{\itshape}{}%
{\bfseries}{.\vphantom{$p_{p_{p_p}}$}}%
{\newline}
{\thmname{#1}\thmnumber{ #2}\thmnote{\ \,\textmd{(#3)}}}

\theoremstyle{break}

\newtheorem{proposition}{Proposition}
\newtheorem{theorem}[proposition]{Theorem}
\newtheorem{lemma}[proposition]{Lemma}
\newtheorem{corollary}[proposition]{Corollary}

\theoremstyle{remark}
\newtheorem{remark}[proposition]{Remark}

\theoremstyle{definition}
\newtheorem{definition}[proposition]{Definition}

\newtheorem{desideratum}{Desideratum}

\newtheorem*{lemma*}{Lemma}

\newcommand{\vocab}[1]{\emph{#1}}

\numberwithin{equation}{section}
\numberwithin{proposition}{section}
\numberwithin{figure}{section}
\numberwithin{table}{section}

\newcommand{\Z}{\mathbb{Z}}

\newcommand{\R}{\mathbb{R}}

\newcommand{\calN}{\mathcal N}

\renewcommand{\P}{\mathop{{}\mathbb{P}}}
\renewcommand{\Pr}{\P}
\newcommand{\Cov}{\mathop{{}\boldsymbol{\mathrm{Cov}}}}
\newcommand{\E}{\mathop{{}\mathbb{E}}}

\newcommand{\hQ}{\widehat{Q}}
\renewcommand{\hm}{\widehat{m}}
\newcommand{\hx}{\widehat{x}}
\newcommand{\hy}{\widehat{y}}
\newcommand{\hw}{\widehat{w}}

\newcommand{\hp}{\widehat{p}}
\newcommand{\hf}{\widehat{f}}
\newcommand{\tx}{\widetilde{x}}
\newcommand{\ty}{\widetilde{y}}
\newcommand{\tw}{\widetilde{w}}

\renewcommand{\le}{\leqslant}
\renewcommand{\ge}{\geqslant}
\renewcommand{\leq}{\leqslant}
\renewcommand{\geq}{\geqslant}

\renewcommand{\subset}{\subseteq}
\renewcommand{\bar}{\overline}

\renewcommand{\tilde}{\widetilde}
\newcommand{\td}{\widetilde}

\renewcommand{\hat}{\widehat}

\newcommand{\subeq}{\subseteq}

\newcommand{\be}{\beta}
\newcommand{\ga}{\gamma}
\newcommand{\de}{\delta}

\newcommand{\lm}{\lambda}

\newcommand{\ph}{\varphi}
\newcommand{\De}{\Delta}
\newcommand{\ep}{\varepsilon}
\newcommand{\eps}{\varepsilon}
\newcommand{\si}{\sigma}
\newcommand{\Si}{\Sigma}
\newcommand{\om}{\omega}
\newcommand{\Om}{\Omega}

\newcommand{\rh}{\rho}

\newcommand{\mg}[0]{m}
\newcommand{\ub}[2]{\underbrace{#1}_{#2}}

\newcommand{\dd}[2]{\frac{d #1}{d #2}}

\newcommand{\ddd}[1]{\frac{d}{d #1}}

\DeclareMathOperator{\dist}{\mathcal{L}}

\DeclareMathOperator{\tr}{tr}

\DeclareMathOperator{\sign}{sign}

\newcommand{\Tr}{\mathsf{Tr}}

\newcommand{\KL}{\operatorname{KL}}
\newcommand{\TV}{\operatorname{TV}}

\newenvironment{e*}{\begin{equation*}}{\end{equation*}\ignorespacesafterend}
\newcommand{\norm}[1]{\left\lVert{#1}\right\rVert}

\renewcommand{\Tr}{\mathsf{Tr}}

\newcommand{\TAP}{{\mathrm{TAP}}}
\newcommand{\FT}{\calF_{\,\mathrm{TAP}}}

\newcommand{\ip}[2]{\langle#1, #2\rangle}
\newcommand{\iprod}[1]{\langle#1\rangle}

\newcommand{\an}[1]{\left\langle#1\right\rangle}

\newcommand{\sT}{\mathsf{T}}

\newcommand{\Lip}{\mathsf{Lip}}
\newcommand{\ot}{\otimes}

\newcommand{\ba}[1]{\left[ {#1} \right]}
\newcommand{\bc}[1]{\left\{ {#1} \right\}}
\newcommand{\pa}[1]{\left( {#1} \right)}
\newcommand{\ve}[1]{\left\Vert {#1}\right\Vert}

\newcommand{\fc}[2]{\frac{#1}{#2}}
\newcommand{\rc}[1]{\frac{1}{#1}}

\newcommand{\pf}[2]{\pa{\frac{#1}{#2}}}
\newcommand{\prc}[1]{\pa{\frac{1}{#1}}}

\newcommand{\sumo}[2]{\sum_{#1=1}^{#2}}

\newcommand{\gd}[0]{\nabla}

\DeclareMathOperator{\sech}{sech}

\newcommand{\Id}[0]{I}
\DeclareMathOperator{\id}{id}

\DeclareMathOperator{\Spec}{Spec}

\renewcommand{\tr}{\operatorname{tr}}
\DeclareMathOperator{\diag}{\mathsf{diag}}

\renewcommand{\norm}[1]{\left\lVert{#1}\right\rVert}

\newcommand{\sop}{_\mathsf{op}}
\newcommand{\opnorm}[1]{\ensuremath{\left\lVert #1 \right\rVert\sop}}
\newcommand{\ED}{E_{\mathcal{D}_n}}
\newcommand{\lpnorm}[2][2]{\ensuremath{\left\lVert {#2} \right\rVert_{#1}}}
\newcommand{\schnorm}[2][2]{\ensuremath{\left\lVert {#2} \right\rVert_{L^{#1}}}}
\newcommand{\Lpnorm}[2][2]{\ensuremath{\left\lVert {#2} \right\rVert_{L^{#1}}}}
\newcommand{\lipnorm}[1]{\ensuremath{\left\lVert {#1} \right\rVert_{\Lip}}}

\newcommand{\poly}{\mathsf{poly}}

\renewcommand{\le}{\leqslant}
\renewcommand{\leq}{\leqslant}
\renewcommand{\ge}{\geqslant}
\renewcommand{\geq}{\geqslant}

\usepackage{prettyref}
\newcommand{\savehyperref}[2]{\texorpdfstring{\hyperref[#1]{#2}}{#2}}

\protected\def\verythinspace{%
  \ifmmode
    \mskip0.5\thinmuskip
  \else
    \ifhmode
      \kern0.083em
    \fi
  \fi
}
\newrefformat{eq}{\savehyperref{#1}{\textup{(\ref*{#1})}}}
\newrefformat{e}{\savehyperref{#1}{\textup{(\ref*{#1})}}}
\newrefformat{ineq}{\savehyperref{#1}{\textup{(\ref*{#1})}}}
\newrefformat{eqn}{\savehyperref{#1}{\textup{(\ref*{#1})}}}
\newrefformat{l}{\savehyperref{#1}{Lemma~\ref*{#1}}}
\newrefformat{lem}{\savehyperref{#1}{Lemma~\ref*{#1}}}
\newrefformat{def}{\savehyperref{#1}{Definition~\ref*{#1}}}
\newrefformat{d}{\savehyperref{#1}{Definition~\ref*{#1}}}
\newrefformat{t}{\savehyperref{#1}{Theorem~\ref*{#1}}}
\newrefformat{thm}{\savehyperref{#1}{Theorem~\ref*{#1}}}
\newrefformat{cor}{\savehyperref{#1}{Corollary~\ref*{#1}}}
\newrefformat{c}{\savehyperref{#1}{Corollary~\ref*{#1}}}
\newrefformat{cha}{\savehyperref{#1}{Chapter~\ref*{#1}}}
\newrefformat{sec}{\savehyperref{#1}{\S\verythinspace\ref*{#1}}}
\newrefformat{s}{\savehyperref{#1}{\S\verythinspace\ref*{#1}}}
\newrefformat{subsec}{\savehyperref{#1}{\S\verythinspace\ref*{#1}}}
\newrefformat{app}{\savehyperref{#1}{\S\verythinspace\ref*{#1}}}
\newrefformat{tab}{\savehyperref{#1}{Table~\ref*{#1}}}
\newrefformat{fig}{\savehyperref{#1}{Figure~\ref*{#1}}}
\newrefformat{hyp}{\savehyperref{#1}{Hypothesis~\ref*{#1}}}
\newrefformat{alg}{\savehyperref{#1}{Algorithm~\ref*{#1}}}
\newrefformat{a}{\savehyperref{#1}{Algorithm~\ref*{#1}}}
\newrefformat{rem}{\savehyperref{#1}{Remark~\ref*{#1}}}
\newrefformat{item}{\savehyperref{#1}{Item~\ref*{#1}}}
\newrefformat{step}{\savehyperref{#1}{step~\ref*{#1}}}
\newrefformat{conj}{\savehyperref{#1}{Conjecture~\ref*{#1}}}
\newrefformat{fact}{\savehyperref{#1}{Fact~\ref*{#1}}}
\newrefformat{p}{\savehyperref{#1}{Proposition~\ref*{#1}}}
\newrefformat{prop}{\savehyperref{#1}{Proposition~\ref*{#1}}}
\newrefformat{prob}{\savehyperref{#1}{Problem~\ref*{#1}}}
\newrefformat{claim}{\savehyperref{#1}{Claim~\ref*{#1}}}
\newrefformat{clm}{\savehyperref{#1}{Claim~\ref*{#1}}}
\newrefformat{relax}{\savehyperref{#1}{Relaxation~\ref*{#1}}}
\newrefformat{rem}{\savehyperref{#1}{Remark~\ref*{#1}}}
\newrefformat{red}{\savehyperref{#1}{Reduction~\ref*{#1}}}
\newrefformat{part}{\savehyperref{#1}{Part~\ref*{#1}}}
\newrefformat{ex}{\savehyperref{#1}{Exercise~\ref*{#1}}}
\newrefformat{property}{\savehyperref{#1}{Property~\ref*{#1}}}
\newrefformat{type}{\savehyperref{#1}{Type~\ref*{#1}}}
\newrefformat{eg}{\savehyperref{#1}{Example~\ref*{#1}}}
\newrefformat{obs}{\savehyperref{#1}{Observation~\ref*{#1}}}
\newrefformat{que}{\savehyperref{#1}{Question~\ref*{#1}}}
\newrefformat{cond}{\savehyperref{#1}{Condition~\ref*{#1}}}
\newrefformat{ass}{\savehyperref{#1}{Assumption~\ref*{#1}}}
\newrefformat{not}{\savehyperref{#1}{Notation~\ref*{#1}}}
\newrefformat{cond}{\savehyperref{#1}{Condition~\ref*{#1}}}
\newcommand{\Sref}[1]{\hyperref[#1]{\S\ref*{#1}}}

\let\pref=\prettyref
\let\Cref=\prettyref

\renewcommand{\eps}{\varepsilon}

\newcommand{\calA}{\mathcal A}
\newcommand{\calB}{\mathcal B}

\newcommand{\calD}{\mathcal D}
\newcommand{\calE}{\mathcal E}
\newcommand{\calF}{\mathcal F}
\newcommand{\calG}{\mathcal G}

\newcommand{\calL}{\mathcal L}

\renewcommand{\calN}{\mathcal N}

\newcommand{\calS}{\mathcal S}

\newcommand{\calW}{\mathcal W}

\newcommand{\one}[0]{\mathds{1}}
\renewcommand{\set}[2]{\left\{{#1}:{#2}\right\}}

\renewcommand{\R}{\mathbb R}
\newcommand{\C}{\mathbb C}

\renewcommand{\Z}{\mathbb Z}

\newcommand{\ved}[0]{\ve{\cdot}}
\newcommand{\fS}[0]{\mathfrak S}

\newcommand{\drift}{\mathrm{drift}}

\newcommand{\iy}{\infty}

\newcommand{\cT}[0]{c_{\TAP}}

\newcommand{\tref}[1]{\textup{\ref{#1}}}

\newcommand{\lip}[0]{^{\mathrm{Lip}}}
\newcommand{\err}[0]{^{\mathrm{err}}}
\newcommand{\stp}[0]{^{\mathrm{stop}}}

\newcommand{\bP}[0]{\mathbf{P}}
\newcommand{\bQ}[0]{\mathbf{Q}}
\newcommand{\bW}[0]{\mathbf{W}}

\newcommand{\esm}{\varepsilon_{\textup{sample}}}
\newcommand{\ewt}{\varepsilon_{\textup{weight}}}
\newcommand{\ert}{\varepsilon_{\textup{ratio}}}
\newcommand{\Safe}[0]{\mathcal{S}_{\beta, A}}

\usepackage{needspace}
\newlength{\ppartneed}
\newcommand{\ppart}[1]{%
  \par\addvspace{\smallskipamount}%
  \noindent\textit{#1.}\hspace{0.5em}\ignorespaces%
}

\usepackage{algorithm}
\usepackage{algpseudocode}
\algnewcommand\algorithmicinput{\textbf{Input: }}
\algnewcommand\INPUT{\State\algorithmicinput}
\algnewcommand\algorithmicinitialize{\textbf{Initialize: }}
\algnewcommand\INIT{\State\algorithmicinitialize}
\algnewcommand\algorithmicrun{\textbf{Run: }}
\algnewcommand\RUN{\State\algorithmicrun}
\algnewcommand\algorithmicupdate{\textbf{Update: }}
\algnewcommand\UPDATE{\State\algorithmicupdate}
\algnewcommand\algorithmicset{\textbf{Set: }}
\algnewcommand\SET{\State\algorithmicset}
\algnewcommand\algorithmicquery{\textbf{Query: }}
\algnewcommand\QUERY{\State\algorithmicquery}
\algnewcommand\algorithmicoutput{\textbf{Output: }}
\algnewcommand\OUTPUT{\State\algorithmicoutput}

\makeatletter
\newcommand\appendix@section[1]{%
  \refstepcounter{section}%
  \orig@section*{\@Alph\c@section.\texorpdfstring{\,\,\,\;}{}#1}
}
\let\orig@section\section
\g@addto@macro\appendix{\let\section\appendix@section}
\makeatother

\renewcommand{\paragraph}[1]{\medskip\noindent{\bf #1{.}}}

\makeatletter
\newcommand{\saveequation}[2]{%
  #2 \label{#1}
  \protected@write\@mainaux{}{\string\SAVEEQUATION{#1}{\unexpanded{\unexpanded{#2}}}}%
}
\newcommand{\savetagequation}[3]{%
  #3 \label{#1} \tag{#2}
  \protected@write\@mainaux{}{\string\SAVEEQUATION{#1}{\unexpanded{\unexpanded{#3}}}}%
}
\newcommand{\SAVEEQUATION}[2]{%
  \global\@namedef{SAVEDEQUATION@#1}{#2}%
}
\newcommand{\repeatequation}[1]{%
  \ifcsname SAVEDEQUATION@#1\endcsname
    \@nameuse{SAVEDEQUATION@#1}\tag{\ref{#1}}%
  \else
    ?? \notag
  \fi
}
\makeatother

\setlist[itemize]{topsep=-4pt, partopsep=2pt}

\usepackage{thmtools}
\declaretheoremstyle[%
  spaceabove=-2pt,%
  spacebelow=6pt,%
  headfont=\normalfont\itshape,%
  postheadspace=1em,%
  qed=\qedsymbol%
]{mystyle} 
\declaretheorem[name={Proof},style=mystyle,unnumbered,
]{prf}

\togglefalse{focs}

\begin{document}

\author{Holden Lee}
\address{Department of Applied Mathematics and Statistics, Johns Hopkins University, USA}
\email{\href{mailto:hlee283@jhu.edu}{hlee283@jhu.edu}}

\author{Juspreet Singh Sandhu}
\address{Department of Computer Science, Colorado State University, USA}
\email{\href{mailto:jsinghsa@ucsc.edu}{js.sandhu@colostate.edu}}

\author{Jonathan Shi}
\address{No affiliation}
\email{\href{mailto:jshi@cs.cornell.edu}{jshi@cs.cornell.edu}}

\title[Potential Hessian Ascent IV: Sampling the Sherrington--Kirkpatrick model at $\beta < 1$]{Potential Hessian Ascent IV: \\ Sampling the Sherrington--Kirkpatrick model at $\beta < 1$}

\begin{abstract}
\small
\noindent 
We give a polynomial-time algorithm to sample from the Gibbs measure of the Sherrington--Kirkpatrick (SK) model with $o_n(1)$ error in total-variation distance (TVD) at any inverse-temperature $\beta < 1$. The algorithm combines algorithmic stochastic localization (ASL) with rejection sampling over path-space via Jarzynski's equality (JE). \\

\noindent The analysis extends the authors' prior $\beta < 1/2$ result \cite{davies2026potential} by replacing all global regularity requirements in the stochastic differential equation error analysis with local regularity around likely trajectories. The relaxed regularity is established using Celentano's proof of the local strong convexity of the TAP free energy~\cite{celentano2024sudakov}. The analysis utilizes the cavity interpolation theory and free probability toolkit developed in the authors' previous result, where the former applies nearly verbatim and the latter applies supplemented with Lipschitz and $C^2$ extensions of various functions. \\

\noindent The ASL and JE analysis arises from using the TAP free energy as an efficiently computable proxy for the actual free energy of the stochastically localized Gibbs measure \cite[\S 3]{davies2026potential}.
We give a list of \vocab{desiderata} encapsulating the approximation and regularity properties required of the TAP free energy, relaxing those of \cite[\S 2.5]{davies2026potential} to only require local regularity.
These generic desiderata are potentially applicable in other settings where a free energy surrogate exists, giving algorithmic sampling guarantees while bypassing the usual functional inequality--based approach.

\end{abstract}

\maketitle

\thispagestyle{empty}
\vspace{-5mm}
\renewcommand{\baselinestretch}{0.9}\normalsize
{
  \hypersetup{linkcolor=Red}
  \setcounter{tocdepth}{1}
  \tableofcontents
}
\renewcommand{\baselinestretch}{1.0}\normalsize

\newpage 
\pagenumbering{arabic}

\section{Introduction}

The Gibbs measure for the Sherrington--Kirkpatrick (SK) model of system-size $n$ is a (random) probability measure on $\{-1,1\}^n$ given by
\[
    \mu_{\beta A}(\sigma) = \frac{1}{Z(\beta A)}\exp\left(\frac{1}{2}\an{\sigma,\beta A\sigma}\right)\,,
\]
where $A \sim \mathsf{GOE}(n)$ and the constant $Z(\beta A) := \sum_{\sigma \in \{-1,1\}^n} \exp\left(\frac{1}{2}\an{\sigma,\beta A\sigma}\right)$ is the partition function. 

Building on the cavity interpolation theory and free-probability toolkit developed by the authors in their previous result \cite{davies2026potential}, we give a polynomial time algorithm to sample from the Gibbs measure of the SK model at every fixed inverse temperature $\beta < 1$ with $o_n(1)$ total-variation distance (TVD) error. In the absence of an external field, this characterizes the entire high-temperature regime \cite{deAlmeidaThouless1978,Lopatto2026ReplicaSymmetry} and $\beta = 1$ is believed to be the onset of computational hardness for sampling \cite{el2022sampling, sellke2025exponentially} since the system exhibits full replica-symmetry breaking (fRSB) at $\beta > 1$ \cite{jekel2024pha,lopatto2026full}. This is another step towards fully resolving the conjecture that Glauber dynamics exhibits fast-mixing for the the SK model below the dAT line \cite[Open-Problem 15]{bandeira2025randomstrasse101}. In forthcoming work \cite{lee2026weak}, we make further progress towards the conjecture by generalizing the conductance arguments and Wiener analysis techniques developed in another prior result of the authors \cite{davies2026weak} -- we prove functional inequalities with access to only ``local'' regularity properties of the algorithmic stochastic localization (ASL) process.\footnote{\,A consequence of the generalization of \cite[\S 3, \S 4 \& \S 6]{davies2026weak} in the forthcoming work \cite{lee2026weak} is a weak Poincar\'e inequality for the SK model up to $\beta < 1$, followed by a proof that a warm start can be algorithmically obtained.} %

\begin{theorem}%
\label{t:main}
    Let $0 < \beta < 1$. Given any fixed $\delta > 0$, for all $n \ge n_0(\beta,\delta)$, with probability at least $1-\delta$ over $A\sim\mathsf{GOE}(n)$, \pref{alg:main} gives a sample from a distribution that is $o_n(1)$ close in TVD error to $\mu_{\beta A}$ in polynomial time.
\end{theorem}

\subsection{Technical overview} The prior result of the authors gives a sampling algorithm for $\beta < 1/2$ through a set of ``desiderata'' \cite[\S 2]{davies2026potential} which include pointwise regularity of certain matricial functions of operators that drive the algorithmic process. There are known counter-examples to such pointwise regularity for $0.798 < \beta < 1$ --- these are a consequence of the fact that the TAP free energy for the SK model \cite{fan2021tap} is known not to be \emph{globally} strongly convex in $(-1,1)^n$ when $0.798 < \beta < 1$ \cite[Theorem 1.2]{gufler2023concavity}. However, a result of Celentano \cite{celentano2024sudakov} uses control of the empirical distribution of approximate-message passing (AMP) iterates \cite[Proposition 3.8]{celentano2024sudakov} in conjunction with the Sudakov--Fernique inequality to demonstrate local strong convexity for the TAP free energy at $\beta < 1$. This hints that a ``relaxation'' of the desiderata in \cite[\S 2]{davies2026potential} to require regularity only over a ``safe'' set of points $S_A(c) \subset (-1,1)^n$, those where the TAP free energy continues to be locally strongly convex, likely allows the analysis template developed in \cite[\S 3--6]{davies2026potential} to cover the entire high-temperature regime. 

In this work, we accomplish this goal through two important contributions:
\begin{enumerate}[itemsep=0.2em]
    \item We develop the precise formulation for local regularity of various key matricial functionals of the ideal and algorithmic processes, updating the desiderata to reflect this (\pref{subsec:desiderata}).
    \item We systematically extend the SDE error analysis (\pref{sec:sde-analysis}), control over the diagonal of the algorithmic covariance (\pref{sec:resolvent-analysis}), and proofs for the regularity properties (\pref{sec:alg-properties}) developed in \cite[\S 3, \S 6 \& \S B]{davies2026potential} to hold over non-convex sets under \emph{local} regularity, and demonstrate uniform local strong convexity of the TAP free energy along the stochastic localization (SL) process (\pref{sec:uniform-tap-convexity}), thereby proving the requisite ``relaxed'' desiderata.    
\end{enumerate}

\subsubsection*{Algorithm}
The only change to the algorithm of \cite{davies2026potential} is to reject any trajectories which leave the safe set, and using a modified sampler known to work for any value of $\be$, provided by \cite{kumar2026high}.

\begin{algorithm}[!ht]
\caption{Informal description of our algorithm}
\label{alg:informal}
\begin{algorithmic}[1]
\Repeat{}
\State Initialize $y_0=0$.
\State $\td y_T \leftarrow \mathsf{approximate}(dy_t = m_tdt + dB_t)$ for time $T=O(1)$, while computing Jarzynski weights $\td w_T$, restarting if trajectory exits safe set.
\State Accept $\td y_T$ with some probability computed from $\td w_T$.
\Until{accept} \Comment{Now $\td y_T$ is approximately from the time-$T$ distribution of $dy_t = m_tdt + dB_t$.}
\State Initialize $\si_0=\sign(\td y_T)$.
\State Run the sampler for the localized distribution (\pref{thm:localized-sampler}).
\end{algorithmic}
\end{algorithm}

\subsubsection*{Analysis} The cavity interpolation theory built
 to bound the error in approximating the Gibbs covariance by the inverse of the TAP Hessian already works up to $\beta < 1$ (see \pref{sec:cavity-interpolation}) and the free probability analysis template remains roughly the same\footnote{\,The inverse of the TAP Hessian (a resolvent) needs to be regularized so that it extends smoothly outside the safe set. This leads to certain analytic technicalities, but does not alter the overall structure of the argument in \cite[\S B]{davies2026potential}. The authors do, however, give a simpler version of the free interpolation argument.} (see \pref{sec:resolvent-analysis}).

The first step in the analysis comes with a proof that, with high probability over the input and the SL process, the TAP Hessian continues to be uniformly locally strongly convex at any time $t \in [0,T(\beta)]$ (\pref{sec:uniform-tap-convexity}). This proof combines Celentano's argument for local strong convexity of the TAP Hessian \cite[Theorem 2]{celentano2024sudakov} in small balls around late-stage AMP iterates with the characterization of ``stability'' of the SL process for the SK model established in \cite[Lemma 4.9]{el2022sampling} (\pref{sec:uniform-strong-convexity-proof}). From this, it is a straightforward corollary to obtain a L\"owner order ``sandwich'' for the TAP Hessian on the safe set (\pref{sec:lowner-sandwich-safe-set}).

The second step in the analysis deals with regularity properties of the inverse of the TAP Hessian, which we interpret as a resolvent.
In \cite[\S B]{davies2026potential}, we obtained precise control over the diagonal elements of the squared resolvent, whose importance arises from the fact that those diagonal elements give a measure of how close the TAP Hessian is to being a covariance on the discrete hypercube (the diagonal elements of the covariance matrix are the variances of individual coordinates, which for binary variables are fully determined by their means).
Using that diagonal control, \cite[\S 6]{davies2026potential} established regularity and SL-like properties for the coefficients of the algorithmic SDE.

In the present work, we first introduce a $C^2$ extension of the resolvent from the safe set $S_A(c)$ onto the entirety of $(-1,1)^n$ (\pref{sec:block-regularization}).
This $C^2$ extension is swapped into both parts of the diagonal control argument where regularity is needed to bound fluctuations: the chaining argument obtaining uniform concentration of the resolvent diagonal (\pref{sec:block-mixed-lipschitzness}), as well as a Gaussian interpolation argument comparing the expected resolvent to its freely independent limit (\pref{sec:option-c-free-comparison}).

The regularity and SL-likeness of the coefficients then follow (\pref{sec:alg-events} -- \pref{sec:alg-phd-asl-tap-closeness}) almost exactly as in \cite[\S 6.1--\S 6.5]{davies2026potential}, only with Lipschitz bounds proved directly through the resolvent identity instead of through derivative bounds on straight lines that now are no longer guaranteed to exist within the safe set.
For SDE well-posedness and transportation inequalities, we swap in a simple Lipschitz extension of the coefficients, and analyze the modifed SDE which is guaranteed to agree with the original up until exit from the safe set (\pref{sec:well-posed-transportation}).
Finally, the Jarzynski weights are the only Lipschitz bound that can't be obtained only with the resolvent identity: for these, we design short paths within the safe set and integrate bounded derivatives along these short paths to obtain the Lipschitz bound (\pref{sec:alg-je-extension}).

The third step in the analysis is straightforward given existing work -- a simple choice of $T \ge T(\beta)$ gives a localized distribution where the external field is strong enough to ensure concentration on a small Hamming ball with high probability over the field, and fast mixing on any Hamming ball of that size \cite{kumar2026high}.

The final step in the analysis is to bound the error between the SDE's of the ideal and algorithmic process, when the latter is designed to stop the moment the process leaves the safe set. This requires redoing the approach taken in \cite[\S 3]{davies2026potential} with delicate error handling for the stopped process to obtain $O(1)$ KL divergence bounds, followed by the a ``survival''-based Jarzynski's equality (JE) analysis that takes into account the stopped process for the analysis of the rejection sampling step.

\subsection{Related work}
\label{sec:related-work}
There are two main approaches to sampling from spin glass models. One is to prove functional inequalities (e.g.\ Poincar\'e or modified log-Sobolev inequalities), typically via stochastic localization, and deduce rapid mixing of Glauber or Langevin dynamics. It can be sufficient to show weak functional inequalities, which essentially establish the efficacy of these dynamics given a warm start. The other is to algorithmize stochastic localization itself. Our method combines these perspectives: we use ASL and Jarzynski reweighing up to a large finite time, and then use a fast-mixing result for the localized distribution. Since our previous result \cite{davies2026potential}, a few notable results have made progress on understanding sampling from the SK model in various settings -- we review only these and a few auxiliary results, and refer the reader to \cite[\S 1.4]{davies2026potential} for a more complete overview.

\subsubsection*{Functional inequalities and fast mixing} Recently, the authors proved a weak Poincar\'e inequality for the SK model when $\beta<1/2$ by transferring regularity of approximate stochastic localization to the target measure, yielding fast-mixing for Glauber dynamics from a warm start \cite{davies2026weak}. Shortly thereafter, Wang \cite{wang2026optimal} proved fast-mixing of single-site Glauber dynamics for the SK model at $\beta<1/2$, from every initial configuration and uniformly over all external fields. Even more recently, Boban, Li, and Oveis Gharan \cite{boban2026rank} introduced rank-$1$-perturbed trickledown theorems and obtained polynomial mixing of Glauber dynamics for the SK model at $\beta\le 1/2+\varepsilon$ for some explicit absolute constant $0<\varepsilon<10^{-4}$. For $\beta\le1/2+\varepsilon$, these results imply fast-mixing for Glauber dynamics.

Bandeira, El Alaoui, and R\"odder \cite{bandeira2026mixing} prove polynomial mixing of Glauber dynamics for high-overlap Gibbs measures, including the planted SK model with sufficiently strong Gaussian external field, which we use to sample from a sufficiently localized distribution. Kumar, Sarkar, Tian, and Zhu \cite{kumar2026high} give polynomial-time samplers for the SK model on sufficiently highly-magnetized slices at every fixed inverse temperature, and combine these with annealing to obtain samplers under strong external fields.

The regime $\beta<1$ is sharp for our approach and believed to be the onset of hardness for efficient sampling. Specifically, it is know that no stable algorithm (a class containing ASL) can sample with small $W_2$-distance error for
$\beta>1$ \cite{el2022sampling}, and worst-case Glauber dynamics mixes in exponential time at sufficiently low temperature \cite{sellke2025exponentially}.

\begin{table}[t]
\centering
\small
\setlength{\tabcolsep}{4pt}
\renewcommand{\arraystretch}{1.15}
\begin{tabular}{@{}p{0.27\textwidth}p{0.27\textwidth}p{0.13\textwidth}p{0.22\textwidth}@{}}
\toprule
Work & Method & Guarantee & Regime ($\beta<\cdots$) \\
\midrule
\multicolumn{4}{@{}l}{\emph{Functional inequalities / fast mixing}} \\[2pt]
Eldan, Koehler, and Zeitouni \cite{eldan2022spectral}; Anari et al.\ \cite{anari2022entropic}
& PI / MLSI via SL & $\varepsilon$-TVD & $1/4$ \\[3pt]
Anari, Koehler, and Vuong \cite{anari2024trickle}
& MLSI via SL & $\varepsilon$-TVD & $\approx0.295$ \\[3pt]
Davies et al.\ \cite{davies2026weak}
& weak PI; warm-start Glauber & $o(1)$-TVD & $1/2$ \\[3pt]
Wang \cite{wang2026optimal}
& PI / MLSI; Glauber & $\varepsilon$-TVD & $1/2$ (any ext. field) \\[3pt]
Boban, Li, and Oveis Gharan \cite{boban2026rank}
& rank-$1$ trickledown; Glauber & $\varepsilon$-TVD & $\le 1/2+\varepsilon$ \\[4pt]
\multicolumn{4}{@{}l}{\emph{Algorithmic stochastic localization}} \\[2pt]
El Alaoui, Montanari, and Sellke \cite{el2022sampling}; Celentano \cite{celentano2024sudakov}
& ASL + AMP & $o(n)$-$W_2^2$ & $1/2$, then $1$ \\[3pt]
Davies et al.\ \cite{davies2026potential}
& ASL-TAP / PHD + JE & $o(1)$-TVD & $1/2$ \\[3pt]
This work
& ASL-TAP / PHD + JE & $o(1)$-TVD & $1$ \\
\bottomrule
\end{tabular}
\vspace{1mm}
\caption{Sampling guarantees for the SK model. ``$\varepsilon$-TVD'' denotes error
$\varepsilon$ achievable for every fixed $\varepsilon>0$ with running time
polynomial in $1/\varepsilon$; ``$o(1)$-TVD'' denotes error vanishing as
$n\to\infty$.}
\label{tab:related-sampling}
\end{table}

\subsubsection*{Interpolations, overlaps and covariances}
Various mathematical results and techniques provide the foundation for analyzing the Gibbs overlap array and covariance for ``planted'' spin-glass models (cf. \cite[\S 1.4]{davies2026potential}). Most relevant to this paper, the prior result of the authors \cite[\S 4--5]{davies2026potential} develops the precise cavity moment estimates for the planted SK model with SL field needed for the covariance approximation. While stated in \cite[\S 5]{davies2026potential} as a result that is bottlenecked at $\beta < 1/2$, the result extends easily to $\beta < 1$ (see \pref{sec:cavity-interpolation}). Inspired by the ability to Guerra interpolate the planted SK model with SL tilt, Li and the last two authors prove sub-Gaussian overlap concentration for the Curie--Weiss random-field model at $\beta<1$ with an external field satisfying the Nishimori fixed-point condition \cite{li2026overlap}.

\subsubsection*{PHA, free probability and local TAP convexity}
Subag pioneered Hessian Ascent---following the top eigenvector of the Hessian of the objective function---to optimize spherical spin glass models \cite{subag2021following}, which inspired both Montanari's AMP-based optimization algorithm \cite{montanari2021optimization} and the later Potential Hessian Ascent (PHA) framework for the SK model \cite{jekel2024pha,jekel2025pha2}. For spin glasses on the hypercube, PHA adds a coordinate-wise potential to the objective function, recovering the generalized TAP free energy \cite{chen2023generalized}. Resolvent bounds and techniques from free-probability control the diagonal entries of matrix functions of the Hessian, which determine how individual coordinates evolve and give access to regularity properties. In the prior work of the authors, these techniques were used to provide global regularity properties \cite[\S 6]{davies2026potential}. The present work establishes diagonal control in the region of local TAP convexity even when global control is impossible.
A simplification of the previous argument is enabled by comparisons of matricial functions of Gaussian random matrices with their free limits developed by Bandeira, Boedihardjo, and van Handel \cite{bandeira2023matrix}.

\section{Sampling via ASL-JE under local regularity}

Our sampling algorithm closely follows the implementation of \vocab{algorithmic stochastic localization} described in \cite[\S 2]{davies2026potential}, so we defer to that section for the full explanation, and will describe only the differences here.
In multiple places, the original approach used \emph{global strong convexity} of the TAP free energy and the consequent regularity of various random processes defined over the domain $(-1,1)^n$ of the magnetization.
Whereas this global strong convexity has a particularly simple proof at $\beta < 1/2$, it is known to fail for $\beta > 0.798$ ~\cite{gufler2023concavity}.
Thus, in the current regime of interest $1/2 \le \beta < 1$, many of the TAP-based quantities used in the argument of \cite{davies2026potential} actually fail to exist.

We circumvent this issue by giving an analysis that only requires \emph{local convexity} of the TAP free energy around likely SL paths, which can be proved up to $\beta < 1$ by building upon an argument of Celentano~\cite{celentano2024sudakov}.
We define a \vocab{safe subset} of $(-1,1)^n$ characterized by this local convexity property, and relax our SDE and JE analyses to require regularity only up to (the unlikely event of an) escape from this safe set.
Different parts of the argument require different treatments of the complement of the safe set, including Lipschitz or $C^2$ extensions of functions to the complement of the safe set, the construction of short paths that stay inside the safe set, and the use of the resolvent identity to prove Lipschitz bounds without derivative bounds on straight-line paths.

\subsection{TAP free energy and the safe set} First, we recall the definition of the TAP free energy, then use that to define our notion of safe sets.
\begin{definition}[TAP free energy for the planted SK model, {\cite{thouless1977solution,fan2021tap}}]\label{def:f-tap}
    Denote the TAP free energy for the Hamiltonian $\an{m,Am}$ at inverse temperature $\beta < 1$ as
    \begin{align}
    \label{e:FTAP}
    \calF_{\TAP}(m,y) &:= -\fc{\be}{2}\iprod{m,A m} - \iprod{y,m} - \sumo in h(m_i) - \rc 4 n\be^2 \pa{1-\rc n \norm{m}_2^2}^2,
    \end{align}
where $h(x) := -\fc{1+x}{2}\log \pf{1+x}2 - \fc{1-x}{2}\log \pf{1-x}2$ and $\hat{Q}(m) := \nabla_m^2\calF_{\TAP}(m)^{-1}$ where $m \in (-1,1)^n$.
\end{definition} 
Note that $\nabla_m^2\calF_{\TAP}(m,y)$ does not depend on $y$, so we will write this simply as $\nabla_m^2\calF_{\TAP}(m)$. Note
\begin{align}
\label{e:Hess-FTAP}
    \nabla_m^2\calF_{\TAP} = -\beta A + \diag\prc{1-m^2} + \beta^2\left(1-\fc{\norm{m}_2^2}{n}\right)\Id_n - \fc{2\beta^2}{n} mm^{\sT} =: \hQ(m)^{-1},
\end{align}
where the definition is when $\nabla_m^2\calF_{\TAP}(m)\succ 0$. 
We define the set of safe points in the solid cube as those near which $\calF_{\TAP}(m,\cdot)$ is sufficiently strongly convex in $m$.
\begin{definition}[Safe set in the hypercube]\label{def:good-points}
    Fix $A \in \R^{n\times n}$ and $D(m) := \diag\left(\frac{1}{1-m^2}\right)$ for any $m \in (-1,1)^n$. Let
        \[ \calS_{\be,A}(c) := \left\{m \in (-1,1)^n \mid c\,D(m) \prec \nabla^2\calF_{\TAP}(m)\right\}\,. \]
\end{definition}
Given $m\in (-1,1)^n$, there is a unique $y\in \R^n$ such that $(m,y)$ solve the TAP equation $\gd_m\FT(m,y)=0$. Define $\hy:(-1,1)^n\to \R^n$ to be this value of $y$:
\begin{align}\label{e:hy}
\hy(m):= -\fc{\be}2 \an{m,Am} - \sumo in h(m_i) - \rc 4 n \be^2 \pa{1-\rc n \ve{m}_2^2}^2.
\end{align}
Note that $\hy$ may not be injective even on $\Safe(c)$, but it is locally; see \pref{l:inv-f}.

\subsection{Desiderata}\label{subsec:desiderata} We now provide the updated the list of desiderata, modifying those provided in \cite[\S 2]{davies2026potential} to work within the safe set and to present a more minimal interface.
These desiderata formalize the requirement of the TAP free energy to be a good approximation of the actual free energy while having good regularity properties, to prove efficient sampling with low error measured in total variation distance.

The following desiderata are written over SL paths $\{m_t\}_{t \ge 0}$ which come as a solution to HD. 
\begin{enumerate}[itemsep=0.7em,label={(d\arabic*)},start=1] 
    \item\label{d:Q-error} \textbf{($O(1)$-error for covariance estimate, \pref{thm:covar-estimate})} The following covariance estimate holds for $0 < \beta < 1$:
    \[
        \E_{A,B_t}\left[\norm{\nabla_m^2\FT(m_t)
        Q_t - \Id_n}^2_F\right] \le \eps_{\mathrm{cov}}(\beta, t)^2\,,
    \]
    where $\sup_{t\in[0,T]}\eps_{\mathrm{cov}}(\beta, t) < \infty$.
    \item[(d2a)] \makeatletter\def\@currentlabel{(d2a)}\makeatother \label{d:safe-set-stay-inside} \textbf{(Uniform local strong convexity of TAP Hessian over SL process, \pref{thm:uniform-local-strong-convexity})}
    For every $T$, there are constants $\rho:=\rho(\beta, T)>0$, $ \cT :=c_{\TAP}(\beta,T) > 0$ such that
    \[ \Pr_{A,B_t}\left\{\forall t \in [0,T]\,,\; B_{\square}(m_t, \rho\sqrt{n}) \subseteq \Safe(c_{\TAP})\right\} \ge 1 - o_n(1), \]
    where $B_{\square}(x,r) := \{u \in (-1,1)^n: \lpnorm{u-x} < r\}$ is the $x$-centered ball of radius $r$ in the $\ell_2$ distance within $(-1,1)^n$.
\end{enumerate}

The following desiderata are written over all $m \in \Safe(c_{\TAP})$, on a high-probability event on $A$: 
\begin{enumerate}[itemsep=0.7em,label={(d\arabic*)},start=3]
    \item[(d2b)]\makeatletter\def\@currentlabel{(d2b)}\makeatother \label{d:Q-reg} \textbf{(Local regularity of the TAP Hessian, \pref{lem:tap-convexity})} There exist finite $c_{\TAP}(\beta,T), C_{\TAP}(\beta) > 0$, such that
    \[
        C_{\TAP}(\beta)^{-1}D^{-1}(m) \preceq \hat{Q}(m) \preceq c_{\TAP}(\beta,T)^{-1}D^{-1}(m). 
    \]
    \item \label{d:drift-error}
    \textbf{(Drift error in PHD-TAP, \pref{cor:alg-desiderata})} 
    \[
    \ve{f(m)-m}_2^2 \le \ep_{\mathrm{drift}}(\be)^2,\]
    where $\ep_{\mathrm{drift}}(\be)$ is a constant depending only on $\beta$ and 
    \begin{align}\saveequation{e:Itomag}{f(\mg) &= \left(E_{\calD_n}\left[\hQ(\mg)^2\right]D(\mg)^2 - \frac{\beta^2}{n}\left(\Tr\left[\hat{Q}^2(\mg)\right]\Id_n +2\hat{Q}^2(\mg)\right)\right)\mg}
\end{align}
    \item \label{d:Q-Lip}
    \textbf{(Lipschitzness of diffusion and drift terms, \pref{cor:alg-desiderata})}
    \begin{enumerate}
        \item
        There exists a constant $L$  such that for all $\mg, w\in \Safe(c_{\TAP})$, 
    \[
    \ve{\hQ(\mg) - \hQ(w)}_F \le L\ve{\mg-w}_2.
    \]
    
        \item There exists a constant $L_{\mathrm{drift}}$  such that for all $\mg, w\in \Safe(c_{\TAP})$, with $f$ defined as in \cite[(2.3)]{davies2026potential},
    \[
    \ve{\hat{Q}(m)(m-f(m)) - \hat{Q}(w)(w-f(w))}_2 \le L_{\mathrm{drift}}\ve{m - w}_2.
    \]
    \end{enumerate}
        \item\label{d:JE-lip}
        \textbf{(Lipschitzness of JE weight integrand, \pref{cor:alg-desiderata})} Let
        \[\omega(m):=\frac12\bigl(\Tr[\hat{Q}(m)]+\lpnorm{m}^2\bigr).\]
        Then there is an $L_\omega<\infty$ such that
\[
       |\omega(m)-\omega(w)|\le L_\omega\lpnorm{m-w}
                       \qquad(\forall m,w\in\Safe(c)).
\]
\end{enumerate}
Note that as a consequence of these desiderata, we will also be able to prove \ref{d:safe-set-stay-inside} for the algorithmic $\hat{m}_t$. This is done via a SDE-based argument in \pref{l:safe-pair}.

Finally, we need an efficient sampling algorithm for the localized distribution. For large enough time, the measure concentrates on a small Hamming ball, so the following statement suffices.
\begin{enumerate}
    \item[(d6)] \makeatletter\def\@currentlabel{(d6)}\makeatother 
    \label{d:sample-localized}
    \textbf{(Efficient sampling on small enough wedges, \pref{thm:localized-sampler})} 
    There exists small enough $\ep(\be)>0$ such that with high probability over $A$, there is a $\poly(n,\log(1/\ep))$-time algorithm to sample from $\mu_{\be' A,y}|_{B(\sign(y),\ep(\be)n)}$ with TV distance $\ep$, for all $y\in \R^n$ and all $0\le \be'\le \be$, where $B(\si,k)= \set{\si'\in \{\pm 1\}^n}{d(\si,\si')\le k}$ is the Hamming ball around $\si$ with radius $k$.
\end{enumerate}

One of the high-probability events we will need is simply an operator norm bound on $A$.

\begin{proposition}[GOE operator norm concentration, {\cite[Theorem II.1]{davidson2001local}}]
\label{prop:goe-norm-tail}
Let $A \sim \mathsf{GOE}(n)$. Then, 
\begin{equation}
\label{eq:goe-norm-tail}
    \Pr\{\opnorm{A}>3\}\le2e^{-n/4}.
\end{equation}
\end{proposition}
\begin{proof}
This is given by the Gaussian eigenvalue tail bound of \cite[Theorem II.11]{davidson2001local} applied to $A$ and $-A$.
\end{proof}

\subsection{Algorithm} We display the full algorithm in \pref{alg:main}.  
First, we run ASL-TAP-JE up to stopping time $\tau_c$ as derived in \cite[\S 2.2--2.3]{davies2026potential}; we separate this out into \pref{alg:asl-ta-je-dre}. As in \cite[Algorithm 3]{davies2026potential} we conduct rejection sampling as described in \cite[\S 2.3]{davies2026potential} to turn our our weighted sample for $\rh_T$ into an approximate sample for $\mu_T$. Finally, we sample from the localized distribution $\mu_{y_T}$ %
using our approximate sample as a warm-start.

\begin{algorithm}[!ht]
\caption{SK Sampler}
\label{alg:main}
\begin{algorithmic}[1]
\INPUT matrix $A$, inverse temperature $\be$, total time $T$ which is a multiple of step size $\eta$. %
\OUTPUT Approximate sample from Gibbs measure $\mu_A$.
\State \textbf{Part 1:} Approximately sample from $\mu_T$. 
\State Run \pref{a:ars} (Approximate Rejection Sampler) %
with subroutine for sampling and density ratio estimation given by \pref{a:asl-ta-je-dre} (ASL-TAP-JE with density ratio estimation) to obtain $\td y_T$.
\State \textbf{Part 2:} Approximately sample from localized distribution $\mu_{\beta A,\td y_T}$.
\State Run the sampler for $\mu_{\be A, \si_0}$ restricted to the Hamming ball $B(\si_0, c(\be)n)$, given by \pref{thm:localized-sampler} (\cite{kumar2026high}).
\end{algorithmic}
\end{algorithm}

\begin{algorithm}[!ht]
\caption{(ASL-TAP-JE) with density ratio estimation}
\begin{algorithmic}[1]
\INPUT matrix $A$, inverse temperature $\be$, total time $T$, step size $\eta$, number of steps $S$ for mirror descent, error parameter $\ep$
\State Let $\td y_0=0$, $\td \mg_0=0$, $\td x_0=0$, and $\tw_0=0$. 
\For{$t\in \{0,\eta,\ldots, T-\eta\}$}
    \State Draw $\xi_t \sim \calN(0,I_n)$. 
    \State $\td y_{t+\eta} = \td y_t + \eta \td\mg_t + \sqrt{\eta}\xi_t$.
    \State $\td x_{t+\eta}^{(0)} = \td x_t +\eta \ba{D(\td\mg_t)\hQ(\td\mg_t) (\td \mg_t - f(\td \mg_t)) + \ED\pa{\hQ(\td\mg_t)^2} D(\td \mg_t)^2\td \mg_t} + \sqrt{\eta}D(\td \mg_t)\hQ(\td \mg_t) \xi_t$.
    \For{$s\in \{0,\ldots, S-1\}$} \Comment{Solve $\nabla_m\calF_\TAP(\mg,y) = 0$ for $\mg$ when $y=\td y_{t+\eta}$ using mirror descent}
    \label{line:asl-tap-je-mirror-descent}
        \State $\td x_{t+\eta}^{(s+1)} = \td x_{t+\eta}^{(s)} - \lm\gd_m \FT(\tanh(\td x_{t+\eta}^{(s)}),\td y_{t+\eta})$ where $\gd_m \FT$ is given in \eqref{e:dTAP} and $\lm=C_{\TAP}(\beta)^{-1}$ in \ref{d:Q-reg}.
    \EndFor
    \State $\td x_{t+\eta} = \td x_{t+\eta}^{(S)}$.
    \State $\td m_{t+\eta} = \tanh(\td x_{t+\eta})$. 
    \State $\td w_{t+\eta}=\td w_t + \fc{\eta}2 \ba{\Tr(\hQ(\td\mg_t))+ \ve{\td\mg_t}^2}$.
    \State If $\gd^2 \FT(\td m_{t+\eta})\not\succ cD(m)$, then restart.
\EndFor
\State Estimate $\hat Z_{\be A, \td y_T}$ 
using simulated annealing (\cite[Algorithm 5]{davies2026potential}) with sampler given by given by \pref{thm:localized-sampler} (\cite{kumar2026high}).
\OUTPUT $\pa{\td y_T , e^{\tw_T}\hat Z_{\be A, \td y_T}e^{\FT(\td \mg_T,\td y_T)}}$
\end{algorithmic}
\label{a:asl-ta-je-dre}
\label{alg:asl-ta-je-dre}
\end{algorithm}

\section{Quantitative SDE analysis}\label{sec:sde-analysis}

We first recall the various processes which appear in our algorithm and analysis. For a full derivation, see \cite[\S3]{davies2026potential}.
First, the ideal process we would like to approximate is stochastic localization in $y$-space or Hessian dynamics in $m$-space: 
\begin{align}
    \savetagequation{e:SL}{SL}{dy_t &= m_t dt + dB_t,& y_0&=0},\\
    \savetagequation{e:HD}{HD}{dm_t &= Q_tdB_t,& m_0&=\E_{x_0 \sim \mu}[\sigma]=0},\\
    \mg_t &= \an{\si}_{y_t} = \E_{\mu_{y_t}}\si, & Q_t &= \Cov(\mu_{y_t}),\label{e:mg}
\end{align}
where $\mu_{y}(\si) \propto \mu(\si)e^{\an{y,\si}}$ and $\mu=\mu_{\be A}$ is the desired Gibbs measure. 
Two natural algorithmic processes arise from using an estimate of $m_t$ in \eqref{e:SL} or an estimate of $Q_t$ in \eqref{e:HD}, and then enforcing the TAP equation 
\begin{align}
\savetagequation{e:TAP}{TAP}{0=\gd\FT(m,y) &= -\be A m - y + \ub{\rc 2\log \pf{1+\mg}{1-\mg}}{\tanh^{-1}(\mg)} 
    + \be^2 \pa{1-\rc n \ve{m}_2^2}m}
\end{align}
to solve for the other variable ($y$ or $m$). This constraint is given by differentiating \eqref{e:TAP}, which gives
\begin{align}
\label{e:dTAP}
        d\hat y_t &= \hat Q(\hat \mg_t)^{-1} d\hat \mg_t + f(\hat \mg_t) dt.
\end{align}
Define $\hQ(m)$ as in \eqref{e:Hess-FTAP}.
Then the two options give 
\begin{align}
\savetagequation{e:ASL-TAP}{ASL-TAP}{&\begin{cases}\begin{aligned}
        d\hat y_t &= \hat \mg_tdt+dB_t\\
        d\hat \mg_t &= \hQ(\hat \mg_t)\pa{\hat \mg_t - f(\hat \mg_t)} dt + \hQ(\hat \mg_t)dB_t
    \end{aligned}\end{cases}}\\
\savetagequation{e:PHD-TAP}{PHD-TAP}{&\begin{cases}\begin{aligned}
        d\hat \mg_t &= \hQ(\hat \mg_t)dB_t\\
        d\hat y_t &= f(\hat \mg_t)dt + dB_t 
    \end{aligned}\end{cases}}
\end{align}
These are defined if $\hm_t$ stays inside a safe set $\Safe(c)$, and we can extend then outside by modifying the drift and diffusion terms appropriately---see \pref{sec:well-posed} and \pref{sec:transportation}. For \eqref{e:ASL-TAP}, we define Jarzynski weights by 
    \begin{align}
    \savetagequation{e:JE}{JE}{
dw_t &= \om(\hm_t)\,dt, & 
\om(m):&= \rc 2 \ba{\Tr(\hQ(m)) + \ve{m}^2}}.
    \end{align}
Finally, for analysis purposes, we also need to consider \eqref{e:SL} and keep track of the estimated magnetization for tilt $y_t$ according to \eqref{e:TAP}, which we denote as $m^y_t$ (i.e., we consider $y$ under the law of the SL process, but find what the $m$ would be if it were the algorithmic process). We derive this by solving \eqref{e:dTAP} with $d\hy_t$ replaced by $dy_t$ in \eqref{e:SL}. 
\begin{align}
\savetagequation{e:SL-TAP}{SL-TAP}{&\begin{cases}\begin{aligned}
        dy_t &= \mg_t\,dt+dB_t\\
        dm^y_t &= \hQ(m^y_t) (dy_t - f(m^y_t)\,dt)= \hQ(m^y_t) ((m_t- f(m_t^y)) \,dt + dB_t).
\end{aligned}\end{cases}}
\end{align}

\subsection{$O(1)$ Wasserstein and KL bounds}

We first consider a general setup and give a Wasserstein bound for general SDE's with different drift and diffusion terms, when Lipschitzness holds locally. 
Let $\calS\subeq \Om \subeq \R^n$, where $\calS$ and $\Om$ are open and $\calS$ is considered the ``safe set''.
Consider 2 It\^o processes
\begin{align}\label{e:sdes}
\begin{split}
    dX_t &= f_t%
    \,dt + \Si_t%
    \,dB_t\\
    d\hat X_t &= \hat f_t%
    \,dt + \hat \Si_t %
    \,dB_t
    \end{split}
\end{align}
and assume that pathwise unique strong solutions exist for $X_t$, and for $\hat X_t$ up to the first exit time 
\[\tau = \inf\set{t\ge 0}{\hat X_t\nin \calS}.\]
Suppose that $X_0=\hat X_0$. 
We will think of $X_t$ as an ideal process and $\hat X_t$ as an approximate process. 
Define the expanded safe set and safe pair set by 
\allowdisplaybreaks
\begin{align*}
\calS(r) &= \set{x\in \Om}{B(x, r)\cap \Om\subeq \calS},\\
\calS^2(r_1,r_2) &= \set{(x_1,x_2)\in \Om^2}{x_1\in \calS(r_1), \, \ve{x_1-x_2} < r_2},\quad r_1\ge r_2\\
\calS^2(r) &= \calS^2(r,r),
\end{align*}
and define the path-space analogues by 
\allowdisplaybreaks
\begin{align*}
\fS_t &= \set{\ga}{\forall s\in [0,t], \, \ga_s\in \calS}\\
\fS_t (r) &= \set{\ga}{\forall s\in [0,t], \, \ga_s\in \calS(r)}\\
\fS_t^2(r_1,r_2) &= \set{(\ga_1,\ga_2)}{\forall s\in [0,t], \, (\ga_{1,s}, \ga_{1,t})\in \calS^2(r_1,r_2)},\quad r_1\ge r_2\\
\fS_t^2(r) &=\fS_t^2(r,r).
\end{align*}
Note that if $(\ga_1,\ga_2)\in \fS_t^2(r_1,r_2)$, then $\ga_1,\ga_2\in \fS_t$.
Note that if we apply this to the safe set $\calS=\Safe(c)$, then for the expanded safe set of \ref{d:safe-set-stay-inside}, 
\[
\set{(y_{[0,t]}, m_{[0,t]})}{\forall t\in [0,t],\,B_{\square}(m_t, r)\subeq \Safe(c)}
\subeq \calS(r),
\]
so the $1-o_n(1)$ probability transfers to $\calS(r)$.
\begin{lemma}[Gr\"onwall bound for Wasserstein error of SDE's, under local Lipschitzness]\label{l:W-gw}
Consider the setup of \eqref{e:sdes} where $X_t$ and $\hat X_t$ are diffusion processes, that is, $f_t = f_t(X_t)$, $\Si_t = \Si_t(X_t)$, $\hat f_t = \hat f_t(X_t)$, $\hat \Si_t = \hat \Si(\hat X_t)$.
Suppose also that for $t\in [0,T]$,
\begin{enumerate}
    \item 
    $\E \ba{
    \one_{\tau(\hat X_t)>t}
    \ve{f_t(\hat X_t)-\hat f_t(\hat X_t)}^2}\le \ep_f(t)^2$ and 
    $\E \ba{
    \one_{\tau(X_t)>t}
    \ve{\Si_t(X_t)-\hat \Si_t(X_t)}_F^2}\le \ep_\Si(t)^2$, where $X_t, \hat X_t$ evolve according to the SDE \eqref{e:sdes}.
    \item 
    $f_t$ is $L_f$-segmentwise Lipschitz and 
    $\hat \Si_t$ is $L_\Si$-segmentwise Lipschitz on $\calS$ (with respect to the $\ve{\cdot}_2$ and $\ve{\cdot}_F$ norms), %
    for some $L_f, L_\Si>0$.
    \item 
    $\P(X_{[0,T]} \in \fS_T(r)) \ge 1-\de_1$.
\end{enumerate}
    Then for $t\in [0,T]$, 
    with $X_t$ and $\hat X_t$ driven by the same Brownian motion in \eqref{e:sdes} (i.e. synchronous coupling), 
we have 
    \[
\E \ba{\one_{\fS^2_t(r)}[X_{[0,t]}, \hat X_{[0,t]}]\ve{X_t -\hat X_t}^2} \le 
\int_0^t 
\pa{\fc{\sqrt 2}{L_f} \ep_f(s)^2 + 2\ep_\Si(s)^2}
e^{(2\sqrt 2 L_f + 2L_\Si^2)(t-s)}ds.
    \]
    Moreover, if $r_2\le r$ and $t\le \rc{2L_1}\log r_2$,
    \[
\P[(X_{[0,t]}, \hat X_{[0,t]}) \nin \fS^2(r,r_2)] \le 
\de_1 + 
\fc{32 tL_\Si^2 + \int_0^t 32 \ep_\Si(s)^2 + 8 \ep_f(s)\,ds}{\log r_2} + \rc{r_2}.
    \]
\end{lemma}
Note in particular that if $\ep_f(t), \ep_{\Si}(t), L_f,L_\Si = O(1)$ (uniformly over $t\in [0,T]$), $\de_1 = o_n(1)$, and $r = \om_n(1)$, then the expectation is $O(1)$ and the bad probability is $o_n(1)$.
\begin{prf}
    Let $\calF_t$ be the filtration for the Brownian motion.
    Let $\De_t:= X_t-\hat X_t$. 
    Note 
    \[
d\De_t = (f_t(X_t) - \hat f_t(\hat X_t))dt
+ (\Si_t(X_t) - \hat \Si_t(\hat X_t))dB_t.
    \]
    By Itô's formula, for $t\ge u$, 
    \allowdisplaybreaks
    \begin{align*}
d\ve{\De_t}^2 &= 
2\an{\De_t, (f_t(X_t) - \hat f_t(\hat X_t))dt
+ (\Si_t(X_t) - \hat \Si_t(\hat X_t))dB_t}
+ \ve{\Si_t(X_t) - \hat \Si_t(\hat X_t)}_F^2dt\\
\implies
\ddd t\E \ba{\ve{\De_t}^2 \mid \calF_u}
&= 2 \E \ba{\an{\De_t, f_t(X_t) - \hat f_t(\hat X_t)} \mid \calF_u} + 
\E \ba{\ve{\Si_t(X_t) - \hat \Si_t(\hat X_t)}_F^2\mid \calF_u}\\
&\le c\E \ba{\ve{\De_t}^2\mid \calF_u} + \rc{c}\E\ba{ \ve{f_t(X_t)-\hat f_t(\hat X_t)}^2\mid \calF_u} + \E\ba{ \ve{\Si_t(X_t) - \hat \Si_t(\hat X_t)}_F^2\mid \calF_u}
    \end{align*}
    for any $c>0$.
Now for $x\in \fS(r)$ and $\hat x\in B(x,r)$, by the definition of $\fS(r)$ and segmentwise Lipschitzness on $\calS$, 
\allowdisplaybreaks
\begin{align*}
    \ve{\hat f_t(\hat x) - f_t(x)}^2 &\le 2 \ve{ f_t(\hat x) - f_t(x)}^2 + 
    2 \ve{\hat f_t(\hat x) - f_t(\hat x)}^2
    \le 2L_f^2 \ve{\hat x - x}^2 + 2 %
    \ve{\hat f_t(\hat x) - f_t(\hat x)}^2
    \\
    \ve{\hat \Si_t(\hat x) - \Si_t(x)}_F^2
    &\le 2\ve{\hat \Si_t(\hat x) - \hat \Si_t(x)}_F^2 + 2\ve{\hat \Si_t(x) -  \Si_t(x)}_F^2
        \le 
        2L_{\Si}^2\ve{\hat x - x}^2 + 2 %
        \ve{\hat \Si_t(x) -  \Si_t(x)}_F^2.
    \end{align*}
    Plugging in gives that for $t=u$, on the event $\fS^2(r)$, 
    \allowdisplaybreaks
    \begin{align*}
        \ddd t\E \ba{\ve{\De_t}^2\mid \calF_u}
        &\le 
        \pa{c+\fc{2L_f^2}{c} + 2L_\Si^2} \E \ba{\ve{\De_t}^2\mid \calF_u} + \fc{2}{c} \E \ba{\ve{\hat f_t(\hat X_t) - f_t(\hat X_t)}^2\mid \calF_u} + 2 \E \ba{\ve{\hat \Si_t(X_t) -  \Si_t(X_t)}_F^2\mid \calF_u}\\
        &= 
        \pa{2\sqrt 2 L_f + 2L_\Si^2} \E \ba{\ve{\De_t}^2\mid \calF_u} + \fc{\sqrt 2}{L_f} \E \ba{\ve{\hat f_t(\hat X_t) - f_t(\hat X_t)}^2\mid \calF_u} + 2\E \ba{\ve{\hat \Si_t(X_t) -  \Si_t(X_t)}_F^2\mid \calF_u}
    \end{align*}
     by taking $c=\sqrt 2 L_f$.
Now take $L = 2\sqrt 2 L_f + 2L_\Si^2$ and %
$E_t(x,\hat x) = \fc{\sqrt2}{L_f}\ve{\hat f_t(\hat x) - f_t(\hat x)}^2 + 2\ve{\hat \Si_t(x) -  \Si_t(x)}_F^2$, and $\ep(t) = \fc{\sqrt 2}{L_f} \ep_f(t)^2 + 2\ep_\Si(t)^2$.
This shows that for $t=u$, on the event $\fS^2(r)$,
\[
        \ddd t\E [M_t \mid \calF_u]\le 0
        \text{ where }
        M_t = e^{-Lt} \ve{\De_t}^2 - \int_0^t %
        E_t(X_t, \hat X_t) e^{-Ls}\,ds 
\]
Define the stopping time 
\[
\tau_2 = \inf\set{t\in [0,T]}{(X_{[0,t]}, \hat X_{[0,t]})\nin \fS^2_t(r)}.
\]
On the event that $\tau_2 \le u$, we have $\E[\ve{\De_{t\wedge \tau_2}}^2 \mid \calF_u]=\ve{\De_{\tau_2}}^2$, so the above chain of inequalities also holds. Hence 
\[
        \ddd t\E \ba{M_{t\wedge \tau_2}} \le 0.
\]
Hence
\allowdisplaybreaks
\begin{align*}
\E[ \one_{\fS_t^2(r)}[X_{[0,t]}, \hat X_{[0,t]}] \ve{X_t-\hat X_t}^2] \le \E\ba{\ve{\De_{t\wedge \tau_2}}^2} &\le 
\int_0^t 
\E\ba{\one_{\fS^2(r)}[X_{[0,t]}, \hat X_{[0,t]}] E_t(X_t,\hat X_t)}
e^{L(t-s)}ds\\
&\le \int_0^t 
\ep(s)
e^{L(t-s)}ds.
\end{align*}
The second part follows from the next lemma, \pref{l:safe-pair} with $L_b = L_f$, $\ep_b(t) = \ep_f(t)$, $L_C=L_\Si$, $\ep_C(t) = \ep_\Si(t)$.
\end{prf}

The following lemma shows that, given that an ``ideal'' process stays within the expanded safe set with high probability, given Lipschitz and expected error bounds for the drift and diffusion terms for the difference, the approximate process also stays close with high probability.
\begin{lemma}
\label{l:safe-pair}
    Consider the setup of \eqref{e:sdes}, and suppose that
    $\De_t = X_t - \hat X_t$ satisfies
    \[
d\De_t = b_t \,dt + C_t\,dB_t.
    \]
Define the stopping time $\tau_2 = \inf\set{t\ge 0}{(X_t,\hat X_t)\nin \calS^2(r)}$.
Suppose that we can write $b_t = b_t^{\mathrm{Lip}} + b_t^{\mathrm{err}}$ and 
$C_t = C_t^{\mathrm{Lip}} + C_t^{\mathrm{err}}$ such that when $(X_t, \hat X_t)\in \calS^2(r)$,
\begin{align*}
    \ve{b_t^{\mathrm{Lip}}} &\le L_b \ve{\De_t},&
    \E \ba{\one_{t<\tau_2}\ve{b_t^{\mathrm{err}}}^2} &\le \ep_b(t)^2\\
    \ve{C_t^{\mathrm{Lip}}}_F &\le L_C \ve{\De_t},&
    \E \ba{\one_{t<\tau_2}\ve{C_t^{\mathrm{err}}}_F^2} &\le \ep_C(t)^2.
\end{align*}
Suppose that $\P(X_{[0,t]}\nin\fS_t(r)) \le \de_1$.
Then if $r\le r_2$ and $t\le \rc{2L_1}\log r_2$,
    \[
\P[(X_{[0,t]}, \hat X_{[0,t]}) \nin \fS^2(r,r_2)] \le 
\de_1 + 
\fc{32 tL_C^2 + \int_0^t 32 \ep_C(s)^2 + 8 \ep_b(s)\,ds}{\log r_2} + \rc{r_2}.
    \]
\end{lemma}
\begin{prf}
We will use a (super)martingale concentration inequality. Let $r_t = \sqrt{\ve{\De_t}^2+1}$.
By It\^o's formula,
\allowdisplaybreaks
\begin{align*}
    d\log (\ve{\De_t}^2 + 1) 
    & = \fc{2\an{\De_t, b_t \,dt + C_t\,dB_t}}{r_t^2}  
    + \fc{2\ve{C_t}_F^2}{r_t^2}\,dt
    - \fc{\ve{C_t^\top \De_t}^2}{r_t^4}dt.
\end{align*}
On the event $\fS_t^2(r)$, 
\allowdisplaybreaks
\begin{align*}
    \fc{\ve{b_t}}{r_t} &\le 
    \fc{\ve{b_t\lip}}{r_t} + 
    \fc{\ve{b_t\err}}{r_t}
    \le L_b  + 
    \ve{b_t\err}\\
    \fc{\ve{C_t}_F^2}{r_t^2} 
    &\le \fc{2\ve{C_t\lip}_F^2}{r_t^2} + \fc{2\ve{C_t\err}_F^2}{r_t^2}
        \le 
        2L_C^2%
         +  
        2\ve{C_t\err}_F^2.
\end{align*}
Hence, for $t<\tau_2$, 
\allowdisplaybreaks
\begin{align*}
    d\log (\ve{\De_t}^2 + 1) 
    & \le \pa{2L_b + 4L_C^2 + 2 \ve{b_t\err} + 4\ve{C_t\err}_F^2}dt + \fc{2\an{\De_t, C_t\,dB_t}}{r_t^2}.
\end{align*}
Hence, letting $L_1 = 2L_b + 4L_C^2$, $F_t =2 \ve{b_t\err} + 4\ve{C_t\err}_F^2$, and $N_t = \log (\ve{\De_t}^2 + 1)  - L_1 t - \int_0^t F_s\,ds$, for $t<\tau_2$,
\allowdisplaybreaks
\begin{align*}
    dN_t \le \fc{2\an{\De_t, C_t\,dB_t}}{r_t^2}.
\end{align*}
The quadratic variation is bounded as 
\allowdisplaybreaks
\begin{align*}
    [N]_{t\wedge \tau_2} &\le \int_0^{t\wedge \tau_2} \fc{4\ve{\De_s}^2 \opnorm{C_s}^2}{r_s^4}\,ds\le \int_0^{t\wedge \tau_2} \fc{4\ve{C_s}_F^2}{r_s^2}\,ds
    \le \int_0^{t\wedge \tau_2} 8L_C^2 + 8 \ve{C_s\err}_F^2\,ds\le 8tL_C^2 +  8\int_0^{t\wedge \tau_2} \ve{C_s\err}_F^2\,ds.
\end{align*}
We have the martingale concentration bound
\begin{align*}
    \P \pa{\sup_{s\le t\wedge \tau_2} N_s \ge x, \, [N]_{t\wedge \tau_2} \le v}\le e^{-\fc{x^2}{2v}}.
\end{align*}
We now take expectations for $[N]_{t\wedge \tau_2}$ and the random part of $N_{t\wedge \tau_2}$. %
We have
\allowdisplaybreaks
\begin{align*}
    \E[N]_{t\wedge \tau_2} &\le 8 tL_C^2 + 8\int_0^t \E\ba{\one_{s\le \tau_2}\ve{C_s\err}_F^2}\,ds
    \le 8 tL_C^2 + 8\int_0^{t} \ep_C(s)^2\,ds=:A\\
    \E \int_0^{t\wedge \tau_2} F_s\,ds &\le \int_0^t 2\ep_b(s) + 4 \ep_C(s)^2 \,ds=:B.
\end{align*}
By Markov's inequality, with probability $1-\de$,
\begin{align*}
    [N]_{t\wedge \tau_2} &\le \fc{2}{\de} A,&%
    \int_0^{t\wedge \tau_2} F_s\,ds &\le \fc{2}{\de} B. %
\end{align*}
Let $\calA$ be the event when both these hold. Then 
\begin{align*}
    \P \pa{\sup_{s\le t\wedge \tau_2} N_s \ge x, \, \calA}\le \exp\pa{-\fc{x^2}{4A/\de}}.
\end{align*}
Note that if $\sup_{s\le t\wedge \tau_2} \ve{\De_s}^2 \ge e^{L_1t+\fc{2}{\de}B+x}$, under $\calA$, then the above event holds. 
We have, for $e^{L_1s + \fc{2}{\de}B+x}\le r$ and $x\ge 0$, 
\allowdisplaybreaks
\begin{align*}
    &\P\pa{\{\forall t\in [0,T],\, X_t\in \calS\}^c\cup \bc{\sup_{s\le t}\ve{\De_s}^2 \ge e^{L_1t + \fc{2}{\de}B+x}}}\\
    &\le 
    \P\pa{\{\forall t\in [0,T],\, X_t\in \calS\}^c}
    + 
    \P\pa{\calA^c}
    + 
    \P\pa{\{\forall t\in [0,T],\, X_t\in \calS\}\cap \calA\cap \bc{\sup_{s\le t\wedge \tau_2}\ve{\De_s}^2 \ge e^{L_1t + \fc{2}{\de}B+x}}}\\
    &\le \de_1 + \de + \exp\pa{-\fc{x^2}{4A/\de}}.
\end{align*}
The first inequality is because, if  $\ve{\De_t}^2 \ge e^{L_1s + \fc{2}{\de}B+x}$, then either $t<\tau_2$ and hence $\sup_{s\le t\wedge \tau_2} \ve{\De_s}^2\ge e^{L_1s + \fc{2}{\de}B+x}$, or $t\ge \tau_2$. In the latter case, 
\[
\sup_{s\le t\wedge \tau_2} \ve{\De_s}^2\ge
\ve{\De_{\tau_2}}^2\ge r\ge e^{L_1s + \fc{2}{\de}B+x}
\]
as well. Choosing $x$ such that $e^{L_1s + \fc{2}{\de}B+x}=r_2^2\le r^2$ gives
\[
\P\ba{(X_{[0,t]}, \hat X_{[0,t]}) \nin \fS^2(r,r_2)}
\le 
\de_1 + \de + \exp\pa{-\fc{\de(2\log r_2 - L_1 t - \fc{2B}{\de})^2}{4A}}.
\]
If $L_1t\le \rc 2\log r_2$, then taking $\de = \fc{4\max\{A,B\}}{\log r_2}$ gives that this is $\le \de_1 + \fc{4\max\{A,B\}}{\log r_2} + \rc{r_2}$. 
\end{prf}

Given a path $(y_t)_{t\ge 0}$, define $m_t^y$ to be the magnetization process if $(m_t^y, y_t)$ solves (ASL-TAP), up to the exit time from $\calS_{\be, A}(c)$. 
Note that because we only have unique solutions to the TAP equation locally but not globally, $m_t^y$ cannot simply be defined as a function $\hat m(y_t)$ as in \cite{davies2026potential}, but can depend on the path. If $\hat y(\Safe(c))$ is not simply connected, then a loop in $y$-space may lift to a path with different endpoints in $m$-space.
\begin{lemma}[Error in estimated magnetization]
\label{l:m-error}
Suppose that \ref{d:Q-error}, \ref{d:safe-set-stay-inside}, \ref{d:Q-reg}, \ref{d:drift-error}, and \ref{d:Q-Lip} hold to time $T$.
    Let $y_t$ denote the SL process, and $\mg_t=\an{\si}_{y_t}$ be the true magnetization for tilt $y_t$. %
    Then for each $t\in [0,T]$ there is a constant $C(t)$ such that 
    \[
\E\ba{\one_{\calG_t}\ve{m_t^y-\mg_t}^2}\le C(t),
    \]
    where $\calG_t = \{((m,y)_{[0,t]}, (\hm,\hy)_{[0,t]})
\in \fS_t^2(\rh\sqrt n,c\rh\sqrt n/2)\}$ and $\P(\calG_t)=1-o_n(1)$. Here, $c=\cT(\be, T)$ and $\rh = \rh(\be, T)$ as in \ref{d:safe-set-stay-inside}.
\end{lemma}
\begin{prf}
Well-posedness of the SDE's is proven by \pref{lem:stay-in-cube}.
Consider the synchronously coupled SDE's
\allowdisplaybreaks
\begin{align*}
    (\textup{PHD-TAP})\quad &
    \begin{cases}
        d\hm_t = \hat Q(\hm_t)dB_t\\
        d\hy_t = f(\hm_t)dt + dB_t
    \end{cases} \\
    (\textup{Ideal})\quad &
    \begin{cases}
        d\mg_t= Q(\mg_t)dB_t = \Cov(\mu_{y_t}) dB_t\\
        dy_t= \mg_tdt+dB_t
    \end{cases}
\end{align*}
over $\Om:= (-1,1)^n\times \R^n$. 
Then
\allowdisplaybreaks
\begin{align*}
    \E_{A, B_t}\ba{\one_{\Safe(c)}[m_t] \ve{Q(m_t) - \hQ(m_t)}_F^2}
    & \le 
     \E_{A, B_t}\ba{\one_{\Safe(c)} [m_t]\opnorm{\hQ(m_t)}^2 \ve{Q(m_t) - \hQ(m_t)}_F^2}\\
     &\le_{\tref{d:safe-set-stay-inside}}
     c^{-2} \E_{A, B_t}\ba{\one_{\Safe(c)}  [m_t]\ve{\hQ(m_t)^{-1}Q(m_t) - I_n}_F^2} \\
     &= c^{-2}\E_{A, B_t}\ba{\one_{\Safe(c)}[m_t]\norm{\nabla_m^2\FT(m_t)
        Q_t - \Id_n}^2_F}\\
    &\le_{\tref{d:Q-error}} c^{-2}\eps_{\mathrm{cov}}(\beta, t)^2\,.
\end{align*}
Therefore, by Markov's inequality, for any $\de$, with probability $\ge 1-\de$ over $A$, 
\[
\E_{B_t}\ba{\one_{\Safe(c)}[m_t] \ve{Q(m_t) - \hQ(m_t)}_F^2} = O(1)
\]
where we consider $\be, T,\de$ fixed. We also have, with probability $\ge 1-o_n(1)$ over $A$, 
\begin{align*}
\E_{B_t} \ba{\one_{\Safe(c)}[\hm_t] \ve{f(\hm_t) - \hm_t}^2}
&\le_{\tref{d:drift-error}} \ep_{\mathrm{drift}}^2 = O(1).
\end{align*}
By \ref{d:Q-Lip}, we have that $\hQ$ is $O(1)$-Lipschitz with respect to $\ved_{\Om}$, and the identity map is clearly 1-Lipschitz. \ref{d:safe-set-stay-inside} supplies $1-o_n(1)$ probability of $(m_t,y_t)$ staying inside $\Safe(c)$. 
Define $\fS_t^2(r_1,r_2)$ using $\Safe(c)$ as the safe set. 
Then \pref{l:W-gw} applies to give 
\allowdisplaybreaks
\begin{align*}
\E\ba{\one_{\fS_t^2(\rh\sqrt n)}[(m,y)_{[0,t]}, (\hm,\hy)_{[0,t]}] 
\ve{(m_t,y_t) - (\hm_t,\hy_t)}^2
} &= O(1)\\
\P[((m,y)_{[0,t]}, (\hm,\hy)_{[0,t]})\nin \fS_t^2(\rh\sqrt n, \rh\sqrt n/2)] &= o_n(1).
\end{align*}
Consider the map $\hy$ given by \eqref{e:hy}. 
We aim to apply \pref{l:inv-f}, but because the ball is truncated by the cube, we work in dual space, $x=\tanh^{-1}(m)$. First, we claim that for $m\in (-1,1)^n$, $B(\tanh^{-1}(m),r)\subeq \tanh^{-1}(B(m,r)\cap (-1,1)^n)$. Since $\tanh:\R^n\to (-1,1)^n$ is a homeomorphism, this is equivalent to $\tanh(B(\tanh^{-1}(m),r))\subeq B(m,r)\cap (-1,1)^n$. This is clear since $\tanh$ is contractive.

Now consider $\hat y\circ \tanh: \tanh^{-1}(\Safe(c))\to \R^n$. Note $D(\hat y\circ \tanh)(\tanh^{-1}(m))\succeq \hQ(m)^{-1}D(m)^{-1} \succeq cI_n$. 
When $((m,y)_{[0,t]}, (\hm,\hy)_{[0,t]})\in \fS_t^2(\rh\sqrt n, \rh\sqrt n/2)$, $B(\tanh^{-1}(\hm_t), \rh\sqrt n/2)\subeq \tanh^{-1}(B(\hat m_t, \rh\sqrt n/2)\cap (-1,1)^n)\subeq \tanh^{-1}(\Safe(c))$, so by \pref{l:inv-f}, 
\[B(\hy(\hm_t) = \hy_t, c\rh\sqrt n/2) \subeq 
\hy(B(\hat m_t, \rh\sqrt n/2)\cap (-1,1)^n)
\subeq \hy(\Safe(c)).\] 
Now let 
\[\calG_t = \{((m,y)_{[0,t]}, (\hm,\hy)_{[0,t]})
\in \fS_t^2(\rh\sqrt n,c\rh\sqrt n/2)\}.\] 
Under $\calG_t$, for all $s\in [0,t]$,
$y_s\in B(\hy_s,c\rh\sqrt n/2)$,  %
so  considering the local inverse map $\hm_{\hy_s}:B(\hy_s, c\rh\sqrt n/2)\to B(\hm_s, \rh\sqrt n/2)$, by continuity of $m^y_s$, we have that $\hm_{\hy_s}(y_s)= m^y_s$ and
\[
\ve{m^y_s - \hm_s}
= \ve{\hm_{\hy_s}(y_s) - \hm_{\hy_s}(\hy_s)}
\le c^{-1}\ve{y_s-\hy_s}.
\]
since $D\hm_{\hy_s}(y) = \hQ(\hm_{\hy_s}(y))$ has operator norm at most $c^{-1}$.
Hence under $\calG_t$,
\allowdisplaybreaks
\begin{align*}
    \E\ba{\one_{\calG_t}\ve{\mg_t- m_t^y}^2}
    &\lesssim
    \E[\one_{\calG_t}\ve{\mg_t-\hm_t}^2] + \E[\one_{\calG_t}\ve{\hm_t-m_t^y}^2] \\
    &\lesssim  \E[\one_{\calG_t}\ve{\mg_t-\hm_t}^2] + \E[\one_{\calG_t}\ve{\hy_t-y_t}^2] 
    =O(1). \qedhere
\end{align*}
\begin{lemma}[{Local Hadamard  Theorem, \cite[5.3.11]{ortega2000iterative}}]\label{l:inv-f}
    Let $f:\Omega\to \R^n$ be a continuously differentiable map where $\Omega \subseteq \R^n$ is open. Suppose $B(m_0,r)\subseteq \Omega$ and $Df\succeq cI_n$ on $\Omega$, $c>0$. Then $B(f(m_0), cr)\subeq f(\Omega)$.
    Moreover, $f$ is injective on any convex subset of $\Om$.
\end{lemma}
\end{prf}
A standard Girsanov argument (see e.g. \cite[\S 4.4]{chewi2024log}) then bounds the KL divergence between the SL and ASL-TAP.
\begin{corollary}[$O(1)$ KL error for ASL-TAP]
\label{c:O1-KL}
Suppose that \ref{d:Q-error}, \ref{d:safe-set-stay-inside}, \ref{d:Q-reg}, \ref{d:drift-error}, and \ref{d:Q-Lip} hold to time $T$. Let $y_t$ and $\hy_t$ follow (SL) and (ASL-TAP) (with any extension outside of $\Safe(c)$). %
Then there is a distribution $Q_T$ such that 
\[
\TV(\dist(y_T), Q_T) = o_n(1), \quad 
\KL(Q_T\|\dist(\hy_T)) = O(1).
\]
\end{corollary}
\begin{prf}
Let $\tau=\inf\set{t\ge0}{((m_t,y_t),(\hm_t,\hy_t))\nin \calS_t^2(\rh\sqrt n, c\rh\sqrt n/2)}$. 
Let $\bP_T$ and $\hat \bP_T$ denote the path measures of $(y_{t\wedge \tau})_{t=0}^T$ and $(\hy_{t\wedge \tau})_{t=0}^T$. 
Since $m^y_t , m_t$ are bounded, Novikov's condition holds: $\E \exp\pa{\rc 2 \int_0^T \ve{m^y_t - m_t}^2dt}<\iy$. 
By the data processing inequality and Girsanov's Theorem,
\allowdisplaybreaks
\begin{align*}
    \KL(\dist(y_{T\wedge \tau})\|\dist(\hy_{T\wedge \tau}))
    &\le 
    \KL(\bP_T\|\hat \bP_T)
    \le 
    \E_{\bP_T}
    \log \dd{\bP_T}{\hat \bP_T}\\
    &=
    -\E  
    \int_0^{T\wedge \tau} \ip{\mg_t - m^y_t}{dB_t} + \rc 2 \E
    \int_0^{T\wedge \tau} \ve{\mg_t - m^y_t}^2dt\\
    &= \rc 2 
    \int_0^{T\wedge \tau} \E\ve{\mg_t - m^y_t}^2\,dt = O(1),
\end{align*}
    where the last equality follows from \pref{l:m-error}. Noting that $\TV(\dist(y_{T\wedge \tau}), \dist(y_T)), \TV(\dist(\hy_{T\wedge \tau}), \dist(\hy_{T}))\le \P[\tau<T] = o_n(1)$ finishes the proof.
\end{prf}

\subsection{Jarzynski's equality with survival}
In the following, we derive a specific Jarzynski's equality formula to compare the Radon-Nikodym derivative between the SL-TAP and ASL-TAP process on path space for $m$, as well as the final-time marginal $m_T$, restricted to survival of the path. 
Note that we consider the final-time marginal $m^y_T$ rather than $y_T$, because the map $\hat y(m)$ may not be injective.
As long as the survival event has probability $1-o(1)$, we can then do rejection sampling with the provided ratio, with an error of $o(1)$.
\begin{lemma}[Jarzynski's equality with survival for SL-to-ASL ratio]
\label{l:je-survival}
Let $\tau = \inf\set{t\ge 0}{m_t \nin \Safe(c)}$ and $S_T = \{\tau>T\}$. For a measure $\bQ_T$ on path space $m:[0,T]\to \R^n$, let $\bQ_T^s (\cdot) = \bQ_T^s(\cdot \cap S_T)$, and let $q_T^s(m) = \E_{m_{[0,T]}\sim \bQ_T}[\one_{S_T} \mid m_T = m]$, i.e., the marginal distribution of $m_T$ under $\bQ_T^s$. 
Let $\bP_T$ denote the path measure $m_{[0,T]}$ in \eqref{e:SL-TAP} and $\hat \bP_T$ denote the path measure of $\hm_{[0,T]}$ in the \eqref{e:ASL-TAP}. 
    Then we have that
    \allowdisplaybreaks
    \begin{align*}
        \dd{\bP_T^s}{\hat \bP_T^s}(m_{[0,T]})
        &= e^{w_T}\cdot 
        \fc{e^{-\FT(0,0)}}{e^{-\FT(m_T, \hy(m_T))}} \cdot e^{-\fc{nT}{2}} \cdot \fc{Z_{\be A, \hy(m_T)}}{Z_{\be A,0}}
        \\
        \dd{p^s_T}{\hat p_T} (m) &= 
        \E[e^{w_T}\one_{S_T} \mid m_T=m] \cdot 
        \fc{e^{-\FT(0,0)}}{e^{-\FT(m, \hy(m))}} \cdot e^{-\fc{nT}{2}} \cdot \fc{Z_{\be A, \hat y(m)}}{Z_{\be A,0}}
    \end{align*}
    where $w_T$ is defined by 
    \begin{align*}
dw_t &= \om(m_t)\,dt, & 
\om(m):&= \rc 2 \ba{\Tr(\hQ(m)) + \ve{m}^2}.
    \end{align*}
\end{lemma}
\begin{prf}
    We rederive the calculations from \cite[\S3.2]{davies2026potential}, but taking $m=m^y_t$ instead of assuming there is a unique TAP solution $m=\hm(y_t)$. Given that $(m_t,y_t)$ is an It\^o process satisfying \eqref{e:TAP} for all $t$, by It\^o's formula,
\begin{align*}
    d\FT(m_t,y_t) &=
    \ub{\an{\gd_m\FT(m_t,y_t), dm_t}}0
    - \an{m_t,dy_t}
    + \rc 2 \Tr\pa{\gd_m^2 \FT(m_t,y_t) d[y]_t}
    - \Tr(d[m,y]_t).
\end{align*}
When $y_t$ satisfies \eqref{e:SL}, letting $m^y_t$ be the solution to \eqref{e:TAP}, 
\begin{align*}
    d\FT(m_t^y,y_t) &= 
    -\an{m_t^y, dy_t} + \rc 2 \Tr\pa{\gd_m^2 \FT(m_t,y_t) \hQ(m_t^y) \hQ(m_t^y)^\top }dt
    - \Tr(\hQ(m^y_t))\\
    &= -\an{\hat m^y_t,dy_t} - \rc 2 \Tr(\hQ(m^y_t))\,dt.
\end{align*}
Let $\bW_T$ denote the Wiener measure. 
We consider the stopped ASL process defined by 
\[
d\hy_t = \one_{t<\tau} \hm_t \,dt + dB_t,
\]
and denote the path measure by $\hat \bP_T^y$.
By Girsanov's Theorem,
\begin{align*}
    \dd{\hat{\bP}_T^y}{\bW_T}(y_{[0,T]}) &= \exp\pa{
\int_0^{T\wedge \tau}
\an{m^y_t, dy_t} - \rc 2\int_0^{T\wedge \tau} \ve{m^y_t}^2\,dt
    }.
\end{align*}
Then
\allowdisplaybreaks
\begin{align*}
    \one_{S_T}[m^y_{[0,T]}] e^{\hw_T(m^y_{[0,T]})} \dd{\hat{\bP}_T^y}{\bW_T}(y_{[0,T]})
    &= \one_{S_T}\exp\pa{\int_0^T \rc 2 \pa{\Tr(\hQ(m^y_t)) + \ve{m^y_t}^2}\,dt}\\
    &= \one_{S_T}\exp\pa{-\int_0^T d\FT(m_t^y,y_t)}= \exp(\FT(0,0) - \FT(m_T^y, y_T)).
\end{align*}
Define the survival event on paths $y_{[0,T]}$ to be that $m^y_{t}\in \Safe(c)$ for $t\le T$.
Therefore, 
\[
\dd{\hat\bP^{y,s}_T}{\bW^s_T}(y_{[0,T]}) = e^{-\hw_T(m_{[0,T]}^y)} \fc{e^{- \FT(m^y_T, y_T)}}{e^{-\FT(0,0)}}.
\]
The SL process is given by $y_t = t\si + B_t$, i.e., $dy_t = \si\,dt + dB_t$, where $\si\sim \mu_{\be A}$. By Girsanov's Theorem, conditioned on $\si$,
\allowdisplaybreaks
\begin{align*}
    \dd{\bP_T^y(\cdot\mid \si)}{\bW_T} (y_{[0,T]})
    &= \exp\pa{\int_0^T\an{\si, dy_t} - \rc 2 \int_0^T \ve{\si}^2\,dt} 
    = \exp\pa{\an{\si,y_T} - \fc{nT}{2}}.
\end{align*}
Averaging over $\si\sim \mu_{\be A}$ then gives
\begin{align*}
    \dd{\bP_T^y}{\bW_T}(y_{[0,T]}) &= \sum_{\si\in \{\pm 1\}^n} 
    \fc{\exp\pa{\fc\be2 \an{\si, A\si}}}{Z_{\be A, 0}} \exp\pa{\an{\si,y_T} - \fc{nT}{2}} = e^{-\fc{nT}{2}} \fc{Z_{\be A, y_T}}{Z_{\be A, 0}}.
\end{align*}
Note this only depends on the endpoint of the path. We have that $\dd{\bP_T^s}{\bW_T^s}$ equals the same expression. Then
\allowdisplaybreaks
\begin{align*}
    \dd{\bP_T^{y,s}}{\hat{\bP}_T^{y,s}}
    &= \dd{\bP_T^{y,s}}{\bW_T^s}\Big/
     \dd{\hat{\bP}_T^{y,s}}{\bW_T^s}
    = e^{\hw_T(m^y_{[0,T]})}\cdot 
        \fc{e^{-\FT(0,0)}}{e^{-\FT(m^y_T, y_T)}} \cdot e^{-\fc{nT}{2}} \cdot \fc{Z_{\be A, y_T}}{Z_{\be A,0}}.
\end{align*}
Pushing forward to $m$-space via $y_{[0,T]}\mapsto m^y_{[0,T]}$, multiplying by $\one_{S_T}$, and taking a conditional expectation with respect to the endpoint $m^y_T$ then gives the second result.
\end{prf}

We need to show that with high probability, for the SL process $y_t$, considering the $m^y_t$ chosen to satisfy the TAP equation, survival occurs with high probability.
\begin{lemma}
\label{l:my-safe}
Suppose that \ref{d:Q-error}, \ref{d:safe-set-stay-inside}, \ref{d:Q-reg}, \ref{d:drift-error}, and \ref{d:Q-Lip} hold to time $T$.
    Let $y_t$ solve (SL). 
    Define $\fS_t$ on the magnetization path space corresponding to $\Safe(c)$. 
    Then for all $t\in [0,T]$,
    \[
    \P[m^y_t\nin \fS_t] = o_n(1).
    \]
\end{lemma}
\begin{prf}
    Abbreviate $\td m_t= m^y_t$. We consider the mismatch between the true magnetization and the magnetization estimated from following TAP equation,
\begin{align*}
    dy_t &= m_t \,dt + dB_t\\
    dm_t &= Q(m_t)\,dB_t\\
    d\td m_t &= \hQ(\td m_t) (dy_t - f(\td m_t)\,dt)= \hQ(\td m_t) ((m_t- f(\td m_t)) \,dt + dB_t).
\end{align*}
Let $\td \De_t= m_t - \td m_t$. We have
\allowdisplaybreaks
\begin{align*}
d\td \De_t &= 
[\ub{\hQ(\td m_t)(m_t - \td m_t)}{b\lip} + \ub{\hQ(\td m_t)(\td m_t - f(\td m_t))}{b\err}]\,dt + 
[\ub{(\hQ(m_t) - \hQ(\td m_t))}{C\lip} + \ub{(Q(m_t) - \hQ(m_t))}{C\err}] \,dB_t
\end{align*}
Note
\allowdisplaybreaks
\begin{align*}
    \ve{b_t\lip} &\le c^{-1}\ve{\De_t} & \E\ba{\one_{t<\tau}\ve{b_t\err}^2} &\le c^{-2}\ep_{\mathrm{drift}}^2\\
    \ve{C_t\lip}_F &\le L\ve{\De_t} & 
    \E\ba{\one_{t<\tau}\ve{C_t\err}_F^2}&\le \ep_{\mathrm{cov}}^2.
\end{align*}
Then \pref{l:safe-pair} gives that (defining the safe sets on just the magnetization path space)
\[
\P[\td m_{[0,t]} \nin \fS_t]\le 
\P[(m_{[0,t]}, \td m_{[0,t]})\nin \fS^2_t(\rh\sqrt{n})] = o_n(1).
\]
\end{prf}

\subsection{Well-posedness}
\label{sec:well-posed}
\label{sec:well-posed-transportation}
We consider the initial condition $\hat m_0=0$ and the equations
\begin{align}
\label{eq:asl-tap-m}
 d\hat m_t&=\hat Q(\hat m_t)(\hat m_t-f(\hat m_t))\,dt
                                      +\hat Q(\hat m_t)\,dB_t,\\
\label{eq:phd-m}
 d\hat m_t&=\hat Q(\hat m_t)\,dB_t,
\end{align}
for ASL-TAP and PHD, respectively. %
Each equation is used up to its first exit from $\Safe(c)$, with its
own exit time denoted $\tau_c$.
As in \cite[\S6.6]{davies2026potential}, we introduce the dual coordinates
$\hat x_t=\tanh^{-1}(\hat m_t)$ and write
\[
 \hat\Sigma(x):=D(m)\hat Q(m),\qquad
 \hat f^*(x):=E_{\mathcal D_n}[\hat Q(m)^2]D(m)^2m,
 \qquad m=\tanh(x).
\]
Here $\hat f^*$ is not a convex conjugate. The two dual equations are
\begin{align}
\label{eq:d-PHD}
 d\hat x_t&=\hat f^*(\hat x_t)\,dt+\hat\Sigma(\hat x_t)\,dB_t,
                                                     \tag{d-PHD}\\
\label{eq:d-ASL-TAP}
 d\hat x_t&=\left[D(\hat m_t)\hat Q(\hat m_t)
              (\hat m_t-f(\hat m_t))+\hat f^*(\hat x_t)\right]dt
                     +\hat\Sigma(\hat x_t)\,dB_t. \tag{d-ASL-TAP}
\end{align}

\begin{lemma}[Boundedness of dual SDE coefficients]
\label{l:dual-phd-bounded}
Assume \ref{d:drift-error}. Then the following hold within the domain $\{x\in\R^n:\tanh(x)\in\Safe(c)\}$:
\begin{enumerate}[itemsep=0.1em]
\item \label{i:lip-hS}
$\opnorm{\hat\Sigma} \le c^{-1}$.
\item \label{i:lip-hfs}
$\lpnorm{\hat f^*} \le c^{-2}\sqrt{n}$.
\item With $m = \tanh(x)$, $\lpnorm{D(m)\hat Q(m)(m-f(m))} \le c^{-1}\eps_{\drift}$
\end{enumerate}
\end{lemma}
\begin{proof}
    These follow directly from the bound on $D(m)\hat{Q}(m)$ given by \pref{def:good-points}, the bound $\lpnorm{m} \le \sqrt{n}$ implied by $m \in (-1,1)^n$, and the bound on $\lpnorm{m - f(m)}$ given by \ref{d:drift-error}.
\end{proof}

\begin{lemma}[ASL-TAP and PHD stay in the cube up to safe-set exit]
\label{lem:stay-in-cube}
Assume $\opnorm{A} \le 3$, \ref{d:drift-error}, and $0^n\in\Safe(c)$ for $c > 0$.
Then the primal and dual equations \eqref{eq:asl-tap-m}--\eqref{eq:phd-m}
and \eqref{eq:d-PHD}--\eqref{eq:d-ASL-TAP} have pathwise unique strong
solutions up to their respective first exits from the good set $\Safe(c)$ or its
dual-coordinate image $\tanh^{-1}(\Safe(c))$. Their stopped solutions are defined for all
$t\ge0$, and almost surely
\[
            \hat m_{t\wedge\tau_c}\in(-1,1)^n
                           \qquad(t<\infty).
\]
The ASL tilt and JE weight also have finite limits at a finite exit and
can be stopped there.
\end{lemma}
\begin{proof}
Write either dual equation as
$d\hat x_t=b(\hat x_t)\,dt+\hat\Sigma(\hat x_t)\,dB_t$ on
$U:=\tanh^{-1}(\Safe(c))$, with $\hat x_0=0$.
Its coefficients are smooth on $U$. Then \pref{l:dual-phd-bounded} gives
\[
    \sup_{x\in U}\lpnorm{b(x)}\le \infty,
    \qquad
    \sup_{x\in U}\opnorm{\hat\Sigma(x)}\le \infty.
\]
With $\operatorname{dist}$ being the distance between a point and a set, define
\[
    U_k:=\{x\in U:\lpnorm{x}<k,\
                  \operatorname{dist}(x,U^c)>k^{-1}\}.
\]
For $k$ large enough that $0\in U_k$, use
$\overline U_k\subset U_{k+1}$ to choose
$\chi_k\in C_c^\infty(U_{k+1})$ with $0\le\chi_k\le1$ and
$\chi_k=1$ on $\overline U_k$.
The coefficient pairs $(\chi_k b,\chi_k\hat\Sigma)$, extended by zero
outside $U$, are globally Lipschitz, so
\cite[Theorem 5.2.1]{oksendal2003stochastic} gives global strong solutions
$\hat x^{(k)}$ with these coefficients, the same Brownian motion, and
initial value $0$.
Set $\tau_k:=\inf\{t:\hat x_t^{(k)}\notin U_k\}$.
The stopped uniqueness argument in that theorem gives, for $\ell\ge k$,
\[
    \tau_k\le\tau_\ell,\qquad
    \hat x^{(\ell)}_{t\wedge\tau_k}
      =\hat x^{(k)}_{t\wedge\tau_k}.
\]
Thus $\tau_c:=\lim_k\tau_k$ and
$\hat x_t:=\hat x_t^{(k)}$ whenever $t<\tau_k$ define the pathwise
unique solution before $\tau_c$.

Taking the integrands to be zero for $s\ge\tau_c$, we have
$\E_B\int_0^T\mathbf1_{\{s<\tau_c\}}
\|\hat\Sigma(\hat x_s)\|_F^2\,ds\le C_{\beta,c}^2nT$.
Hence It\^o isometry, \cite[Theorem 3.2.5]{oksendal2003stochastic},
and the drift bound give the continuous extension
\[
    \hat x_{t\wedge\tau_c}
    :=\int_0^t\mathbf1_{\{s<\tau_c\}}b(\hat x_s)\,ds
      +\int_0^t\mathbf1_{\{s<\tau_c\}}\hat\Sigma(\hat x_s)\,dB_s,
    \qquad t<\infty.
\]
On $\{\tau_c<\infty\}$, continuity gives $\hat x_{\tau_k}\in\partial U_k$, so
\[
    \lpnorm{\hat x_{\tau_k}}=k
    \quad\text{or}\quad
    \operatorname{dist}(\hat x_{\tau_k},U^c)=k^{-1}.
\]
Since $\hat x_{\tau_k}\to\hat x_{\tau_c}\in\R^n$, the first alternative is impossible for large $k$. Thus $\hat x_{\tau_c}\in\partial U$, and $\tau_c$ is the first exit from $U$.

It\^o's formula applied to $\tanh$ and locally to $\tanh^{-1}$ as in \cite[Lemma 6.9]{davies2026potential}, gives the primal equations, pathwise uniqueness, and $\hat m_{t\wedge\tau_c}=\tanh(\hat x_{t\wedge\tau_c})\in(-1,1)^n$.
Finally, for $t<\tau_c$, recover
\[
    \hat y_t=\int_0^t\hat m_s\,ds+B_t,\qquad
    \hat w_t=\frac12\int_0^t
       \bigl(\Tr\hat Q(\hat m_s)+\lpnorm{\hat m_s}^2\bigr)\,ds.
\]
Since $\lpnorm{\hat m_s}\le\sqrt n$ and $0\le\Tr\hat Q(\hat m_s)\le n/c$ before exit, both have finite limits at a finite $\tau_c$ and can be stopped there.
\end{proof}

\subsection{Pathwise concentration}
\label{sec:transportation}
We use the same transportation argument as
\cite[Definition 6.10 and Theorems 6.11--6.12]{davies2026potential}, applying it with globally Lipschitz extensions of the SDE's.
We briefly restate the definition and theorems here.

\begin{definition}
\label{def:Tp-transport-inequality}
A probability measure $P$ on a metric space satisfies a $T_p$ transportation inequality with constant $C$ if $d_{W,p}(P,\mu)^2\le2C H(\mu\mid P)$ for every probability measure $\mu$, where $d_{W,p}$ denotes Wasserstein distance and $H$ relative
entropy.
\end{definition}

\begin{theorem}[$T_1$ inequality and sub-Gaussian concentration~{\cite[Theorem 1.1, Bobkov \& G\"otze]{djellout2004}}]
\label{thm:t1-subgaussian}
If $P$ satisfies $T_1$ with constant $C$, then every integrable
$K_F$-Lipschitz function $F$ satisfies
\[
 \E_P [e^{\lambda(F-\E_P[F])}]
       \le e^{\lambda^2CK_F^2/2}\quad(\forall \lambda\in\R),\qquad
 P\{|F-\E_P[F]|>r\}\le2e^{-r^2/(2CK_F^2)}\quad(\forall r>0).
\]
\end{theorem}

\begin{theorem}[$T_2$ transport inequality for regular It\^o processes~{\cite[modification of Theorem 1]{ustunel2012transportation}}]
\label{thm:t2-transport-SDE}
Let $C,K,T>0$. Consider the global solution of $dX_t=b(X_t)\,dt+\sigma(X_t)\,dB_t$, $X_0=0$, on $[0,T]$.
Suppose that $\opnorm{\sigma(x)}\le C$ for all $x\in\R^n$, that $b$ is $K$-Lipschitz in Euclidean norm, and that $\sigma$ is
$K/\sqrt n$-Lipschitz in $\schnorm{\cdot}$.
Then its path law on $C^0([0,T],\R^n)$ with uniform Euclidean distance satisfies $T_2$ with constant $3TC^2\exp(3T(T+4)K^2)$.
\end{theorem}
\begin{proof}
This is the time-homogeneous special case of \cite[Theorem 6.12]{davies2026potential}.
\end{proof}

We now apply the transportation argument to the SDE's with Lipschitz-extended coefficients.

\begin{theorem}[Transport and sub-Gaussian concentration for an ASL-TAP extension]
\label{thm:mSDE-path-subgaussian}
Assume $\opnorm{A} \le 3$, \ref{d:drift-error}, \ref{d:Q-Lip}, and $0^n\in\Safe(c)$ for $c > 0$. Let $0<T<\infty$.
Then there is a global strong solution
\begin{equation}
\label{eq:alg-global-extension}
 d\hat m_t^{\mathrm{ext}}=b_{\mathrm{ext}}(\hat m_t^{\mathrm{ext}})\,dt
              +\sigma_{\mathrm{ext}}(\hat m_t^{\mathrm{ext}})\,dB_t,
                    \qquad \hat m_0^{\mathrm{ext}}=0,
\end{equation}
whose coefficients agree on $\Safe(c)$ with
$\hat Q(m)(m-f(m))$ and $\hat Q(m)$, respectively.
Driven by the same Brownian motion, it agrees with the original ASL-TAP process up to their common first exit from $\Safe(c)$.
Let $P$ be its path law on $\calW=C^0([0,T],\R^n)$ with the uniform norm $\|x\|_\infty=\sup_{0\le t\le T}\lpnorm{x_t}$.
For a constant $K$ depending on $L$ and $L_{\drift}$, set
\begin{equation}
\label{eq:alg-T2-constant}
            c_{T_2}:=3Tc^{-2}\exp\!\left(3T(T+4)K^2\right).
\end{equation}
Then:
\begin{enumerate}[label=\textup{(\roman*)},itemsep=0.3em]
\item For every probability measure $\mu$ on $\calW$,
$d_{W,2}(P,\mu)^2\le2c_{T_2}H(\mu\mid P)$.
\item For every $P$-integrable $K_F$-Lipschitz functional $F$,
\begin{equation}
\label{eq:alg-global-MGF}
 \E_P [e^{\lambda(F-\E_P[F])}]
       \le\exp\!\left(\frac{\lambda^2c_{T_2}K_F^2}{2}\right)
                                      \qquad(\forall \lambda\in\R).
\end{equation}
In sub-Gaussian tail-bound form,
$P\{|F-\E_P[F]|>r\}\le2\exp(-r^2/(2c_{T_2}K_F^2))$ for $r>0$.
\end{enumerate}
\end{theorem}
\begin{proof}
Follow \cite[Theorem 6.13]{davies2026potential}, extending from
$\Safe(c)$ rather than from the whole cube.
By \ref{d:Q-Lip}, Kirszbraun's theorem gives
extensions satisfying
\begin{equation}
\label{eq:alg-extended-coefficients}
 \|\sigma_{\mathrm{ext}}(x)-\sigma_{\mathrm{ext}}(y)\|_F
       \le K\lpnorm{x-y},\qquad
 \lpnorm{b_{\mathrm{ext}}(x)-b_{\mathrm{ext}}(y)}
       \le K\lpnorm{x-y}.
\end{equation}
Project the diffusion extension in Frobenius norm onto the closed convex set $\{Q\in\mathrm{Sym}_n:0\preceq Q\preceq c^{-1}\Id_n\}$.
This preserves its Lipschitz constant and its values on the good set.
Global strong existence follows, and pathwise uniqueness identifies the paths until exit. 
Apply \pref{thm:t2-transport-SDE} with $C=c^{-1}$ and the above $K$. 
As in \cite[Theorem 6.13]{davies2026potential}, $d_{W,1}\le d_{W,2}$ then implies $T_1$, and \pref{thm:t1-subgaussian} gives (ii).
\end{proof}

\begin{lemma}[Indicator-weighted centered exponential moments]
\label{lem:indicator-MGF}
Let $F$ be a real random variable with
$\E [e^{s(F-\E[F])}]\le e^{vs^2/2}$ for all $s\in\R$, for some $v>0$.
Then for any event $S$ and every $\lambda\in\R$,
\[
 \E\!\left[\mathbf1_S e^{\lambda(F-\E[F\mid S])}\right]
                              \le e^{2v\lambda^2}.
\]
If $\Pr(S)=0$, then the left-hand side is defined to be zero.
\end{lemma}
\begin{proof}
Let $q=\Pr(S) > 0$ since the $q=0$ case holds by convention.
Conditional Jensen's inequality and the MGF hypothesis at $s=-2\lambda$
give
\[
    q\,e^{-2\lambda(\E[F\mid S]-\E[F])}
    \le q\E[e^{-2\lambda(F - \E[F])}\mid S] = \E\!\left[\mathbf1_S e^{-2\lambda(F-\E[F])}\right]
    \le e^{2v\lambda^2}.
\]
Taking square roots and then applying Cauchy--Schwarz, we obtain
\begin{align*}
    \E\!\left[\mathbf1_S e^{\lambda(F-\E[F\mid S])}\right]
    &=
    e^{\lambda(\E[F]-\E[F\mid S])}
    \E\!\left[\mathbf1_S e^{\lambda(F-\E[F])}\right]\\
    &\le
    \bigl(q\,e^{-2\lambda(\E[F\mid S]-\E[F])}\bigr)^{1/2}
    \bigl(\E[e^{2\lambda(F-\E[F])}]\bigr)^{1/2}\\
    &\le e^{v\lambda^2}e^{v\lambda^2}
     =e^{2v\lambda^2},
\end{align*}
where the last step also uses the MGF hypothesis at $s=2\lambda$.
\end{proof}

\begin{corollary}[Sub-gaussian concentration of JE weights]
\label{cor:JE-subgaussian}
    Suppose $0^n\in\Safe(c)$ and assume \ref{d:drift-error}, \ref{d:Q-Lip}, and \ref{d:JE-lip}.
    Stop the original ASL-TAP-JE process at
$\tau_c$, and set $S_T=\{\tau_c>T\}$.
Then for every $0<T<\infty$ and $\lambda\in\R$,
\begin{equation}
\label{eq:alg-killed-JE-MGF}
 \E_B\!\left[\mathbf1_{S_T}
    e^{\lambda(\hat w_T-\E_B[\hat w_T\mid S_T])}\right]
      \le\exp\!\left(2\lambda^2c_{T_2}T^2L_\omega^2\right),
\end{equation}
where $c_{T_2}$ is as in \pref{thm:mSDE-path-subgaussian} and
$L_\omega$ as in \ref{d:JE-lip}.
The left-hand side is zero by convention if $\Pr_B(S_T)=0$.
\end{corollary}
\begin{proof}
    Follow the path-functional argument in
\cite[Corollary 6.14(d5)]{davies2026potential}, using
\ref{d:JE-lip} instead of integrating a gradient along
the whole cube, and extend $\omega$ to all of $\R^n$, for example by
\[
 \omega_{\mathrm{ext}}(x):=
       \inf_{m\in\Safe(c)}
             \{\omega(m)+L_\omega\lpnorm{x-m}\}.
\]
This agrees with $\omega$ on the good set and is $L_\omega$-Lipschitz.
The path functional $F(x):=\int_0^T\omega_{\mathrm{ext}}(x_t)\,dt$
is $TL_\omega$-Lipschitz and integrable for the global law in
\pref{thm:mSDE-path-subgaussian}, so \eqref{eq:alg-global-MGF} gives a sub-Gaussian MGF with
variance proxy $v=c_{T_2}T^2L_\omega^2$ for $F(\hat m^{\mathrm{ext}})$ on the extended SDE.
On $S_T$, the original and extended paths agree, hence
$F(\hat m^{\mathrm{ext}})=\hat w_T$ and their surviving-path means coincide, so
finally, apply \pref{lem:indicator-MGF} to remove paths outside of $S_T$ from the bound.
\end{proof}

\subsection{Algorithmic error analysis and proof of main theorem}

The following lemma will allow us to give a final error bound for the algorithm, given all the sources of error. 
\begin{lemma}[Jarzynski's equality with survival and rejection sampling]
\label{l:je-rs}
    Suppose that %
    $f_t$ and $\hat f_t$ are such that pathwise unique strong solutions to the following SDE's exist:
    \begin{align}
        & dx_t = f_t\, dt + dB_t&& x_0\sim  p_0\\
        \label{e:je-generic}
        &%
        d\hx_t = \hf_t\, dt+ dB_t 
        && %
        \hx_0\sim  p_0. %
    \end{align}
    Let $\calS\subeq \R^n$ and for a process $\ga$ adapted to $B_t$, define the stopping time $\tau(\ga):= \inf_{t\ge 0} \set{t}{\ga_t\nin \calS}$. 
    Let $\bP_T,\hat\bP_T$ be the path measures for $x_{[0,T]},\hx_{[0,T]}$ and $\bP_T^s, \hat\bP_T^s$ be the path measures for $x_{[0,T]}, \hx_{[0,T]}$ conditioned on survival, $S_T = \{\tau > T\}$.
    Let $ p_t, \hat p_t$ be the density of $x_t, \hat x_t$, respectively and $p_t^s$, $\hp_t^s$ be the densities for $x_t$, $\hat x_t$, restricted to surviving paths $S_t$. 
    Suppose that $\hw_t(\hx_{[0,T]})$ is defined so that 
    \[
\dd{\bP_T^s}{\hat{\bP}_T} (\ga) = \one_{S_T}(\ga) \fc{e^{\hat w_T(\ga)}}{\hat Z_T} \fc{R_2(\ga_T)}{Z_2}
    \]
    where $\hat Z_t = \E[e^{\hw_t}\mid S_t]$ 
    and so $\dd{p_T^s}{\hat p_T}(x) = \E[\one_{S_T}(\hx_{[0,T]}) e^{\hat w_T}/\hat Z_T \mid \hx_T=x] \fc{R_2(x)}{Z_2}$.
    Here, $\hw_t=\hw_t(\hx_{[0,t]})$ unless specified otherwise.
    Suppose with $L=\exp\pa{\fc{3}{\ep}\pa{\fc 12 \int_0^T \ep(t)^2dt+1}}$,
    \begin{enumerate}
    \item[(0)] (Survival probability) $\bP_T(S_T) \ge 1-\de$ and $\hat \bP_T(S_T) \ge 1-\hat \de$.
        \item (Drift error) For a stopping time $\tau'$ such that $\tau'\le \tau$, we have
        $\E_{x_{[0,T]}}\ba{\one_{\{t<\tau'\}}\ve{\hf_t-f_t}^2}\le \ep(t)^2$ and $\bP_T[\tau'> T]\le \de'$.
        \item (Exponential upper tail for log-weights)
        $\E\ba{\one_{S_T}(\hx_{[0,T]}) e^{\hw_T}\one_{e^{\hw_T}/\hat Z_T\ge C_1}}\le \ep_1 \hat Z_T$, for $\ep_1 = \fc{c_3\ep}{3 L}$.
        \item (Lower tail for log-weights)
        $\P\ba{\{\hx_{[0,T]}\in S_T\}\cap \bc{\fc{e^{\hw_T}}{\hat Z_T}\le c_3}}\le \ep_3$, for $\ep_3 = \fc{\ep}{3 L}$. 
    \end{enumerate}
    Then $\E\ba{
            \fc{e^{\hw_T}}{\hat Z_T} \dd{ p_T}{\rh_T}(\hx_T)\wedge \fc{2C_1L}{c_3}}\ge 1-\ep-\de$, and given query access to $R_2(x)$, we can sample from $ p_T$ with TV error $2\ep + %
            \de + \de'$ by $O\pf{C_1^2L^2\log\prc\ep}{c_3^2}$ simulations of \eqref{e:je-generic}, provided that $2\ep+\de<\rc2$, by choosing the parameters  $C_1'=C_2'=\fc{C_1L}{c_3}$ and using the estimate 
            $e^{\hw_T} R_2(\hx_T)$
            in \pref{a:ars}. 
Suppose moreover that %
$\tx_{[0,T]}$
is a random process and %
$\tw_T(x_{[0,T]})$ are weights such that 
\begin{enumerate}[resume]
    \item (Path error) \label{i:je-path-error}
    $\TV(\hat \bP_T^s, \td{\bP}_T^s) \le \ep_4$ for some $\td{\bP}_T^s$ such that $\dd{\td{\bP}_T^s}{\td{\bP}_T}\le 1$, where $\ep_4 = \fc{\ep}{384C_1^{\prime2}}$.\footnote{We define TV error for measures which may not be probability measures in the same way, $\TV(P,Q) = \sup_A |P(A)-Q(A)|$.}
    \item (Weight error) \label{i:je-ewt}
    $\td \bP_T^s \pa{|\tw_T - \hw_T|> \fc{\ep_4}2}\le \fc{\ep_4}2$.
    \item (Ratio error) \label{i:je-ert}
    $\td R_2(x)$ is a (randomized) estimate such that %
\allowdisplaybreaks
\begin{align*}
(a) &&\forall x,\quad  \P\pa{\fc{\td R_2(x_T)}{R_2(x_T)}\nin [0, e^{\ep_4/2}] \Big| x_T=x} &\le \fc{\ep_4}2\\
(b)&&    \P\pa{\fc{\td R_2(x_T)}{R_2(x_T)}\nin [e^{-\ep_4/2}, e^{\ep_4/2}]}&\le \ep
\end{align*}
\end{enumerate}
Then we can sample from $p_T$ with TV error $6\ep+\de + \de'+ 768{C_1'}^2\hat \de$ by $O\pf{C_1^2L^2\log\prc\ep}{c_3^2}$ simulations of \eqref{e:je-generic} using the estimate %
$e^{\tw_T(\td x_{ [0,T]})} \td R_2(\tx_T)$
            in \pref{a:ars}, provided that this error is $<\rc2$.
\end{lemma}
As the proof is very similar to \cite[Lemma 3.9]{davies2026potential}, but with restriction to survival events in the right places, we defer it to \pref{s:alg-error}. Note that we directly use the Radon-Nikodym derivative on path space rather than on the time-$T$ marginal. We also allow an earlier stopping time in bounding the drift error: $\tau'$ can be more complicated (as in Lemma \ref{l:m-error}), so we only require the probability of $\tau'\le T$ to be bounded under $\bP_T$, whereas we need the probability of $\tau\le T$ to be bounded under both $\bP_T$ and $\hat \bP_T$.

In particular, note that if all the parameters are constant in $n$ and $\de,\de',\hat\de = o_n(1)$, then there is a choice of $\ep=o_n(1)$ such that the total error is still $o_n(1)$.

We can now prove the main theorem. As many of the arguments carry over from \cite[\S8]{davies2026potential} in a straightforward manner, we focus on the necessary modifications.

\begin{prf}[Proof of \pref{t:main}]
We check the conditions of \Cref{l:je-rs} where the law $\bP_T$ of the ideal process $m_t$ is given by \eqref{e:HD}, and the law of the approximate process $\hat \bP_T$ is given by \eqref{e:ASL-TAP} with weights \eqref{e:JE}. The algorithm computes the Euler-Maruyama discretization with some step size $h$ to be determined:
\begin{align*}
    \ty_{t+h} &= \ty_t + h \td{\mg}_t + \sqrt{h}\xi_t, \quad \xi_t \sim \calN(0,I_n), & \ty_0&=0\\
    \tw_{t+h} &= \tw_t + \fc h2 \ba{\Tr(\hQ(\td m_t)) + \ve{\td m_t}^2}. 
\end{align*}
We first consider the case where $\td m_t$ is computed exactly from $\td y_t$ from \eqref{e:TAP}.
We take $\calS = \Safe(c)$, so that $\tau=\inf \set{t\ge 0}{m_t\nin \Safe(c)}$ and $S_T = \{\forall t\in [0,T], \, m_t \in \Safe(c)\}$. We take 
\[\tau' = \inf\set{t\ge 0}{((m_t,y_t),(\hm_t,\hy_t))\nin \calS^2(\rh\sqrt n , c\rh\sqrt n/2)},\] 
where the expanded pair safe set is defined with $\calS=\set{(m,y)}{m\in \Safe(c)}$. The Radon-Nikodym derivative is in the form specified by the lemma, by \pref{l:je-survival}.

\ppart{Part 0 (survival probability)} Note that $\bP_T(S_T^c) = o_n(1)$ directly from \ref{d:safe-set-stay-inside}, while $\bP_T(S_T^c) = o_n(1)$ follows from \pref{l:my-safe}.

\ppart{Part 1 (drift error)} This holds with $\ep(t) = O(1)$ by \Cref{l:m-error}. 

\ppart{Part 2 \& 3 (tails for log-weights)} 
This follows from the subgaussianity of Jarzynski weights, \pref{cor:JE-subgaussian} by \cite[Lemma 3.10]{davies2026potential}, applied to the measure conditional on survival. 

\ppart{Part 4 (path error)} 
Discretization error in \cite{davies2026potential} is analyzed in path space, using Lipschitzness of a uniquely defined $\hm(y)$ solving \eqref{e:TAP}. Instead, we can conduct the analysis on a lift of $y$-space $\hat y(\calS(c))$ based on homotopy equivalence in $\hat y(\calS(c))$; then $\hm$ is well-defined and Lipschitz on the lifted space. With $O\prc n$ step size, Brownian excursions are $o(1)$ with high probability, so when $\hm$ lies within the expanded safe set, all arguments can be done within a single branch.
We find that it suffices to take $O\pf{\ep_4^2}{n}$ step size with an appropriate constant to obtain TV error $\ep_4$.

\ppart{Part 5 (weight error)} Following the discretization arguments in \cite{davies2026potential} with the above remarks, to obtain the needed error bound, it suffices to take $h=O\pf{\ep_4^2}{n^2}$ with an appropriate constant.

\ppart{Part 6 (ratio error)} 
We need to check the estimation of $\hat Z_{\be A,y}$. 
With high probability over $A$, we can estimate the partition function on any small enough Hamming ball $B(\si, c(\be)n)\in \{\pm1\}^n$, by using the sampling guarantee for all temperatures $\be'\le \be$ (\ref{thm:localized-sampler}) along with simulated annealing, as in \cite[\S7.5]{davies2026potential}; this gives an underestimate with high probability not depending on $\si$.
When $\si = \sign(y_T)$ where $y_T$ is the SL tilt, for large enough $T=T(\be)$, with probability $1-e^{-\Om(n)}$, this wedge contains all but $e^{-\Om(n)}$ of the mass of the entire partition function, by \cite[Lemma 7.43]{davies2026potential}. 

\ppart{Putting everything together} Thus by \Cref{l:je-rs}, the output of Part 1 of \pref{alg:main} is a sample $\ty_T\sim \td  p_T$ with $\TV(\td p_T,  p_T)=o_n(1)$, in $\poly(n)$ time. 
By \pref{thm:localized-sampler}, %
with $1-o_n(1)$ probability over $y_T\sim  p_T$, we obtain a sample $\hat \si$ that is $o_n(1)$ in TV distance from $ \mu_{\be A, y_T}$ in polynomial time. 

Letting $\td \mu$ be the distribution of $\hat\si$ for tilt $\ty_T$, we have by the chain rule for TV distance that
\begin{align*}
    \TV(\td \mu, \mu)
    &\le 
    \TV(\td  p_T,  p_T) + 
    \E_{y_T\sim  p_T} \TV (\calL(\hat\si | y_T), \calL(\si | y_T) ) = o_n(1),
\end{align*}
as needed.

\ppart{Error in solving \eqref{e:TAP}} Robustness to approximating the solution to \eqref{e:TAP} is shown in \cite[\S D]{davies2026potential}. The same mirror descent analysis works with initialization that is within a ball around the true solution $m$ that is inside the expanded safe set, which is ensured under Brownian excursions at every step being at most $O(1)$.
\end{prf}

\section{Uniform TAP Hessian positivity over the stochastic localization process}\label{sec:uniform-tap-convexity}
We will use a ``hybrid'' argument that combines the local strong convexity of the TAP Hessian $\hat{Q}(\cdot)$ for $0 < \beta < 1$ in small balls around the AMP iterates $\{\bar{m}_t\}_{t \ge 0}$ due to a result of Celentano \cite{celentano2024sudakov} and the $o(\sqrt{n})$-$\ell_2$ bound between the AMP sequence $\{\bar{m}_t\}_{t \ge 0}$ and SL process $\{m_t\}_{t\ge 0}$ implied by the result of El-Alaoui, Montanari and Sellke \cite{el2022sampling} with a covering argument. This will immediately imply desideratum \ref{d:safe-set-stay-inside}, namely:
\stepcounter{desideratum}
\begin{desideratum}
    For every $0 \le T(\beta) < \infty$, there exist $\rho(\beta, T), c_{\TAP}(\beta,T) > 0$ such that
    \[ \Pr_{A,B_t}\left\{\forall t \in [0,T] \;.\; B_{\square}(m_t, \rho\sqrt{n}) \subseteq \calS_A(c_{\TAP})\right\} \ge 1 - o_n(1)\,. \]
\end{desideratum}

\paragraph{AMS Hessian, TAP Hessian and AMP iterates} Let $D(m_t) := \mathsf{diag}\left(\frac{1}{1-m^2_t}\right)$. The Hessian used in \cite{el2022sampling} is given as
\[
    \bar{Q}(m_t)^{-1} := \beta^2(1-q^*(t))\Id_n - \beta A + D(m_t)\,.
\]
Recall the Hessian of the SL-tilted planted-SK model used in \cite[\S 4]{davies2026potential} is 
\[
    \nabla^2\calF_{\TAP}(m_t,\cdot) := \beta^2\left(1-\frac{\norm{m_t}^2_2}{n}\right)\Id_n -\beta A + D(m_t) - \frac{\beta^2}{n}x_0x_0^\sT -\frac{2\beta^2}{n}m_tm_t^\sT\,.
\]
The AMP iterates, run for $k \ge k_0(\eps,\beta,t)$ iterations at \emph{every} $t \in [0,T(\beta)]$, start at $\bar{m}_{-1} = 0^n$ and $y_t = tx_0 + g_0$ where $g_0 \sim \calN(0,\Id_n)$, and follow the standard updates
\begin{align*}
    \bar{m}_j &= \tanh(z_j)\,, \\
    z_{j+1} &= \beta \left(A + \frac{\beta}{n}x_0x_0^\sT\right) + y_t - \beta^2\left(1 - \frac{\norm{\bar{m}_j}^2_2}{n}\right)\bar{m}_{j-1}\,,
\end{align*}
run for $j \in [0,k]$ iterations. The result of Celentano \cite{celentano2024sudakov} considers a comparison with a linear functional of an auxiliary Gaussian field, the definition of which requires considering an approximate isometry $\mathbf T$ from the  subspace spanned by the AMP magnetization estimates $\{\bar{m}_0,\dots,\bar{m}_{k-1}\}$ to that of the vectors that come from the ``pure'' SK component
\[
    g_j = \beta A \bar{m}_{j-1} - \beta^2\left(1 - \frac{\norm{\bar{m}_{j-1}}^2_2}{n}\right)\bar{m}_{j-2}\,,
\]
which remove the contributions from the tilt $y_t$ and the plant $J = \mathbf{1}^n(\mathbf{1}^n)^\sT$.

\subsection{``Hybrid'' local strong convexity} For our use case, the only non-trivial part of this proof is to check that the Schur complement of a $2\times2$-matrix induced by a specific low-dimensional objective induced by the Hessian $\hat{Q}(\cdot)$ around test-points in a compact set $K$ corresponding to the AMP iterates is negative. Once the correct reduction to the low-dimensional optimization problem is setup in our ``hybrid'' model, no change is needed -- the proof is exactly that of \cite[Proof of Theorem 2]{celentano2024sudakov} with just one syntactic change.

To briefly recap the proof, we introduce some notation which closely follows the definitions in \cite[\S 5]{davies2026potential} and allows a modular invocation of \cite[Proof of Theorem 2]{celentano2024sudakov}. Let $(m^*, q^*)$ solve the fixed-point equations in \cite[Theorem 5.1]{davies2026potential}. Let $G \sim \calN(0,\beta^2 q^*)$ and $G_0 \sim \calN(0,1)$ be independent, and denote various parts of the functionals in the fixed-point equations as
\allowdisplaybreaks
\begin{align*}
    t_\infty &:= \beta^2q^* + t\,,\\
    M &:= \tanh(t_\infty + G + \sqrt{t}G_0) = \tanh(Y^*)\,,\\
    b^* &:= \E[M^3] = \E[M^4]\mbox{\footnotemark}\,, \\
    \Delta &:= \E\left[(1-M^2)^2\right] = \E\left[\sech^4(Y^*)\right] = 1-2q^* + b^*\,, \\
    r &:= 1 - \frac{b^*}{q^*}\,, \\
    s &:= 1 -\beta^2\Delta = 1-\beta^2\E[\sech^4(Y^*)]\,.
\end{align*}
\footnotetext{\,This can be proved via a simple extension of the proof of \cite[Lemma 2.1]{li2026overlap} -- see \pref{lem:nishimori-third}.}\hspace{-2mm}
In the notational choices above, the dependencies of various constants on $t$ have been suppressed for convenience\footnote{\,The quantities $m^*$, $q^*$, $Y^*$ and $b^*$ all dependent on $t$.}. The astute reader may observe that the quantity $s$ is a ``self-stability'' term that shows up in the denominator of various Taylor-expanded terms in the cavity estimates conducted in \cite[\S 5.4]{davies2026potential}. For $0 < \beta < 1$ and $0<t\le T(\beta)<\infty$, it is straightforward to see that $s > 0$ and $0<b^*<q^*<1$. The latter immediately implies that $\Delta, r > 0$.

We now state the result, which shows that after sufficiently many runs of the AMP algorithm initialized with tilt $y_t$, the Hessian of the TAP free energy for the SL-tilted planted SK model is PSD in a $\ell_2$-ball around of radius $\Omega_{\beta,t}(\sqrt{n})$ around the final AMP iterate. 
\begin{theorem}[``Hybrid'' local strong convexity {\cite[Syntactic modification of proof of Theorem 2]{celentano2024sudakov}}]\label{thm:hybrid-local-convexity}
    For every $\beta, t > 0$, there exist $c_0(\beta,t), c_1(\beta,t) > 0$, such that, for all $0 < \eps < c_1(\beta,t)$ and $k \ge k_0(\eps,\beta,t)$,
    \[
        \Pr_{A, g_0}\left\{\inf_{\sigma \in B_\square(\bar{m}_{k-1},\eps\sqrt{n})} \nabla^2\calF_{\TAP}(\sigma,\cdot) \succeq c_0(\beta,t)\right\} \ge 1 - o_n(1)\,.
    \]
\end{theorem}
\begin{prf}
    The proof is exactly analogous to \cite[Theorem 2]{celentano2024sudakov} with the small modification that it tests the quadratic forms of the TAP Hessian considered in our previous work \cite{davies2026potential}\footnote{\,This is equivalent to the Hessian used by Fan, Mei and Montanari \cite{fan2021tap}.} evaluated at the AMP iterates given above. Using gauge symmetry, we fix $x_0=\mathbf{1}^n$. Throughout the proof, the AMP law is the one corresponding to the iterates which come from \cite[Algorithm 1, (1)--(5)]{el2022sampling} and the Hessian is $\nabla^2\calF_{\TAP}(\sigma,\cdot)$\footnote{\,The original result of Celentano \cite[Theorems 1--2]{celentano2024sudakov} tests AMP iterates designed by Fan et al \cite{fan2021tap} on the TAP Hessian $\nabla^2\calF_\TAP(\sigma,\cdot)$ or the AMP iterates designed by El-Alaoui et al \cite{el2022sampling} on $\bar{Q}^{-1}(\cdot)$. We require testing the AMP iterates of Alaoui et al \cite{el2022sampling} on the TAP Hessian $\nabla^2\calF_{\TAP}(\sigma,\cdot)$ -- this leads to the name ``hybrid'' local convexity.}. For completeness, we briefly recreate the proof of \cite[Theorem 2, Pg 11--18] {celentano2024sudakov} by analyzing the objective $-v^\sT\nabla^2\calF_{\TAP}(u,\cdot)v$ on the AMP transcript $(\bar{m}_0,\dots,\bar{m}_{k-1},g_0,g_1,\dots,g_k)$ and the limiting joint law $\calS^{(2)}_1(\eps) := (V,U,M,G,G_0,\Xi)$ where $\Xi \sim \calN(0,1)$ (independently of $G$ and $G_0$), $V$ and $U$ are the empirical limiting empirical laws of the test vector $v \in B_\square(r,x)$ and the evaluation vector $u$ respectively, and the joint law satisfies the constraints in the definition of the set given in \cite[Pg 14]{celentano2024sudakov}.
    \ppart{Reduction to the scalar objective} For every $v \in \calS^{n-1}(1)$, expanding the Hessian yields
    \allowdisplaybreaks
    \begin{align*}
        -\frac{1}{n}v^{\sT}\nabla^2\calF_{\TAP}(\sigma,\cdot)v
        &= \beta v^{\sT}A v
           +\frac{\beta^2}{n}\langle\mathbf{1},v\rangle^2
           -\sum_{i=1}^n\frac{v_i^2}{1-u_i^2}
           -\lambda^2(1-Q(u))
           +\frac{2\lambda^2}{n}\langle u,v\rangle^2\,,
    \end{align*}
    which is just $f_{\text{FMM}}(v,u)$ as defined on \cite[Pg.\,11]{celentano2024sudakov}. We now apply the post-AMP comparison to the AMS transcript $(\bar{m}_0,\dots,\bar{m}_{k-1})$ and $f_{\text{FMM}}(v,u)$, followed by the empirical-law and fixed-point reductions \cite[Corollary\,1, Lemma\,1 and \S 4.2--4.3]{celentano2024sudakov}. These give
    \begin{equation}\label{eq:hybproof-reduction}
        \limsup_{\epsilon\downarrow0}\limsup_{k\to\infty}
        \mathop{\mathrm{p}\text{-}\limsup}_{n\to\infty}
        T_{n,k}(\epsilon)
        \leq
        \limsup_{\eps\downarrow0}
        \sup_{\mu\in\calS^{(2)}_1(\eps)}\mathcal{F}^{(2)}_0(\mu)\,,
    \end{equation}
    where $T_{n,k}(\eps) := \sup_{u \in B_\square(\bar{m}_{k-1},\eps\sqrt{n}), \norm{v}_2=1} -v^\sT\hat{Q}(u)v$ and $\calF^{(2)}_0(\mu)$ is as defined in \cite[(24)]{celentano2024sudakov}. Carrying the computations in \cite[\S 4.4]{celentano2024sudakov} further by rewriting the Lagrange multiplier to combine with the AMS state evolution and FMM TAP Hesisan gives
    \[
         \sup_{\mu\in\calS^{(2)}_1(\eps)}\mathcal{F}^{(2)}_0(\mu) \le \sup_{\rho,u \in [-1,1]}\min_{(\alpha_\rho,\alpha_u) \in K} \mathsf{L}(\rho,u;\alpha_\rho,\alpha_u,\alpha_v) + C_K\eps\,,
    \]
    for some constant $C_K > 0$. Expanding the Lagrangian as in after \cite[(27)]{celentano2024sudakov} yields
    \[
        \sup_{\mu\in\calS^{(2)}_1(\eps)}\mathcal{F}^{(2)}_0(\mu) \le \sup_{\rho,u \in [-1,1]} (\rho,u)^\sT S_0(\alpha_v) (\rho,u) + C_K\eps\,,
    \]
    where $S_0(\alpha_v) = A_{11,0}(\alpha_v) - A_{12}(\alpha_v)A_{22}(\alpha_v)^{-1}A_{12}(\alpha_v)^\sT$ and these matrices are explicitly derived next\footnote{\,This is done for completeness, though the calculations in \cite[\S 4.4]{celentano2024sudakov} also cover our case almost entirely.}.

\ppart{The hybrid Schur-complement calculation} We compute the blocks at $\alpha_v=0$. We first apply Stein's lemma over $G$ and keep $G_0$ fixed. Since the variance is $\beta^2q^*$, and we have $\partial_GM=1-M^2$, together with the Nishimori moment identities ($m^* = q^*$ and \pref{lem:nishimori-third}), this gives
    \allowdisplaybreaks
    \begin{align*}
        \E[G M(1-M^2)] &=\beta^2q^*\,\E[(1-M^2)(1-3M^2)] =\beta^2q^*\,d_0\,,\\
        \E[G(1-M^2)] &=-2\beta^2q\,\E[M(1-M^2)] =-2\beta^2q^*(q^*-b^*)\,,\\
        \E[G^2(1-M^2)] &=\beta^2q^*\,\E[1-M^2-2GM(1-M^2)] =\beta^2q^*(1-q^*)-2\beta^4(q^*)^2d_0\,,
    \end{align*}
    where $d_0 = 1-4q^* + 3b^*$.
    Substituting $h_0(M)=1-M^2$ yields
    \allowdisplaybreaks
    \begin{align*}
        A_{11,0}(0)
        &=\begin{pmatrix}
            2\beta^2q^*-2\beta^4q^*d_0&0\\
            0&\beta^2
          \end{pmatrix}\,,\\
        A_{12}(0)
        &=\frac12\begin{pmatrix}
            -1+\beta^2d_0&-2\beta^2\sqrt q^*\,(q^*-b^*)\\
            0&-1
          \end{pmatrix}\,,\\
        A_{22}(0)
        &=\frac14\begin{pmatrix}
            r&\sqrt q^*\,r\\
            \sqrt q^*\,r&1-q^*
          \end{pmatrix}\,.
    \end{align*}
    Only the first block depends explicitly on the Hessian choice, while the other blocks use the AMS scalar law. %
    Using $1-q^*=\Delta+rq^*$ gives
    \allowdisplaybreaks
    \begin{align*}
        A_{22}(0)^{-1}
        &=4\begin{pmatrix}
            r^{-1}+q/\Delta&-\sqrt q^*/\Delta\\
            -\sqrt q^*/\Delta&1/\Delta
          \end{pmatrix}\,.
    \end{align*}
    The $(2,2)$ and $(1,2)$ entries of the Schur complement are therefore
    \allowdisplaybreaks
    \begin{align*}
        S_0(0)_{22}&=\beta^2-\Delta^{-1}=-\frac{s}{\Delta}\,,\\
        S_0(0)_{12}&=\frac{\sqrt q^*\,s}{\Delta}\,.
    \end{align*}
    For the remaining scalar Schur complement, evaluate the quadratic form at $(1,\sqrt q)^{\top}$. Substituting the displayed inverse and using $d_0=\Delta-2rq^*$ and $s=1-\beta^2\Delta$ yields
    \allowdisplaybreaks
    \begin{align*}
        S_0(0)_{11}+\frac{sq^*}{\Delta} &=3\beta^2q^*-2\beta^4q^*d_0 -\frac{(s+2\beta^2rq^*)^2}{r} -q\beta^4\Delta =-q\beta^2s-\frac{s^2}{r}\,.
    \end{align*}
    Consequently, for $z=(\rho,a)^{\top}$,
    \begin{equation}\label{eq:hybproof-negative-squares}
        z^{\top}S_0(0)z =-\frac{s}{\Delta}(a-\sqrt q\,\rho)^2 -\left(q\lambda^2s+\frac{s^2}{\kappa}\right)\rho^2\,.
    \end{equation}
    Both coefficients are strictly positive, and the two squares vanish simultaneously only at $z=0$, so $S_0(0)\prec0$. At this point, a simple continuity argument implies that $S_0(\alpha_v) \prec 0$ for small enough $\alpha_v > 0$, and a choice of $\alpha_v \le \min\{s/2,1/2\}$ yields a strict upper bound of $0$ on the objective function. \qedhere    
\end{prf}

 \subsection{Covering the entire SL path via $\ell_2$-balls centered at posterior means}\label{sec:uniform-strong-convexity-proof} The proof of \pref{thm:hybrid-local-convexity} implies that for every fixed $t \in [0,T(\beta)]$, after taking sufficiently many AMP iterates and outputting $\bar{m}_{k-1}$, the TAP Hessian $\nabla^2 \calF_\TAP(u)$ is PSD at any point $u \in B_\square(\bar{m}_{k-1},\Omega_{\beta,t}(\sqrt{n}))$. However, this does not yield \emph{uniform} control over the entire SL trajectory $\{m_t\}_{t \in [0,T(\beta)]}$ which is critical to proving the requisite desideratum. We now ``boost'' the fixed-time, local strong convexity in reasonable-sized $\ell_2$-balls around AMP iterates implied by \pref{thm:hybrid-local-convexity} to uniform local convexity along the SL paths by using these balls to cover a significant measure of the SL paths. Doing this requires using the uniform control over the fluctuations of the SL process inside small intervals supplied by \cite[Lemma 4.9]{el2022sampling} along with the $\ell_2$-bound
\[
    \norm{m_t - \bar{m}_{k-1}} \le \eps\sqrt{n}\,,
\] 
for a large-enough AMP iterate corresponding to an initialization at every $t \in [0,T]$ implied by \cite[Proposition 4.8]{el2022sampling}.   

\begin{theorem}[Uniform local strong convexity along $\{m_t\}_{0\le t \le T(\beta)}$]\label{thm:uniform-local-strong-convexity}
    Let $\{m_t\}_{t \ge 0}$ evolve according to the HD process, and $S_A(c_{\mathrm{TAP}})$ be as in \pref{def:good-points}. Then, 
    \[
         \P_{A,B}\!\left\{\forall t\in[0,T],\quad B_\square(m_t,\rho\sqrt n)\subseteq\mathcal S_A(c_{\mathrm{TAP}})\right\}=1-o_n(1)\,.
    \]
\end{theorem}
\begin{prf}
We first show there are small $\ell_2$-balls around the true means $\{m_t\}_{t\ge 0}$ in which $\hat{Q}(m_t) \succeq c(\beta,t) \Id_n$ at every \emph{fixed} $t$. At $t>0$, this uses contiguity between the planted and original model for $\beta < 1$ \cite[\S 2]{el2022sampling} as well as \cite[Proposition 4.8]{el2022sampling} along with a union bound. At $t=0$, we prove this directly by conditioning on the event that $\opnorm{A} \le 2 + \frac{(1-\beta)^2}{8\beta}$ and the radius of the ball $\eps(\beta,0) \le \frac{(1-\beta)^2}{24\beta^2}$. We then use \cite[Lemma 4.9]{el2022sampling} through Doob's inequality to control the fluctuations of $m_t$ in a small time-interval $(t-\delta_t,t+\delta_t) \subset [0,T]$ and finally construct a ``tube'' composing balls for any finite sub-cover of the interval $[0,T]$ conditioned on the fixed-time convexity estimates and the small-time stability estimates.
\ppart{Fixed-time convexity balls} For every fixed $t>0$, \cite[Proposition 4.8]{el2022sampling} immediately implies that, in the planted model,
\[
    \P_{A, B_t}\left\{\norm{m_t - \hat{m}^{k-1}(A,y_t)}_2 \le \eps'(\beta,t)\right\} \ge 1 - o_n(1)\,,
\]
for any $\eps'(\beta,t) > 0$ and $\hat{m}^{k-1}(A,y_t)$ denoting the $k$-th output of the AMP iteration of AMS initialized with tilt $y_t$. Invoking \pref{thm:hybrid-local-convexity} we obtain parameters $\eps''(\beta,t)$ and $k'(\beta,t)$ so that 
\[
    \P_{A,g_0}\left\{\forall u \in B_\square(\bar{m}_{k'(\beta,t)-1}, \eps''(\beta,t)\sqrt{n}),\,\nabla^2\calF_{\TAP}(u,\cdot) \ge \eps''(\beta,t)\right\} \ge 1 - o_n(1)\,.
\]
Choose $k(\beta,t) \ge \max\{k,k'(\beta,t)\}$ and set $\bar{m}_{k(\beta,t)-1} = \hat{m}^{k-1}(A,y_t)$, along with a choice of radius $\eps(\beta,t) = \frac{1}{2}\min\left\{\eps'(\beta,t), \eps''(\beta,t)\right\}$. Then, a union bound yields 
\[
    \Pr_{A,B_t}\left\{\forall u \in B_\square(\bar{m}_{k(\beta,t)-1}, \eps(\beta,t)\sqrt{n}),\,\nabla^2\calF_{\TAP}(u,\cdot) \succeq c(\beta,t)\Id_n\right\} \ge 1 - o_n(1) \,,
\]
where we invoke contiguity to transfer back to the original SK law from the planted model. \\
To obtain the same statement at $t=0$, use $m_0=0$ and choose $\delta_A = \frac{(1-\beta)^2}{8\beta}$ and $\eps(\beta,0) = \frac{(1-\beta)^2}{24\beta^2}$ so that $\P_A\left\{\norm{A}_{\mathrm{op}}\leq 2+ \frac{(1-\beta)^2}{8\beta}\right\}\ge 1-o_n(1)$. Conditioning on this event, for every $u\in B_\square(0,\eps(\beta,0)\sqrt n)$, the bounds $D(u)\succeq\Id_n$ and $uu^{\top}/n\preceq \frac{\norm{u}^2_2}{n}\Id_n$ yield
\allowdisplaybreaks
\begin{align*}
    \nabla^2\calF_{\TAP}(u,\cdot) &= \beta^2\left(1-\frac{\norm{u}^2_2}{n}\right)\Id_n -\beta A + D(u) -\frac{2\beta^2}{n}uu^\sT \succeq \left(1+\beta^2-\beta\norm{A}_{\mathrm{op}} -\frac{3\beta^2}{n}\norm{u}^2_2\right)\Id_n\\
    &\succeq \left((1-\beta)^2 - \frac{(1-\beta)^2}{8}-\frac{(1-\beta)^2}{8}\right)\Id_n \succeq\frac{(1-\beta)^2}{2}\Id_n\,.
\end{align*}
Set $c(\beta,0)=(1-\beta)^2/2$. Thus, for every fixed $t\in[0,T]$, there is a deterministic tuple $(c(\beta,t),\eps(\beta,t))$ for which $\P_{A,B}\left\{\forall u \in B_\square(\bar{m}_{k(\beta,t)-1},\eps(\beta,t)\sqrt{n}),\,\nabla^2\calF_{\TAP}(u,\cdot) \succeq c(\beta,t)\Id_n\right\}\ge1-o_n(1)$, and we denote these events $\calE_{n,t} := \left\{\forall u \in B_\square(\bar{m}_{k(\beta,t)-1},\eps(\beta,t)\sqrt{n}),\,\nabla^2\calF_{\TAP}(u,\cdot) \succeq c(\beta,t)\Id_n\right\}$.
\ppart{Fluctuations across small time intervals} %
For each $t\in[0,T]$, choose a deterministic $0<\delta_t\leq \frac{\left(\eps(\beta,t)(1-\beta)\right)^2}{128}$ and set
\begin{align*}
    a_t&:=\max\{0,t-\delta_t\}\,,\qquad b_t:=\min\{T,t+\delta_t\}\,,
\end{align*}
and $J_t:=(t-\delta_t,t+\delta_t)\cap[0,T]\,$. Note that $J_t$ is an open neighborhood of $t$ relative to $[0,T]$. By the proof of \cite[Lemma 4.9]{el2022sampling} and \cite[(4.24)]{el2022sampling},
\allowdisplaybreaks
\begin{align*}
    &\P_{A,B_t}\left\{\sup_{r\in[a_t,b_t]}
       \frac{\norm{m_r-m_{a_t}}_2^2}{n} \ge q(b_t)-q(a_t) + \frac{\eps(\beta,t)^2}{64}\right\} \le o_n(1)\,,\\
    &q(b_t)-q(a_t)
    \leq \frac{(b_t-a_t)}{(1-\beta)^2}
     \leq\frac{\delta_t}{(1-\beta)^2}
     \leq\frac{\eps(\beta,t)^2}{64}\,.
\end{align*}
In particular, the event $\mathcal F_{n,t}:=\left\{\sup_{r\in[a_t,b_t]}\frac{\norm{m_r-m_{a_t}}_2^2}{n}\leq \frac{\eps(\beta,t)^2}{32}\right\}$ has probability $1-o_n(1)$. On this event, for every $s\in J_t$,
\allowdisplaybreaks
\begin{align*}
    \norm{m_s-m_t}_2
    &\leq\norm{m_s-m_{a_t}}_2+\norm{m_t-m_{a_t}}_2\\
    &\leq2\sqrt{nZ_{n,t}}
     \leq\frac{\eps(\beta,t)}{2}\sqrt n\,.
\end{align*}
\ppart{Finite sub-cover and a uniform bound} The open sets $\{J_t:t\in[0,T]\}$ cover $[0,T]$ and compactness implies the existence of a finite subcover $J_{t_1},\ldots,J_{t_m}$. Denote the uniform radius and convexity parameters
\allowdisplaybreaks
\begin{align*}
    \eps_M&:=\frac{1}{2}\min_{j \in [M]}\eps(\beta,t_j)>0\,, \\
    c_M&:=\min_{j \in [M]}c(\beta,t_j)>0\,.
\end{align*}
Taking a union bound over the $M$ events $\bigcap_{j=1}^M\calF_{n,t_j}\cap\calE_{n,t_j}$ gives
\allowdisplaybreaks
\begin{align*}
    \P_{A,B}\left\{\mathcal \bigcap_{j=1}^M\calF_{n,t_j}\cap\calE_{n,t_j}\right\} \ge 1-o_n(1)\,.
\end{align*}
Note that each of the $j$ events concern the same disorder and SL path. On this event, fix any $s\in[0,T]$ and choose $j$ with
$s\in J_{t_j}$. For every $u\in B_\square(m_s,\eps_M\sqrt n)$,
\allowdisplaybreaks
\begin{align*}
    \norm{u-m_{t_j}}_2 &\leq\norm{u-m_s}_2+\norm{m_s-m_{t_j}}_2\\
    &\leq\left(\eps_M+\frac{\eps(\beta,t_j)}{2}\right)\sqrt n \leq \eps(\beta,t_j)\sqrt n\,.
\end{align*}
This implies the containment $B_\square(m_s,\eps_M\sqrt n) \subseteq B_\square(m_{t_j},\eps(\beta,t_j)\sqrt n)$, and $\mathcal E_{n,t_j}$ gives $\nabla^2\calF_{\TAP}(u,\cdot)\succeq c_M\Id_n$. Therefore,
\allowdisplaybreaks
\begin{align*}
    \P_{A,B}\!\left\{ \forall s\in[0,T],\ \forall u\in B_\square(m_s,\eps_M\sqrt n), \ \nabla^2\calF_{\TAP}(u,\cdot)\succeq c_M\Id_n\right\}&\ge 1-o_n(1)\,.
\end{align*}
If $T=0$, the same conclusion follows from the first part of the proof with $\eps_M:=\eps(\beta,0)/2$ and $c_M:=c(\beta,0)$.
\ppart{L\"owner order vis-a-vis $D(u)$} Intersect the preceding event $\bigcap_{j=1}^M\calF_{n,t_j}\cap\calE_{n,t_j}$ with $\{\norm{A}_{\mathrm{op}}\leq3\}$, which also has probability $1-o_n(1)$. Since $\norm{u}^2_2/n \le 1$ for any $u \in (-1,1)^n$, this gives the following L\"owner order lower bound
\allowdisplaybreaks
\begin{align*}
    \nabla^2\calF_{\TAP}(u,\cdot) &= \beta^2(1-\|u\|^2_2/n)\Id_n - \beta A + D(u) - \frac{2\beta^2}{n}uu^\sT \succeq D(u)-\beta(2\beta + 3)\Id_n\,.
\end{align*}
Throughout the tube we also have $\nabla^2\calF_{\TAP}(u,\cdot)\succeq c_M\Id_n$, so combining with the previous inequality gives
\allowdisplaybreaks
\begin{align*}
    D(u) &\preceq \nabla^2\calF_{\TAP}(u,\cdot) +\beta(2\beta + 3)\Id_n
     \preceq\left(1+\frac{\beta(2\beta + 3)}{c_M}\right)\nabla^2\calF_{\TAP}(u,\cdot)\,.
\end{align*}
Choose
\allowdisplaybreaks
\begin{align*}
    c_{\mathrm{TAP}}
    &=\frac{c_M}{c_M+\beta(2\beta + 3)}>0\,,
\end{align*}
and note the preceding bound immediately yields $\nabla^2\calF_{\TAP}(u,\cdot)\succeq c_{\mathrm{TAP}}D(u)$ throughout the tube. Since $D(u)$ is PD, it follows that
\allowdisplaybreaks
\begin{align*}
    \P_{A,B_t}\!\left\{\forall t\in[0,T],\, B_\square(m_t,\rho\sqrt n)\subseteq\mathcal S_A(c_{\mathrm{TAP}})\right\}
    &\ge 1-o_n(1)\,. \qedhere
\end{align*}
\end{prf}

\subsection{L\"owner order sandwich on the safe set}\label{sec:lowner-sandwich-safe-set} The fact that $D(m) \succ 0$ plus the fact that $\hat{Q}(m_t) \succeq c(\beta,t)\Id_n$ allows one to conclude the PSD-ness in the primal space (obtained by conjugating by $D(m_t)$).

\begin{lemma}[Desideratum 2(b) -- L\"owner sandwich for dual coordinate chart]\label{lem:tap-convexity}
    Let $m \in \calS_A(c_{\TAP})$ and condition on the ``master'' event $\calE_A$. Then, there exist $0 < c_{\TAP}(\beta,T), C_{\TAP}(\beta) < \infty$, such that
    \[
        C_{\TAP}(\beta)^{-1}D^{-1}(m) \preceq \hat{Q}(m) \preceq c_{\TAP}(\beta,T)^{-1}D^{-1}(m). 
    \]
\end{lemma}
\begin{prf}
    Conditioning again on $\opnorm{A} \le $ 3, and using the fact that $D(m) \succ \Id_n$ for any $m \in (-1,1)^n$, yields
    \allowdisplaybreaks
    \begin{align*}
        \nabla^2 \calF_{\TAP}(m,\cdot) &= \beta^2(1-\norm{m}^2_2/n)\Id_n - \beta A + D(m) - \frac{2\beta^2}{n}mm^\sT \\
        &\preceq D(m) + \beta\left(\opnorm{A} + \beta^2\right)\Id_n \\
        &\preceq(1+3\beta+\beta^2)D(m)\,.
    \end{align*}
    By \pref{thm:uniform-local-strong-convexity}, $\nabla^2 \calF_{\TAP}(m,\cdot) \succeq c_{\TAP}D(m)$ for any $m \in \calS_A(c_{\TAP})$, conditioned on the desirable event $\calE_A$. Inverting the Hessian and combining the PD-ness with the preceding upper bound yields
    \[
        \frac{1}{1+3\beta + \beta^2}D^{-1}(m) \preceq \hat{Q}(m) \preceq \frac{1}{c_{\TAP}}D^{-1}(m)\, ,
    \]
    for any $m \in \calS_A(c_{\TAP})$. \qedhere
\end{prf}

\section{Resolvent analysis: Regularization, uniform concentration and interpolation}\label{sec:resolvent-analysis}
We now repeat the structure of the argument in \cite[\S B]{davies2026potential} to develop control over the diagonal sub-algebra of $\hat{Q}^2(m_t)$ for all points $m_t \in S_A(c)$. While the idea of repeating the same argument restricted to the safe set $S_A(c)$ (\pref{def:good-points}) is natural, it causes technical problems in the chaining argument used for uniform concentration of $E_{\calD_n}[\hat{Q}^2(m_t)]$ around its expected value (\pref{sec:block-mixed-lipschitzness}) -- this happens because the supremum is taken over a set of points $m_t$ that now depend on $A$ through $S_A(c)$ and any sort of conditioning destroys the structure of the underlying Gaussian randomness. One solution to this is to define a regularization $\phi(\hat{Q}^{-2}(m_t))$ that agrees with $\hat{Q}^2(m_t)$ for every $m_t \in S_A(c)$ and is Lipschitz, bounded and continuously twice-differentiable for every $m_t \in \{-1,1\}^n \setminus S_A(c)$ (\S \ref{sec:block-regularization}). As it turns out, our definition of this regularization $\phi$ as a function of $\hat{Q}^{-2}(m_t)$ merges well with a representation for $\hat{Q}^{-1}(m_t)$ in block-form with each block being an affine function of $A$ --- this representation is the key to a comparatively simple free interpolation argument comparing $E_{\calD_n}[\ph(\hat{Q}^{-2}(m_t))]$ with the ideal free limit (\pref{sec:option-c-free-comparison}). After obtaining uniform concentration and the control of the diagonal of the expected squared resolvent, one combines these estimates via a triangle inequality and restricts to the safe set $m \in S_A(c)$ at the very end to obtain a valid spectral cut-off $\gamma_Y(\beta)$ with which to invoke $\phi$ (\pref{sec:uniform-disagonal-estimate}). 

\subsection{Bounded extension of matrix inverse}
\label{sec:block-regularization}
In \cite[\S B]{davies2026potential}, the authors established a deterministic diagonal approximation to $E_{\mathcal D_n}[D(a\Id_n-\beta A+D)^{-2}D]$, uniformly over $a\in[0,1]$ and $D\in\mathcal D_n([1,\infty))$, with high probability over $A$.
To do this, we required uniform L\"owner lower bounds on $(a\Id_n-\beta A+D)$.

In our current setting, with $1/2 < \beta < 1$, it is possible for $(a\Id_n-\beta A + D)$ to become singular.
Therefore, we replace the weighted squared inverse $D(\cdot)^{-2}D$ with a regularized version $\phi(D^{-1}(\cdot)^2D^{-1})$, where $\phi$ is a bounded $C^2$ extension of the reciprocal $(\cdot)^{-1}$ defined, for fixed $\gamma_Y>0$, as
\begin{equation}
\label{eq:rational-reciprocal-cutoff}
    \phi(y):=
    \left[y+\gamma_Y
       \left(\frac{(\gamma_Y-y)_+}{\gamma_Y+y}\right)^3
    \right]^{-1},\qquad y>-\gamma_Y.
\end{equation}
This scalar function agrees with $y^{-1}$ when $y \ge \gamma_Y$ and has bounded and continuous first and second derivatives on $[0,\infty)$\footnote{\,This particular regularization $\phi$ bears technical and conceptual parallel to the regularization used in \cite[\S 2]{jekel2024pha} to ``smooth'' out the first derivative of the primal Parisi PDE in the event that a PHA iterate's coordinate leaves the solid cube. Both regularizations are of the form $a^{-1} \rightsquigarrow (a + y(a))^{-1}$ where $a = \Phi''(x_\gamma)$ in \cite[\S 2]{jekel2024pha} and $a \in \mathsf{Spec}\left((\Id_n + D^{-1}(a\Id_n - \beta A))(\Id_n+(a\Id_n-\beta A)D^{-1})\right)$ in \pref{sec:block-regularization}. The regularization of the primal Parisi PDE regularizes the (generalized) TAP free energy at the level of coordinate curvature, whereas the current regularization $\phi$ regularizes the Hessian of the TAP free energy at the level of spectral invertibility.}.
Following the standard practice in matrix analysis~\cite{bhatia2013matrix}, we will lift $\phi$ to symmetric matrices $Y$ satisfying $Y \succ -\gamma_YI_n$ so that $\phi(Y)$ is the matrix with the same eigenvectors as $Y$ but with eigenvalues transformed by $\phi$.

Fix $\beta>0$, $a\in[0,1]$, $A\in\mathrm{Sym}_n(\mathbb R)$,
$W\in\mathcal D_n([0,1])$, and $T\in\mathcal D_n(\mathbb R)$. Denote $M:=a\Id_n-\beta A$ and define thee test functional over the diagonal sub-algebra
\begin{equation}
\label{eq:regularized-scalar-functional}
    \Phi_{a,W,T}(A)
    :=\tr_n\!\left[
        T\,\phi\!\left((\Id_n+WM)(\Id_n+MW)\right)\right].
\end{equation}
In \cite[Lemma B.8]{davies2026potential}, the corresponding function was defined as $\tr_n\!\left(T\,E_{\mathcal D_n}\bigl[D\,(a-\beta A + D)^{-2}D\bigr]\right)$, and we will show in \pref{prop:block-regularization-safe} that these two definitions agree on the safe set defined in \pref{def:good-points}. As in \cite[Lemma B.8]{davies2026potential}, since $T$ is diagonal, inserting $E_{\mathcal D_n}$ around the matrix function leaves \pref{eq:regularized-scalar-functional} unchanged.

We continue to use normalized Schatten norms throughout, with v$\Lpnorm[\infty]{\cdot}=\opnorm{\cdot}$.

To obtain dimension-independent bounds for the matrix Fr\'echet derivatives in the required norms, we first prove bounds on derivatives of the function $(\cdot)_+^3$ applied to matrix arguments, using the same Fr\'echet derivative notation defined in~\cite[\S B.2]{davies2026potential}.
\begin{lemma}[Derivative bounds for the cubed positive part]
\label{lem:positive-part-cubic}
The matrix map $Z\mapsto Z_+^3$ is twice continuously Fr\'echet
differentiable on $\mathrm{Sym}_n(\mathbb R)$.
For every symmetric
$Z$ with $\opnorm{Z}\le1$, symmetric directions $H,K$, and
$p\in\{2,\infty\}$,
\begin{equation}
\label{eq:positive-part-cubic-derivatives}
\allowdisplaybreaks
\begin{aligned}
    \Lpnorm[p]{d(Z_+^3)[H]}
    &\le 4\Lpnorm[p]{H},\\
    \schnorm{d^2(Z_+^3)[H,K]}
    &\le 10\schnorm{H}\opnorm{K}.
\end{aligned}
\end{equation}
\end{lemma}

\begin{proof}
For $z\ne0$, the elementary arctangent integral gives
\[
    \int_0^\infty\frac{z^4}{z^2+t^2}\,dt
    =|z|^3\left[\arctan\!\left(\frac{t}{|z|}\right)\right]_0^\infty
    =\frac{\pi}{2}|z|^3.
\]
The same identity holds at $z=0$, since the integrand vanishes for $t>0$.
Thus $z_+^3=(z^3+|z|^3)/2$ and this generalizes to the matrix case (consider a basis where $Z$ is diagonal) as
\begin{equation}
\label{eq:positive-part-cubic-integral}
    Z_+^3
    =\frac{Z^3}{2}
     +\frac1\pi\int_0^\infty Z^4(Z^2+t^2\Id_n)^{-1}\,dt.
\end{equation}
We apply a partial fraction decomposition. For $t>0$,
\[
    (Z^2+t^2\Id_n)^{-1}
    =\frac{(Z-it\Id_n)^{-1}-(Z+it\Id_n)^{-1}}{2it},.
\]
Letting $S_{\pm} = (Z \pm itI_n)^{-1}$, the product and inverse rules and the
non-commutative H\"older inequality yield
\allowdisplaybreaks
\begin{alignat}{2}
\label{eq:quadratic-resolvent-derivative}
    \Lpnorm[p]{d[(Z^2+t^2\Id_n)^{-1}][H]}
    &= \Lpnorm[p]{\frac{-S_-HS_-+S_+HS_+}{2it}}
    &&\le t^{-3}\Lpnorm[p]{H},\\
\label{eq:quadratic-resolvent-derivative-2}
    \schnorm{d^2[(Z^2+t^2\Id_n)^{-1}][H,K]}
    &= \Lpnorm{\frac{S_-HS_-KS_-+S_-KS_-HS_- - \dots}{2it}} &&\le 2t^{-4}\schnorm{H}\opnorm{K}.
\end{alignat}
For $0<t\le1$, the identity $(Z^2+t^2\Id_n)(Z^2+t^2\Id_n)^{-1} = I_n$ gives a cancellation of the integrand of \pref{eq:positive-part-cubic-integral}:
\allowdisplaybreaks
\begin{align*}
    Z^4(Z^2+t^2\Id_n)^{-1}
    &=Z^2-t^2\Id_n+t^4(Z^2+t^2\Id_n)^{-1},
    \\
    d[Z^4(Z^2+t^2\Id_n)^{-1}][H]
    &=ZH + HZ+t^4\cdot d[(Z^2+t^2\Id_n)^{-1}][H],
    \\
    d^2[Z^4(Z^2+t^2\Id_n)^{-1}][H,K]
    &=KH + HK +t^4\cdot d^2[(Z^2+t^2\Id_n)^{-1}][H,K].
\end{align*}
For $t\ge1$, instead differentiate the integrand directly with the product rule, using
$\opnorm{(Z^2+t^2\Id_n)^{-1}}\le t^{-2}$.
Combining both cases therefore gives, when $\opnorm{Z}\le1$,
\allowdisplaybreaks
\begin{align*}
    \Lpnorm[p]{d[Z^4(Z^2+t^2\Id_n)^{-1}][H]}
    &\le
    \begin{cases}
        3\Lpnorm[p]{H},&0<t\le1,\\
        (4t^{-2}+t^{-3})\Lpnorm[p]{H},&t\ge1,
    \end{cases}\\
    \schnorm{d^2[Z^4(Z^2+t^2\Id_n)^{-1}][H,K]}
    &\le
    \begin{cases}
        4\schnorm{H}\opnorm{K},&0<t\le1,\\
        (12t^{-2}+8t^{-3}+2t^{-4})\schnorm{H}\opnorm{K},&t\ge1.
    \end{cases}
\end{align*}

The integrand itself has operator norm at most $\min\{1,t^{-2}\}$.
The same estimates, with larger numerical constants, hold on $\opnorm{Z}<2$. This gives that the integrand, and its first two derivatives, have locally uniform integrable bounds. this justifies differentiation under the integral and continuity of both derivatives. 
Integrating these bounds and adding the product-rule bounds for $Z^3/2$ yields
\allowdisplaybreaks
\begin{align*}
    \Lpnorm[p]{d(Z_+^3)[H]}
    &\le
    \left[
        \frac32+\frac1\pi\left(3+4+\frac12\right)
    \right]\Lpnorm[p]{H}
    \le4\Lpnorm[p]{H},\\
    \schnorm{d^2(Z_+^3)[H,K]}
    &\le
    \left[
        3+\frac1\pi\left(4+12+4+\frac23\right)
    \right]\schnorm{H}\opnorm{K}
    \le10\schnorm{H}\opnorm{K}.
\end{align*}
By positive homogeneity of degree three, rescaling extends this
twice continuous Fr\'echet differentiability to all of
$\mathrm{Sym}_n(\mathbb R)$.
\end{proof}

\begin{proposition}[Regularity and block identities for the reciprocal extension]
\label{prop:block-regularization}
With $\phi$ and $\Phi_{a,W,T}$ as above, the following claims hold.
All constants below are independent of the matrix dimension.
\begin{enumerate}[label=(\alph*), ref=\theproposition(\alph*), leftmargin=*]

\item \label{prop:block-regularization-cutoff}
\textup{Cutoff bounds.}
The cutoff is well-defined and positive on $(-\gamma_Y,\infty)$.
For $y\ge0$, its denominator satisfies
$\phi(y)^{-1}\ge y$ and
$\phi(y)^{-1}\ge(\sqrt6-2)\gamma_Y\ge4\gamma_Y/9$.
Also $\phi(y)=y^{-1}$ for $y\ge\gamma_Y$.
For every symmetric $Z$,
\begin{equation}
\label{eq:cutoff-square-weighted-bounds}
    \opnorm{\phi(Z^2)}\le\frac{9}{4\gamma_Y},
    \qquad
    \opnorm{Z\phi(Z^2)^{1/2}}\le1.
\end{equation}

\item \label{prop:block-regularization-square}
\textup{Regularity of the squared-argument cutoff.}
The following derivatives are of the whole map $Z\mapsto\phi(Z^2)$.
This map is twice continuously Fr\'echet differentiable
on $\mathrm{Sym}_n(\mathbb R)$ and there is a universal constant $C>0$ such that, for all symmetric
$Z,\widetilde Z,H,K$ and $p\in\{2,\infty\}$,
\allowdisplaybreaks
\begin{align}
\label{eq:cutoff-square-first-derivative}
    \Lpnorm[p]{d[\phi(Z^2)][H]}
    &\le C\gamma_Y^{-3/2}\Lpnorm[p]{H},\\
\label{eq:cutoff-square-second-derivative}
    \schnorm{d^2[\phi(Z^2)][H,K]}
    &\le C\gamma_Y^{-2}\schnorm{H}\opnorm{K},\\
\label{eq:cutoff-square-derivative-difference}
    \schnorm{\bigl(d[\phi(Z^2)]-d[\phi(\widetilde Z^2)]\bigr)[H]}
    &\le C\gamma_Y^{-2}\opnorm{Z-\widetilde Z}\schnorm{H}.
\end{align}

\item \label{prop:block-regularization-representation}
\textup{Block representation and normalization.}
\pref{eq:regularized-scalar-functional} has the exact representation
\begin{equation}
\label{eq:block-trace-functional}
    \Phi_{a,W,T}(A)
    =\tr_{2n}\!\left[
      \begin{pmatrix}2T&0\\0&0\end{pmatrix}
      \phi\!\left(
        \begin{pmatrix}0&\Id_n+WM\\\Id_n+MW&0\end{pmatrix}^{\!2}
      \right)\right].
\end{equation}
For every symmetric direction $H$,
\begin{equation}
\label{eq:block-normalizations}
    \schnorm{\begin{pmatrix}2T&0\\0&0\end{pmatrix}}
       =\sqrt2\schnorm{T},
    \qquad
    \schnorm{\begin{pmatrix}0&WH\\HW&0\end{pmatrix}}
       =\schnorm{WH}\le\schnorm{H}.
\end{equation}

\item \label{prop:block-regularization-affine}
\textup{Affine dependence on $A$.}
For every symmetric direction $H$,
\begin{equation}
\label{eq:block-A-derivative}
    d_A\!\left[
      \begin{pmatrix}0&\Id_n+WM\\\Id_n+MW&0\end{pmatrix}
    \right][H]
    =-\beta\begin{pmatrix}0&WH\\HW&0\end{pmatrix}.
\end{equation}
Its normalized Schatten $2$-norm is at most $\beta\schnorm{H}$.

\item \label{prop:block-regularization-safe}
\textup{Agreement on the safe set.}
For $W\succ0$, set $D=W^{-1}$. If
$(\Id_n+WM)(\Id_n+MW)\succeq\gamma_Y\Id_n$, then
\[
    \phi\!\left((\Id_n+WM)(\Id_n+MW)\right)
    =\left((\Id_n+WM)(\Id_n+MW)\right)^{-1}
    =D(M+D)^{-2}D.
\]
\end{enumerate}
\end{proposition}
\begin{proof}
We prove each of the regularity bounds and the block identity and its affine dependence on $A$ in sequence.
\ppart{Part (a): cutoff bounds}
For $y\ge0$, the inequality $x_+^3\ge3x/4-1/4$ gives
\allowdisplaybreaks
\begin{align*}
    \phi(y)^{-1}
    &\ge y+\gamma_Y\left(
       \frac34\frac{\gamma_Y-y}{\gamma_Y+y}-\frac14\right)\\
    &=(\gamma_Y+y)+\frac{3\gamma_Y^2}{2(\gamma_Y+y)}-2\gamma_Y
    \ge(\sqrt6-2)\gamma_Y\ge\frac49\gamma_Y.
\end{align*}
Also $\phi(y)^{-1}\ge y$. For $-\gamma_Y<y<0$, its denominator is
greater than $y+\gamma_Y>0$.
The bounds in \eqref{eq:cutoff-square-weighted-bounds} follow by applying the same bound to each eigenvalue by using the commutativity of $Z$ and $\phi(Z^2)$.
For $y\ge\gamma_Y$, the correction vanishes and $\phi(y)=y^{-1}$.

\ppart{Part (b): regularity of the squared-argument cutoff}
Write
\begin{equation}
\label{eq:cutoff-square-denominator}
    \phi(Z^2)^{-1}
    =Z^2+\gamma_Y
       \left(2\gamma_Y(\gamma_Y\Id_n+Z^2)^{-1}-\Id_n\right)_+^3.
\end{equation}
The argument of the positive part is a self-adjoint contraction.
Apply \eqref{eq:quadratic-resolvent-derivative} and \eqref{eq:quadratic-resolvent-derivative-2} with
$t=\sqrt{\gamma_Y}$ to obtain
\allowdisplaybreaks
\begin{align*}
    \Lpnorm[p]{d[2\gamma_Y(\gamma_Y\Id_n+Z^2)^{-1}-\Id_n][H]}
    &\le2\gamma_Y^{-1/2}\Lpnorm[p]{H},\\
    \schnorm{d^2[2\gamma_Y(\gamma_Y\Id_n+Z^2)^{-1}-\Id_n][H,K]}
    &\le4\gamma_Y^{-1}\schnorm{H}\opnorm{K}.
\end{align*}
Apply \pref{lem:positive-part-cubic} on the matrix
$2\gamma_Y(\gamma_Y\Id_n+Z^2)^{-1}-\Id_n$, and then use the chain rule.
The correction term in \eqref{eq:cutoff-square-denominator} consequently
satisfies
\allowdisplaybreaks
\begin{align*}
    \Lpnorm[p]{d[\phi(Z^2)^{-1}-Z^2][H]}
    &\le8\gamma_Y^{1/2}\Lpnorm[p]{H},\\
    \schnorm{d^2[\phi(Z^2)^{-1}-Z^2][H,K]}
    &\le56\schnorm{H}\opnorm{K}.
\end{align*}
The second constant is $10\cdot2^2+4\cdot4=56$.
Since $d(Z^2)[H]=ZH+HZ$ and $d^2(Z^2)[H,K]=HK+KH$,
\eqref{eq:cutoff-square-weighted-bounds} and H\"older's inequality give
\begin{equation}
\label{eq:cutoff-square-denominator-derivatives}
\begin{aligned}
    \Lpnorm[p]{\phi(Z^2)^{1/2}
        d[\phi(Z^2)^{-1}][H]\phi(Z^2)^{1/2}}
    &\le21\gamma_Y^{-1/2}\Lpnorm[p]{H},\\
    \schnorm{d^2[\phi(Z^2)^{-1}][H,K]}
    &\le58\schnorm{H}\opnorm{K}.
\end{aligned}
\end{equation}
In the first line, the two terms from $ZH+HZ$ contribute
$3\gamma_Y^{-1/2}\Lpnorm[p]{H}$, and the correction contributes at most
$18\gamma_Y^{-1/2}\Lpnorm[p]{H}$, using \pref{prop:block-regularization-cutoff} to bound the factors of $\phi(Z^2)^{1/2}$.

\ppart{Inverting the denominator}
\pref{lem:positive-part-cubic}, \eqref{eq:cutoff-square-denominator}, and positivity
of the denominator imply twice continuous differentiability on all of
$\mathrm{Sym}_n(\mathbb R)$. Applying the inverse rule gives
\begin{align*}
    d[\phi(Z^2)][H]
    &=-\phi(Z^2)\cdot d[\phi(Z^2)^{-1}][H]\cdot \phi(Z^2),\\
    d^2[\phi(Z^2)][H,K]
    &={}\phi(Z^2)\cdot d[\phi(Z^2)^{-1}][K]\cdot \phi(Z^2)
           \cdot d[\phi(Z^2)^{-1}][H]\cdot \phi(Z^2)\\
    &\quad+\;\phi(Z^2)\cdot d[\phi(Z^2)^{-1}][H]\cdot \phi(Z^2)
           \cdot d[\phi(Z^2)^{-1}][K]\cdot \phi(Z^2)\\
    &\quad-\;\phi(Z^2)\cdot d^2[\phi(Z^2)^{-1}][H,K]\cdot \phi(Z^2).
\end{align*}
In each product containing a first derivative of the denominator, insert
its two adjacent factors $\phi(Z^2)^{1/2}$ and apply
\eqref{eq:cutoff-square-denominator-derivatives}.
The remaining outside factors have total operator norm at most
$9/(4\gamma_Y)$ by \pref{prop:block-regularization-cutoff}. In the last term, apply the second line of
\eqref{eq:cutoff-square-denominator-derivatives} directly.
Thus
\begin{align*}
    \Lpnorm[p]{d[\phi(Z^2)][H]}
    &\le\frac94\,21\,\gamma_Y^{-3/2}\Lpnorm[p]{H},\\
    \schnorm{d^2[\phi(Z^2)][H,K]}
    &\le\left(2\cdot\frac94\,21^2+\frac{81}{16}\,58\right)
        \gamma_Y^{-2}\schnorm{H}\opnorm{K}.
\end{align*}
This proves \eqref{eq:cutoff-square-first-derivative} and
\eqref{eq:cutoff-square-second-derivative}.
The square roots here are used only to bound products after differentiation;
no square root is differentiated.
Finally, integrate the second derivative along the segment from
$\widetilde Z$ to $Z$ to obtain
\eqref{eq:cutoff-square-derivative-difference}.

\ppart{Part (c): block representation and normalization}
The symmetric block matrix satisfies
\[
    \begin{pmatrix}0&\Id_n+WM\\\Id_n+MW&0\end{pmatrix}^{\!2}
    =\begin{pmatrix}
       (\Id_n+WM)(\Id_n+MW)&0\\
       0&(\Id_n+MW)(\Id_n+WM)
     \end{pmatrix}.
\]
Applying $\phi$ and taking the normalized
trace against $\left(\begin{smallmatrix}2T&0\\0&0\end{smallmatrix}\right)$
proves \eqref{eq:block-trace-functional}.
The identities in \eqref{eq:block-normalizations} follow by taking the
normalized traces of the corresponding products with their transposes;
the last inequality uses $\opnorm{W}\le1$.

\ppart{Part (d): affine dependence on $A$}
Since $M=a\Id_n-\beta A$, differentiation with $a$ and $W$ fixed gives
\eqref{eq:block-A-derivative}. Its norm bound follows from
\eqref{eq:block-normalizations}.

\ppart{Part (e): agreement on the safe set}
The cutoff agrees with the reciprocal on the spectrum of
$(\Id_n+WM)(\Id_n+MW)$ under the stated assumption.
Since $\Id_n+WM=W(M+D)$, this product is $W(M+D)^2W$.
Its invertibility implies that $M+D$ is invertible, and taking the inverse
gives $D(M+D)^{-2}D$.
\end{proof}

\subsection{Chaining argument for fluctuations around expected value}
\label{sec:block-mixed-lipschitzness}

In \cite[\S B]{davies2026potential}, the uniform approximation bound on $E_{\mathcal D_n}[D(a\Id_n-\beta A+D)^{-2}D]$ was divided into two parts, with a free interpolation argument bounding its expected value over random $A$, then a chaining argument bounding its fluctuations around its expectation. We now tackle the chaining part of the argument.

Since the set of values of $D$ leading to singular or near-singular $(a\Id_n-\beta A+D)$ depends on the chaining variable $A$, we cannot simply perform the same argument in a safe subset of its original domain $(-1,1)^n$. Therefore, we adapt the argument to prove the bound for the regularized extension $\Phi_{a,W,T}(A)$ from \pref{sec:block-regularization}, which is defined over all of $(-1,1)^n$ and agrees with the original quantity on the safe set.

First, we apply the regularity estimates and block identities from
\pref{prop:block-regularization} to compare the gradients of the
regularized scalar functionals.

\begin{lemma}[Mixed Lipschitzness for chaining]
\label{lem:mixed-lipschitz-gradient}
Fix $\beta,\gamma_Y,C_{\mathrm{op}}>0$ and let
$C_1:=1+\beta C_{\mathrm{op}}$.
Let $A\in\mathrm{Sym}_n(\mathbb R)$ satisfy
$\opnorm{A}\le C_{\mathrm{op}}$, let $a,\tilde a\in[0,1]$, and let
$W,\widetilde W\in\mathcal D_n([0,1])$.
Let $T,\widetilde T\in\mathcal D_n(\mathbb R)$ satisfy
$\schnorm{T},\schnorm{\widetilde T}\le1$.
Set
\[
    M:=a\Id_n-\beta A,
    \qquad \widetilde M:=\tilde a\,\Id_n-\beta A,
\]
and let $\Phi_{a,W,T}$ be the functional defined in
\eqref{eq:regularized-scalar-functional}.
Let $\nabla_A\Phi_{a,W,T}(A)$ be the gradient on
$\mathrm{Sym}_n(\mathbb R)$ with respect to
$\langle H,K\rangle=\tr_n(HK)$.
Then there is a universal constant $C>0$ such that
\begin{equation}
\label{eq:mixed-lipschitz-bound}
\begin{aligned}
    \schnorm{\nabla_A\Phi_{a,W,T}(A)
        -\nabla_A\Phi_{\tilde a,\widetilde W,\widetilde T}(A)}%
    \le C\beta\Big[
       \gamma_Y^{-2}|a-\tilde a|
       +\big(C_1\gamma_Y^{-2}+\gamma_Y^{-3/2}\big)
           \opnorm{W-\widetilde W}
       +\gamma_Y^{-3/2}\schnorm{T-\widetilde T}\Big].
\end{aligned}
\end{equation}
Consequently,
\begin{equation}
\label{eq:cutoff-gradient-bound}
    \schnorm{\nabla_A\Phi_{a,W,T}(A)}
    \le C\beta\gamma_Y^{-3/2}\schnorm{T}.
\end{equation}
\end{lemma}

\begin{proof} We will use the block representation of the test function. We use the fact that the block representation is an affine function of $A$, whose derivative is easy to compute. Extending this to the full test function is a matter of chain rule, followed by combining standard Fr\'echet derivative computations, Cauchy--Schwarz and the established regularity bounds. 
\ppart{Differences of the block matrix}
With $A$ fixed, let $\Delta$ denote the difference between
$(a,W,T)$ and $(\tilde a,\widetilde W,\widetilde T)$, consistently across
all its uses, so that $\Delta M=(\Delta a)\Id_n$.
Note also that $\opnorm{M},\opnorm{\widetilde M}\le C_1$.
For a symmetric direction $H$, \eqref{eq:block-A-derivative} and
\eqref{eq:block-normalizations} bound the normalized Schatten
$2$-norm of the $A$-derivative by $\beta\schnorm{H}$.
Then the discrete product rule gives
\allowdisplaybreaks
\begin{align}
\label{eq:block-parameter-difference}
    \opnorm{\Delta\!\begin{pmatrix}
        0&\Id_n+WM\\\Id_n+MW&0\end{pmatrix}}
    &=\opnorm{WM-\widetilde W\widetilde M}
      \le|\Delta a|+C_1\opnorm{\Delta W},\\
\label{eq:block-derivative-difference}
    \schnorm{\Delta\!\left(
       d_A\!\left[\begin{pmatrix}
          0&\Id_n+WM\\\Id_n+MW&0\end{pmatrix}\right][H]\right)}
    &\le\beta\opnorm{\Delta W}\schnorm{H}.
\end{align}
Here the off-diagonal blocks are transposes of one another.
In particular, the $A$-derivative does not change when only $a$ changes.
\ppart{The scalar functional}
Apply \pref{prop:block-regularization-square} in dimension $2n$, with
its matrix argument $Z$ replaced by
\[
    Z \;\;\gets\;\; \begin{pmatrix}0&\Id_n+WM\\\Id_n+MW&0\end{pmatrix}.
\]
By the chain rule,
\allowdisplaybreaks
\begin{align*}
    d_A\Phi_{a,W,T}(A)[H]
    &=
    -\beta\,\tr_{2n}\!\left[
        \begin{pmatrix}2T&0\\0&0\end{pmatrix}
        \cdot
            d[\phi(Z^2)]
        \!\left[
            \begin{pmatrix}0&WH\\HW&0\end{pmatrix}
        \right]
    \right].
\end{align*}
By Cauchy--Schwarz for $\tr_{2n}$,
\eqref{eq:cutoff-square-first-derivative} in dimension $2n$, \eqref{eq:block-normalizations},
\eqref{eq:block-A-derivative}, $\opnorm{W} \le 1$, and absorbing a factor of $\sqrt{2}$ into $C$,
\allowdisplaybreaks
\begin{align*}
    |d_A\Phi_{a,W,T}(A)[H]|
    &\le C\beta\gamma_Y^{-3/2}
       \schnorm{\begin{pmatrix}2T&0\\0&0\end{pmatrix}}
       \schnorm{\begin{pmatrix}0&WH\\HW&0\end{pmatrix}}\\
    &=\sqrt2\,C\beta\gamma_Y^{-3/2}\schnorm{T}\schnorm{WH}\\
    &\le C\beta\gamma_Y^{-3/2}\schnorm{T}\schnorm{H}.
\end{align*}
Now keep $A$ and $H$ fixed.
In the chain-rule expression for $d_A\Phi_{a,W,T}(A)[H]$,
first change $T$ to $\widetilde T$, then change the direction
supplied to the derivative from the block matrix containing $W$
to that containing $\widetilde W$, and finally change the
evaluation point of $d[\phi(Z^2)]$ from the block matrix associated with $(a,W)$
to that associated with $(\tilde a,\widetilde W)$.
The discrete product rule and Cauchy--Schwarz
for $\tr_{2n}$ then give
\allowdisplaybreaks
\begin{align*}
    &|\Delta(d_A\Phi_{a,W,T}(A)[H])|\\
    &\quad\le
    \sqrt2\,\beta\schnorm{\Delta T}
    \schnorm{
        d[\phi(Z^2)]
        \!\left[
            \begin{pmatrix}0&WH\\HW&0\end{pmatrix}
        \right]
    }\\
    &\qquad+
    \sqrt2\,\beta\schnorm{\widetilde T}
    \schnorm{
        d[\phi(Z^2)]
        \!\left[
            \begin{pmatrix}
                0&(\Delta W)H\\
                H(\Delta W)&0
            \end{pmatrix}
        \right]
    }\\
    &\qquad+
    \sqrt2\,\beta\schnorm{\widetilde T}
    \schnorm{
        \left[\Delta\!\left(
            d[\phi(Z^2)]
        \right)\right]
        \!\left[
            \begin{pmatrix}0&\widetilde{W}H\\H\widetilde{W}&0\end{pmatrix}
        \right]
    }.
\end{align*}

The $\sqrt2$ factors come from
$\schnorm{\left(\begin{smallmatrix}2T&0\\0&0\end{smallmatrix}\right)}
=\sqrt2\schnorm{T}$.

For the first and second terms, apply
\eqref{eq:cutoff-square-first-derivative}.
For the third term, apply
\eqref{eq:cutoff-square-derivative-difference}.
The required bounds on the block matrices are
\allowdisplaybreaks
\begin{align*}
    \schnorm{\begin{pmatrix}0&WH\\HW&0\end{pmatrix}}
    &\le\opnorm{W}\schnorm{H},\\
    \schnorm{\begin{pmatrix}
        0&(\Delta W)H\\
        H(\Delta W)&0
    \end{pmatrix}}
    &\le\opnorm{\Delta W}\schnorm{H},\\
    \opnorm{\Delta\!\begin{pmatrix}
        0&\Id_n+WM\\
        \Id_n+MW&0
    \end{pmatrix}}
    &\le|\Delta a|+C_1\opnorm{\Delta W},
\end{align*}
by \eqref{eq:block-normalizations},
\eqref{eq:block-derivative-difference}, and
\eqref{eq:block-parameter-difference}.
Substituting these estimates into the three terms above and absorbing
$\sqrt2$ into the universal constant gives
\allowdisplaybreaks
\begin{align*}
    |\Delta(d_A\Phi_{a,W,T}(A)[H])|
    \le C\beta\Big[&
       \gamma_Y^{-3/2}\opnorm{W}\schnorm{\Delta T}\\
       &+\gamma_Y^{-3/2}\schnorm{\widetilde T}
          \opnorm{\Delta W}\\
       &+\gamma_Y^{-2}\opnorm{\widetilde{W}}\schnorm{\widetilde T}
          \bigl(|\Delta a|+C_1\opnorm{\Delta W}\bigr)
       \Big]\,\schnorm{H}.
\end{align*}
Finally, using $\schnorm{\widetilde T}\le1$ and $\opnorm{W} \le 1$ and $\opnorm{\widetilde{W}} \le 1$, we obtain
\allowdisplaybreaks
\begin{align*}
    |\Delta(d_A\Phi_{a,W,T}(A)[H])|
    \le C\beta\Big[&
       \gamma_Y^{-3/2}\schnorm{\Delta T}
       +\gamma_Y^{-2}\bigl(|\Delta a|+C_1\opnorm{\Delta W}\bigr)
       +\gamma_Y^{-3/2}\opnorm{\Delta W}
       \Big]\schnorm{H}.
\end{align*}
Taking the supremum over symmetric $H$ with $\schnorm{H}=1$ proves
\eqref{eq:mixed-lipschitz-bound} and then \eqref{eq:cutoff-gradient-bound} follows from the special case when $(\tilde a,\widetilde W,\widetilde T) = (a, W, 0)$, so that $\nabla_A\Phi_{\tilde a,\widetilde W,\widetilde T}(A) = 0$.
\end{proof}
We now follow the proof of \cite[Theorem B.9]{davies2026potential},
using \pref{lem:mixed-lipschitz-gradient} in place of its mixed-Lipschitz
estimate.

\begin{theorem}[Chaining bound for fluctuations of the reciprocal extension $\Phi_{a,W,T}$]
\label{thm:chaining-square-no-sqrtlog}
Fix $\beta,\gamma_Y,C_{\mathrm{op}}>0$, let
$C_1:=1+\beta C_{\mathrm{op}}$, and let $A\sim\mathsf{GOE}(n)$.
Let $\Pi_{C_{\mathrm{op}}}: \mathrm{Sym}_n(\mathbb R)\to\mathrm{Sym}_n(\mathbb R)$ be the metric projection with respect to
$\schnorm{\cdot}$, onto the closed convex set $\{M:\opnorm{M}\le C_{\mathrm{op}}\}$, and set
$\overline A:=\Pi_{C_{\mathrm{op}}}(A)$.
For the functional in \eqref{eq:regularized-scalar-functional}, define
\allowdisplaybreaks
\label{eq:chaining-centered-process}
\begin{align}
    \mathcal J
    &:= [0,1]\times\mathcal D_n([0,1])
       \times\{T\in\mathcal D_n(\mathbb R):\schnorm{T}=1\},\\
    Z_{a,W,T}
    &:=\Phi_{a,W,T}(\overline A)
       -\E_A[\Phi_{a,W,T}(\overline A)].
\end{align}
Then there are universal constants $C_{\mathrm{ch}},C_{\mathrm{tail}}>0$ such that
\allowdisplaybreaks
\begin{equation}
\label{eq:chaining-expectation}
    \E_A\left[\sup_{(a,W,T)\in\mathcal J}|Z_{a,W,T}|\right]
    \le\frac{C_{\mathrm{ch}}\beta}{\sqrt n}
       \left((1+C_1)\gamma_Y^{-2}+\gamma_Y^{-3/2}\right),
\end{equation}
and, for every $p\in(0,1)$,
\allowdisplaybreaks
\begin{equation}
\label{eq:chaining-highprob}
    \P_A\left\{\sup_{(a,W,T)\in\mathcal J}|Z_{a,W,T}|
    \le\frac{C_{\mathrm{ch}}\beta}{\sqrt n}
       \left((1+C_1)\gamma_Y^{-2}+\gamma_Y^{-3/2}\right)
       +\frac{C_{\mathrm{tail}}\beta\gamma_Y^{-3/2}}{n}
          \sqrt{\log\frac2p}\right\} \ge 1 - p.
\end{equation}
\end{theorem}

\begin{proof} The proof will first utilize \pref{lem:mixed-lipschitz-gradient} on the test function induced by the regularized resolvent. After that, the proof is standard by invoking sub-Gaussian concentration via chaining and computing an upper bound on Talagrand's-$\gamma_2$ functional using the Dudley entropy integral.
\ppart{Increment metric}
Define the weights
\[
    \kappa_a:=C\gamma_Y^{-2},\qquad
    \kappa_W:=C\bigl(C_1\gamma_Y^{-2}+\gamma_Y^{-3/2}\bigr),\qquad
    \kappa_T:=C\gamma_Y^{-3/2},
\]
with a common universal prefactor large enough for
\eqref{eq:mixed-lipschitz-bound} and \eqref{eq:cutoff-gradient-bound}.
For $\theta=(a,W,T)$ and
$\tilde\theta=(\tilde a,\widetilde W,\widetilde T)$ in $\mathcal J$, set
\begin{equation}
\label{eq:rho-def}
    \rho(\theta,\tilde\theta)
    :=\kappa_a|a-\tilde a|
      +\kappa_W\opnorm{W-\widetilde W}
      +\kappa_T\schnorm{T-\widetilde T}.
\end{equation}
By \pref{lem:mixed-lipschitz-gradient},
\[
    \schnorm{\nabla_A\Phi_\theta(A)-\nabla_A\Phi_{\tilde\theta}(A)}
    \le\beta\rho(\theta,\tilde\theta)
    \qquad(\opnorm{A}\le C_{\mathrm{op}}).
\]
Integrating along line segments in this convex ball shows that
$\Phi_\theta-\Phi_{\tilde\theta}$ is
$\beta\rho(\theta,\tilde\theta)$-Lipschitz there.

\ppart{Sub-Gaussian increments}
The preceding Lipschitz bound, the 1-Lipschitzness of
$\Pi_{C_{\mathrm{op}}}$, and GOE Herbst concentration
\cite[Lemma B.7]{davies2026potential} give, exactly as in
\cite[proof of Theorem B.9]{davies2026potential},
\begin{equation}
\label{eq:chaining-increment-tail}
    \Pr\{|Z_\theta-Z_{\tilde\theta}|>u\}
    \le2\exp\!\left(
      -\frac{n^2u^2}{4\beta^2\rho(\theta,\tilde\theta)^2}\right)
    =2\exp\!\left(-\frac{u^2}{2d(\theta,\tilde\theta)^2}\right),
    \qquad d:=\frac{\sqrt2\beta}{n}\rho,
\end{equation}
for $u>0$ and $\theta\ne\tilde\theta$.

\ppart{Chaining bound}
Continuity of $\Phi_\theta(\overline A)$ in $\theta$ and its uniform
bound $|\Phi_\theta(\overline A)|\le9/(4\gamma_Y)$ from
\pref{prop:block-regularization-cutoff} imply, by dominated convergence,
that the centered process has continuous sample paths on the compact
set $\mathcal J$ and is therefore separable.
Since $Z_{a,W,-T}=-Z_{a,W,T}$, Talagrand's chaining
\cite[Theorem B.6]{davies2026potential} gives
\[
    \E_A\sup_{\theta\in\mathcal J}|Z_\theta|
    =\E_A\sup_{\theta\in\mathcal J}Z_\theta
    \le C\gamma_2(\mathcal J,d).
\]

\ppart{Entropy integral}
By~\cite[Theorem B.5]{davies2026potential} and
$d=(\sqrt2\beta/n)\rho$,
\begin{equation}
\label{eq:chaining-entropy-integral}
    \gamma_2(\mathcal J,d)
    \le C\frac{\sqrt2\beta}{n}
       \int_0^{\operatorname{diam}(\mathcal J,\rho)}
          \sqrt{\log N(\mathcal J,\rho,u)}\,du,
\end{equation}
where $N$ denotes the covering number.

To bound $N(\mathcal J,\rho,u)$, note that $C_1\ge1$ implies
$\kappa_a+\kappa_T\le\kappa_W$, and hence
\[
    \operatorname{diam}(\mathcal J,\rho)
    \le\kappa_a+\kappa_W+2\kappa_T
    \le3\kappa_W.
\]
For $0<u\le\operatorname{diam}(\mathcal J,\rho)$, cover each parameter
factor at radius $u/(3\kappa_W)\le1$ in its respective norm.
Since $\rho$ is a sum metric and each weight is at most $\kappa_W$,
the resulting product cover has $\rho$-radius at most $u$.
The same small-radius volumetric bounds as in
\cite[proof of Theorem B.9]{davies2026potential} imply
\[
    \log N(\mathcal J,\rho,u)
    \le(2n+1)\log\!\left(\frac{9\kappa_W}{u}\right)\,,
\]
for $0<u\le\operatorname{diam}(\mathcal J,\rho)$. Consequently, extending the integral of this nonnegative upper bound to $9\kappa_W$ and using the same substitution $u=9\kappa_We^{-t}$ as in \cite[proof of Theorem B.9]{davies2026potential}, we obtain
\begin{align*}
    \int_0^{\operatorname{diam}(\mathcal J,\rho)}
       \sqrt{\log N(\mathcal J,\rho,u)}\,du
    &\le
    \sqrt{2n+1}\int_0^{9\kappa_W}
       \sqrt{\log(9\kappa_W/u)}\,du\\
    &=\frac{9\sqrt\pi}{2}\kappa_W\sqrt{2n+1}.
\end{align*}
Substituting this into \eqref{eq:chaining-entropy-integral} yields
\[
    \gamma_2(\mathcal J,d)
    \le C\frac{\beta\kappa_W}{\sqrt n}
    \le\frac{C\beta}{\sqrt n}
       \left(C_1\gamma_Y^{-2}+\gamma_Y^{-3/2}\right),
\]
where the last step uses the definition of $\kappa_W$.
Together with the preceding chaining bound, this proves
\eqref{eq:chaining-expectation}.

\ppart{High-probability bound}
By \eqref{eq:cutoff-gradient-bound} and the same line-segment and
projection argument, each $A\mapsto\Phi_\theta(\overline A)$ is globally
$\beta\kappa_T$-Lipschitz. Centering preserves this constant, and so does
taking the supremum. Thus
\[
    \mathcal Z(A):=\sup_{\theta\in\mathcal J}Z_\theta
                 =\sup_{\theta\in\mathcal J}|Z_\theta|
\]
is $\beta\kappa_T$-Lipschitz.
As in \cite[proof of Theorem B.9]{davies2026potential}, applying GOE Herbst concentration~\cite[Lemma B.7]{davies2026potential} again to this supremum gives
\[
    \Pr\left\{\mathcal Z(A)-\E_A[\mathcal Z(A)]>u\right\}
    \le2\exp\!\left(-\frac{n^2u^2}{4\beta^2\kappa_T^2}\right).
\]
Taking $u=2\beta\kappa_T n^{-1}\sqrt{\log(2/p)}$ and using \eqref{eq:chaining-expectation} proves \eqref{eq:chaining-highprob}.
\end{proof}

\subsection{Expected value from free probability}
\label{sec:option-c-free-comparison}
We compare the unconditional expectation of the reciprocal extension with its free counterpart. Using the block representation from \pref{prop:block-regularization-representation}, we compare GOE to GUE and then apply the matrix-valued resolvent estimate of \cite[Theorem 2.8]{bandeira2023matrix} in dimension $2n$, without a growing identity amplification. Smooth envelopes then recover the same reciprocal extension\footnote{\,This replaces the previous approach of the authors in \cite[\S B.6]{davies2026potential} of taking a real-axis inverse-square limit - see \pref{rem:simpler-free}.}.

We retain the spectral-cutoff $\gamma_Y$, regularized resolvent $\phi$, and diagonal test functional $\Phi_{a,W,T}$ from \eqref{eq:rational-reciprocal-cutoff} and \eqref{eq:regularized-scalar-functional}. Let $S$ be a semicircular element free from
$(M_n(\mathbb C),\tr_n)$ in $(\mathcal M_n,\tau_n)$, and let $E_{M_n}$ be the trace-preserving conditional expectation onto $M_n(\mathbb C)$. Fix deterministic $a\in[0,1]$, $\beta>0$, and $W\in\mathcal D_n([0,1])$. The block identities and reciprocal agreement in \pref{prop:block-regularization-representation}, \pref{prop:block-regularization-affine}, and \pref{prop:block-regularization-safe} are algebraic and hold for Hermitian matrices and $S$. %
We retain the normalized Schatten norms from the preliminaries, except for the explicitly unnormalized Frobenius norm in the covariance-parameter calculation. %

We introduce the GUE as an intermediate ensemble because its coefficient-defined free model has exactly the $M_n$-valued law of $S$, whereas the GOE covariance contains an additional transpose term of order $1/n$. This permits \cite[Theorem 2.8]{bandeira2023matrix} to compare directly with the free limit, while the covariance correction is handled by the GOE--GUE interpolation. This approach is simpler than the previous approach adopted by the authors in \cite[\S B.6]{davies2026potential} where the interpolation for $\sqrt{t}A \ot \Id_k + \sqrt{1-t}B$ to its corresponding free counter-part encounters a covariance structure that cannot be tightly analyzed by the techniques in \cite{bandeira2023matrix,van2025strong}.

\begin{lemma}[Comparison of GOE block resolvents with GUE free limit]
\label{lem:option-c-block-resolvents}
Let $A \sim \mathsf{GOE}(n)$ and $A_{\mathrm{GUE}} \sim \mathsf{GUE}(n)$ independently.
For $t\in[0,1]$ and $z\in\mathbb C\setminus\mathbb R$, set
\[
    A_t=\sqrt t\,A+\sqrt{1-t}\,A_{\mathrm{GUE}}\,,
\]
and
\[
    R_t(z)=\left[z\Id_{2n}-
      \begin{pmatrix}
      0&\Id_n+D^{-1}(a\Id_n-\beta A_t)\\
      \Id_n+(a\Id_n-\beta A_t)D^{-1}&0
      \end{pmatrix}\right]^{-1}\,.
\]
For $b=\Im z\ne0$, we have
\begin{align}
\label{eq:option-c-goe-gue-resolvents}
    &\opnorm{\E R_1(z)-\E R_0(z)}
       \le\frac{2\beta^2}{n|b|^3},\\
\label{eq:option-c-gue-free-resolvents}
    &\opnorm{\E R_0(z)
       -(\mathrm{id}_2\otimes E_{M_n})
       \left[\left(z\Id_{2n}-
       \begin{pmatrix}
       0&\Id_n+D^{-1}(a\Id_n-\beta S)\\
       \Id_n+(a\Id_n-\beta S)D^{-1}&0
       \end{pmatrix}\right)^{-1}\right]}
       \le\frac{2\beta^4}{n|b|^5}.
\end{align}
\end{lemma}
\begin{proof}
We use the characterization of the $M_n(\C)$-valued covariances $\eta_1$ and $\eta_2$ of the $\mathsf{GOE}(n)$ and $\mathsf{GUE}(n)$ ensembles \cite[Remark 1.6]{jekel2025strong} to identify the difference between them, and then invoke \cite[Proposition 1.8]{jekel2025strong} and \cite[Theorem 2.8]{bandeira2023matrix} to get quantitative strong convergence of $A_{\mathrm{GUE}}$ to its $(M_n(\C),\eta)$-semicircular limit. After that, one can use the characterization of $R_t(z)$ as an affine function of $A$ to compute its time derivative, and apply the standard operator-norm bound on the resolvent to compute the error between the $\mathsf{GOE}(n)$ and $\mathsf{GUE}(n)$ ensemble.
\ppart{The free $\mathsf{GUE}(n)$ model}
For every $H\in M_n(\mathbb C)$, the normalized GUE covariance is
\begin{equation}
\label{eq:option-c-gue-covariance}
    \eta_2 = \E[A_{\mathrm{GUE}}HA_{\mathrm{GUE}}]
       =\frac1n\sum_{i,j=1}^nE_{ij}HE_{ji}
       =\tr_n(H)\Id_n,
\end{equation}
where $E_{ij}$ are the standard basis elements for $M_n(\C)$. The Hermitian coefficient matrices multiplying independent standard real Gaussians are
\[
    \calA := \left\{\frac{E_{ii}}{\sqrt n}\right\}_{1\le i\le n} \bigcup \left\{\frac{E_{ij}+E_{ji}}{\sqrt{2n}}\right\}_{1 \le i < j \le n} \bigcup \left\{\frac{i(E_{ij}-E_{ji})}{\sqrt{2n}}\right\}_{1\le i < j \le n}\,.
\]
Replacing the standard Gaussians in front of the coefficient matrices by copies of freely independent standard semicirculars gives an $M_n$-valued semicircular variable with covariance $\eta_2$ \eqref{eq:option-c-gue-covariance} that describes the free limit via \cite[Proposition 1.8]{jekel2025strong}. By the proof of \cite[Propositon 1.8]{jekel2025strong}, all cumulants of $M_n(\C)$ but the first two vanish under the non-commutative conditional expectation $\mathrm{id}\ot\tau$. The
affine (in $A$) block expression of $R_t(z)$ preserves this agreement \cite[Proposition 9.3.13]{mingo2017free}. The block resolvents consequently agree for sufficiently large $|z|$ by their moment series, and then on both half-planes by analytic continuation. This identifies the coefficient-defined free model in \cite[Theorem 2.8]{bandeira2023matrix} exactly with the free resolvent in \eqref{eq:option-c-gue-free-resolvents}.

For the block affine representation at $t=0$, the relevant coefficients vis-a-vis the representation in \cite[(1.1)]{bandeira2023matrix} are $A_0 = \begin{pmatrix}
       0&\Id_n+aD^{-1}\\
       \Id_n+aD^{-1}&0
       \end{pmatrix}$ and $A_{ij} := -\beta 
       \begin{pmatrix}
       0&D^{-1}\calA_{ij} \\
       \calA_{ij}D^{-1}&0
       \end{pmatrix}$.
For these matrices, the parameters $\sigma(A_{\mathrm{GUE}})$ and $v(A_{\mathrm{GUE}})$ can be straightforwardly bounded using the fact that $\opnorm{D^{-1}} \le 1$ along with the operator covariance calculation for $\eta_2$ in \eqref{eq:option-c-gue-covariance} with $H = D^{-2}$ and $H=\Id_n$ to be 
\allowdisplaybreaks
\begin{align*}
    \sigma^2(A_{\mathrm{GUE}}) &= \opnorm{\sum_{\calA_{ij}\in\calA}A^2_{ij}} = \beta^2\opnorm{\sum_{ij}\begin{pmatrix}
       D^{-1}\calA^2_{ij}D^{-1}&0 \\
       0&\calA_{ij}D^{-2}\calA_{ij}
       \end{pmatrix}} \\
       &= \beta^2\opnorm{\begin{pmatrix}
       D^{-2}&0 \\
       0&\tr_n[D^{-2}]\Id_n
       \end{pmatrix}} \le_{\opnorm{D^{-1}}\le 1, \tr_n[D^{-1}] \le 1} \le \beta^2\,,
\end{align*}
and
\allowdisplaybreaks
\begin{align*}
    v^2(A_{\mathrm{GUE}}) &= \opnorm{\Cov(A_{\mathrm{GUE}})} = \sup_{M\in M_n(\C)\,,\norm{M}^2_F=1}\sum_{ij}\left|\Tr\left[A_{ij}M\right]\right|^2 \\
    &= \beta^2\sup_{M\in M_n(\C)\,,\norm{M}^2_F=1}\sum_{ij}\left|\Tr\left[\begin{pmatrix}
       0&D^{-1}\calA_{ij} \\
       \calA_{ij}D^{-1}&0
       \end{pmatrix}\begin{pmatrix}
       M_{11}&M_{12} \\
       M_{21}&M_{22}
       \end{pmatrix}\right]\right|^2 \\
       &= \beta^2\sup_{M\in M_n(\C)\,,\norm{M}^2_F=1}\sum_{ij}\left|\Tr\left[D^{-1}\calA_{ij}M_{21}\right] + \Tr\left[A_{ij}D^{-1}M_{12}\right]\right|^2 \\
       &\le_{\sum_{ij}|\Tr[\calA_{ij}H]|^2 = 1/n\norm{H}^2_F, \norm{M_{12}}_F^2 \le 1, \norm{M_{21}}^2_F \le 1} \frac{2\beta^2}{n}\left|\opnorm{D^{-1}}\right|^2  \le \frac{2\beta^2}{n}\,. 
\end{align*}
Noting that $\tilde{v}^4(A_{\mathrm{GUE}}) = v^2(A_{\mathrm{GUE}})\sigma^2(A_\mathrm{GUE})\le\frac{2\beta^4}{n}$ and invoking \cite[Theorem 2.8]{bandeira2023matrix} proves \eqref{eq:option-c-gue-free-resolvents}.

\ppart{The GOE--GUE correction} The underlying entry covariances, without complex conjugation, differ by
\[
    \E[A_{ij}A_{kl}]
       -\E[(A_{\mathrm{GUE}})_{ij}(A_{\mathrm{GUE}})_{kl}]
       =\frac1n\delta_{ik}\delta_{jl}.
\]
Apply the Gaussian covariance interpolation lemma \cite[Lemma 4.11]{bandeira2023matrix} entrywise to every two Gaussians in $R_t(z)$, in real Hermitian coordinates. The variance differences in the directions $E_{ii}$, $E_{ij}+E_{ji}$, and $i(E_{ij}-E_{ji})$ are $1/n$, $1/(2n)$, and $-1/(2n)$. Combining the latter two directions gives the matrix-unit contraction below. By the affine dependence on $A$ provided by \pref{prop:block-regularization-affine}, the two second-derivative terms from the inverse rule cancel the factor $1/2$ in covariance interpolation. Consequently,
\begin{equation}
\label{eq:option-c-covariance-interpolation}
    \frac{d}{dt}\E R_t(z)
    =\frac{\beta^2}{n}\E\!\left[
      R_t(z)\left\{\sum_{i,j}
       \begin{pmatrix}0&D^{-1}E_{ij}\\E_{ij}D^{-1}&0\end{pmatrix}
       R_t(z)
       \begin{pmatrix}0&D^{-1}E_{ij}\\E_{ij}D^{-1}&0\end{pmatrix}
      \right\}R_t(z)\right].
\end{equation}
The integrand is well-defined with bounded derivatives for fixed $b\ne0$. Using $\sum_{i,j}E_{ij}HE_{ij}=H^{\mathsf T}$, the matrix in braces equals
\[
    \begin{pmatrix}
       D^{-1}(R_t)_{22}^{\mathsf T}D^{-1}&(D^{-1})^2(R_t)_{21}^{\mathsf T}\\
       (R_t)_{12}^{\mathsf T}(D^{-1})^2&D^{-1}(R_t)_{11}^{\mathsf T}D^{-1}
    \end{pmatrix}\,.
\]
Transposition preserves operator norm, and each block has norm at most $\opnorm{R_t}$. A $2\times2$ block matrix has operator norm at most twice its largest block norm, and since $\opnorm{R_t(z)}\le|b|^{-1}$, we have
\[
    \opnorm{\frac{d}{dt}\E R_t(z)}
       \le\frac{2\beta^2}{n}\E\opnorm{R_t(z)}^3
       \le\frac{2\beta^2}{n|b|^3}\,.
\]
Integrating $R_z(t)$ over $t\in[0,1]$ gives \eqref{eq:option-c-goe-gue-resolvents}. The lower-half-plane version of the GUE--free estimate also follows from the upper-half-plane version by taking adjoints.
\end{proof}

\begin{remark}\label{rem:simpler-free}
The application of \cite[Theorem 2.8]{bandeira2023matrix} and \cite[Lemma 5.11]{bandeira2023matrix} was obstructed in the previous argument of \cite[\S B.6]{davies2026potential} by the use of $A\otimes\Id_k$, which has $v(A\otimes\Id_k)=\sqrt{k}\,v(A)$ by
\cite[Lemma 8.1]{bandeira2023matrix}.
In the present approach, using the regularizer $\phi$ applied to the two-block resolvent, one only incurs a factor $2$ in $v^2$.
While the previous tensor interpolation argument could've been adapted to this new regularized version of the functional $\Phi_{a,W,T}(A)$, this simplification allows for a much shorter argument.
\end{remark}

The lemma below will be invoked to justify comparison between the application of $\phi$ to the linear block resolvent and the ideal free limit, by using a smooth envelope that will agree under \pref{prop:block-regularization-safe} and will be shown to have negligible error outside it. 
\begin{lemma}[Smooth test transfer]
\label{lem:option-c-smooth-transfer}
Under the hypotheses of \pref{lem:option-c-block-resolvents}, every real $f \in L^1(\R)$ that is six times weakly differentiable, with all derivatives $f^{(j)} \in L^1(\R)$, satisfies
\begin{equation}
\label{eq:option-c-smooth-transfer}
\begin{aligned}
 &\left\|\E_A f\!\left(\gamma_Y^{-1/2}
       \begin{pmatrix}0&\Id_n+D^{-1}(a\Id_n-\beta A)\\
       \Id_n+(a\Id_n-\beta A)D^{-1}&0\end{pmatrix}\right)\right.\\[-2pt]
 &\qquad\left.-(\mathrm{id}_2\otimes E_{M_n})
       \left[f\!\left(\gamma_Y^{-1/2}
       \begin{pmatrix}0&\Id_n+D^{-1}(a\Id_n-\beta S)\\
       \Id_n+(a\Id_n-\beta S)D^{-1}&0\end{pmatrix}\right)\right]
       \right\|_{\mathrm{op}}\\
 &\quad\le\frac{C}{n}\left[
      \beta^2\gamma_Y^{-1}\sum_{j=0}^4\|f^{(j)}\|_{L^1(\mathbb R)}
     +\beta^4\gamma_Y^{-2}\sum_{j=0}^6\|f^{(j)}\|_{L^1(\mathbb R)}\right]\, ,
\end{aligned}
\end{equation}
where $C>0$ is a universal constant.
\end{lemma}

\begin{proof}
For a self-adjoint operator $K$, scaling its resolvent gives
\[
    (z\Id-\gamma_Y^{-1/2}K)^{-1}
       =\sqrt{\gamma_Y}\,(\sqrt{\gamma_Y}z\Id-K)^{-1}.
\]
Therefore, the two resolvent bounds from \pref{lem:option-c-block-resolvents} become $2\beta^2/(n\gamma_Yb^3)$ and $2\beta^4/(n\gamma_Y^2b^5)$ for the block matrices rescaled by the spectral cutoff $\gamma_Y$. Now, choose any $u \in \calS_\C^{2n-1}(1)$ and consider the (average) $u$-weighted empirical eigendensities with $A$ and $A_{\mathrm{GUE}}$
\[
    u^\sT\left(\E_A[R_{1}(z)]-\E_{\mathrm{GUE}}[R_{0}(z)]\right)u = g_{\mu_A,u} - g_{\mu_{A_{\mathrm{GUE},u}}} \le_{u^\sT C u \le \opnorm{C}} \opnorm{\E_A[R_{1}(z)]-\E_{\mathrm{GUE}}[R_{0}(z)]} \le_{\text{\pref{eq:option-c-goe-gue-resolvents}}} \frac{2\beta^2}{nb^3}\,,
\]
and the $u$-weighted spectral measure with $S$ and $A_\mathrm{GUE}$ 
\[
    u^\sT \left(\E_{\mathrm{GUE}}[R_{0}(z) ]- (\id_2\ot E_{M_n})R_S(z)\right) u = g_{\mu_{A_{\mathrm{GUE},u}}} - g_{\mu_S,u} \le \opnorm{\E_{\mathrm{GUE}}[R_{0}(z) ]- (\id_2\ot E_{M_n})R_S(z)} \le_{\text{\pref{eq:option-c-gue-free-resolvents}}} \frac{2\beta^4}{nb^5}\,,
\]
where $g_{\mu_{A_{\mathrm{GUE},u}}}, g_{\mu_{A_{u}}}, g_{\mu_S,u}$ denote the Cauchy--Stieltjes transforms for the subscripted operators with measure given by the empirical eigendensity reweighed by $u$\footnote{\,To see this, apply the spectral theorem to the respective operators, and take the quadratic form. Then, use the fact that $u$ is unit-norm to construct the $u$-reweighed empirical eigendensity.}. By \cite[Lemma 5.11]{bandeira2023matrix}, a Stieltjes-transform difference bounded by $Kb^{-m}$ gives a smooth-test difference at most $C_mK\sum_{j=0}^{m+1}\|f^{(j)}\|_1$. We apply this twice---once to the GOE--GUE pair with $m=3$, and the second time to the GUE--free pair with $m=5$. We conclude with a triangle inequality, followed by taking the supremum over unit-vectors $u$, which proves \eqref{eq:option-c-smooth-transfer}.
\end{proof}

\begin{proposition}[Expectation comparison for the reciprocal extension]
\label{prop:option-c-expectation}
Fix $D\in\mathcal D_n(\mathbb R)$ with $D\succeq\Id_n$ and let $a\in[0,1]$, $\beta>0$, $\gamma_Y>0$. Suppose that the spectral gap condition in the free limit
\begin{equation}
\label{eq:option-c-free-buffer}
    (\Id_n+D^{-1}(a\Id_n-\beta S))(\Id_n+(a\Id_n-\beta S)D^{-1})
       \succeq4\gamma_Y\Id_n
\end{equation}
holds in $\mathcal M_n$. Then
\begin{equation}
\label{eq:option-c-expectation-bound}
\begin{aligned}
 &\Big\|\E_A\phi\!\left(
       (\Id_n+D^{-1}(a\Id_n-\beta A))(\Id_n+(a\Id_n-\beta A)D^{-1})\right)\\
 &\qquad- E_{M_n}\!\left[D(a\Id_n-\beta S+D)^{-2}D\right]
       \Big\|_{\mathrm{op}}
       \le\frac{C}{n}\left(\beta^2\gamma_Y^{-2}
                             +\beta^4\gamma_Y^{-3}\right).
\end{aligned}
\end{equation}
In particular, for every diagonal $T$,
\begin{equation}
\label{eq:option-c-scalar-expectation}
    \left|\E_A\Phi_{a,W,T}(A)
       -\tau_n\!\left[TD(a\Id_n-\beta S+D)^{-2}D\right]\right|
       \le\frac{C}{n}\left(\beta^2\gamma_Y^{-2}
                             +\beta^4\gamma_Y^{-3}\right)\schnorm{T},
\end{equation}
where $W:=D^{-1}$.
\end{proposition}

\begin{proof} The main goal of the proof is to ``envelope'' (or sandwich) the function $\phi$ inside its range by two functions that satisfy the differentiability requirements of \pref{lem:option-c-smooth-transfer}. We then show that both the enveloping functions will lead to the the same free limit using the spectral-gap condition in \eqref{eq:option-c-free-buffer}. After that, it is a simple matter of invoking \pref{lem:option-c-smooth-transfer} on the sandwiching functions to ``lift'' them to sandwich the operator-valued function.

\ppart{Smooth envelopes}
To apply \pref{lem:option-c-smooth-transfer}, which requires derivatives through order six, we sandwich the reciprocal extension $\sqrt{\gamma_Y}x$ by smooth ``envelopes''. Basic algebra shows that the rescaled function
\[
    x\longmapsto\gamma_Y\phi\left((\sqrt{\gamma_Y}x)^2\right)
      =\left(x^2+
        \left(\frac{(1-x^2)_+}{1+x^2}\right)^3\right)^{-1}
\]
is even and independent of $\gamma_Y$. By \pref{prop:block-regularization-cutoff}, the image of $\gamma_T\phi(\cdot)$ lies in $(0,9/4]$ and equals $x^{-2}$ for $|x|\ge1$.
We set $f_{-}(x)$ to be a smooth function that is equal to $0$ on $[-1,1]$ and to $x^{-2}$ for $|x|\ge2$, and is the lower sandwich for $\gamma_Y\phi(\cdot)$. Similarly, $f_{+}(x)$ is a smooth function that is equal to $9/4$ on $[-1,1]$ and to $x^{-2}$ for $|x|\ge2$, and is the upper sandwich for $\gamma_Y\phi(\cdot)$. On $1<|x|<2$, interpolate smoothly between the respective inner formula and $x^{-2}$ using a weight in $[0,1]$ that is constant near the endpoints\footnote{\,An explicit such choice comes from taking any fixed small $0<\delta<1/2$, and a smooth and even function $g: \R \to [0,1]$ that is $0$ whenever $|x| \le 1+\delta$ and $1$ whenever $|x| \ge 1+2\delta$. Then, set $f_{-}(x) = g(x)x^{-2}$ for $x\neq 0$ (and $0$ otherwise) and $f_{+}(x) = f_{-}(x) + (1-g(x))9/4$.}. The function $f_{-}(x)$ is at most $x^{-2}$ there, and the function $f_{+}(x)$ is at least $x^{-2}$. These two functions, therefore, sandwich $\gamma_Y\phi(\gamma_Yx^2)$ on $\mathbb R$ while being even and smooth.
Since $(-2,2)$ is a closed set on which both functions are bounded, and outside that they are equal to $x^{-2}$, both these functions are integrable. Furthermore, all their derivatives derivatives exist, are $0$ inside $[-1,1]$, are some bounded smooth function in $(-2,-1] \cup [1,2)$, and are equal to $ f^{(k)} = d/dx^{k} x^{-2} = \prod_{i=1}^k (-i+1) x^{-2-k}$ for $\R \setminus (-2,2)$, and so they are integrable. Consequently, $\sum_{j=0}^6 \norm{f^{(j)}}_1 \le C$ for some $C>0$. Lastly, note that the rescalings $\gamma_Y^{-1}f_{\pm}(x/\sqrt{\gamma_Y})$ sandwich $\phi(x^2)$ and equal $x^{-2}$ for $|x| \geq 2 \sqrt{\gamma_Y}$.

\ppart{The common free limit} \eqref{eq:option-c-free-buffer} implies that $a\Id_n-\beta S+D$ is invertible, since, conjugating by $D$ gives
\[
    (a\Id_n-\beta S+D)^2\succeq4\gamma_YD^2\succeq_{D \succ I_n}4\gamma_Y\Id_n\,.
\]
Note that \eqref{eq:option-c-free-buffer} also implies that $\Id_n+D^{-1}(a\Id_n-\beta S)$ is invertible since $D\succeq \Id_n$, and its two squared singular-value operators  are unitarily conjugate\footnote{\,This simply means that there exists a unitary $U$, such that, for any invertible operator $Y$ in the algebra $YY^* = U(Y^*Y)U^*$.}.
This immediately implies that both diagonal blocks in the square of the free block matrix consequently
have spectrum in $[4\gamma_Y,\infty)$, since
\[
    (a\Id_n-\beta S +D)D^{-2}(a\Id_n-\beta S +D) = U(D^{-1}(a\Id_n-\beta S + D)^{2}D^{-1})U^* \succeq U(4\gamma_Y \Id_n)U^* \ge 4\gamma_Y\Id_n\,.
\]
This means that $\mathsf{Spec}\left(\text{free block}\right) \notin (-2\sqrt{\gamma_Y},2\sqrt{\gamma_Y})$.
Both $\gamma^{-1}_Yf_{\pm}(x/\sqrt{\gamma_Y})$ agree outside the spectrum with $x^{-2}$ (which is $\phi(x^2)$ in this range) -
this means their (non-commutative) free expectations are identical.

\ppart{Operator sandwich}
We now invoke \pref{lem:option-c-smooth-transfer} on $f_{\pm}(x/\sqrt{\gamma_Y})$ and
multiply by $\gamma_Y^{-1}$, which gives
\[
    \opnorm{\E_A[f_{\pm}(\gamma_Y^{-1/2}\calB(A))]-\id_2\ot E_{M_n}[f_{\pm}(\gamma^{-1/2}\calB(S))]} \le \frac{C}{n}\left(\beta^2\gamma^{-1}_Y + \beta^4\gamma^{-2}_Y\right)\,,
\]
where $\calB(\cdot)$ is the block matrix taken as input to $f$ in \pref{lem:option-c-smooth-transfer}. Multiplying by $\gamma_Y^{-1}$, the finite expected value of each rescaled envelope $\gamma^{-1}_Yf(\gamma^{-1/2}_Yx)$ is within $Cn^{-1}(\beta^2\gamma_Y^{-2}+\beta^4\gamma_Y^{-3})$ in operator norm of that common free conditional expectation.
Spectral functional calculus at each fixed block matrix, followed by
positivity of expectation, places the expected reciprocal extension
between these two matrix bounds. It therefore satisfies the same
operator-norm estimate. This uses pointwise ordering of scalar functions,
not operator monotonicity with respect to different matrix arguments.
Taking the upper-left block and applying the block identity and reciprocal
agreement from \pref{prop:block-regularization-representation} and
\pref{prop:block-regularization-safe} at the free endpoint proves
\eqref{eq:option-c-expectation-bound}.
The definition of $\Phi_{a,W,T}$, trace duality, and
$\Lpnorm[1]{T}\le\schnorm{T}$ then give
\eqref{eq:option-c-scalar-expectation}, using that $E_{M_n}$ is
trace-preserving. We defer diagonal norm duality and the restriction of
the realized parameters to the section-final corollary.
\end{proof}

\subsection{Uniform estimate for the diagonal}\label{sec:uniform-disagonal-estimate} We combine the chaining estimate with the expectation estimate established by the free interpolation between the expected regularized resolvent and the ideal free object, as in \cite[Corollary B.35]{davies2026potential}, restricting to the safe set in the final step\footnote{\,As stated before, the free interpolation in \S \ref{sec:option-c-free-comparison} is a simplification of the original interpolation argument used in \cite[\S B]{davies2026potential}. The key observation that allows for this is the representation of the input to the regularized resolvent $\phi(\cdot)$ having a block representation with entries that are affine functions of $A$ (\pref{prop:block-regularization-affine}). This allows the results of \cite[\S 2]{bandeira2023matrix} to directly control the error between the ideal free object and (average) regularized resolvent.}. This allows us to control the deviation between the diagonal of the squared resolvent from its ideal free-probabilistic limit with high probability over $A$.

\begin{corollary}[Diagonal of the squared resolvent on the safe set]\label{cor:diagonal-Q2-controlled-in-l2}
Let $A \sim \mathsf{GOE}(n)$ and fix $0<\beta<1$, $0<\gamma_Y\le\frac{(1-\beta)^4}{4(1+2\beta^2)^2}$. For all $n \ge n_0(\beta,\gamma_Y)$, uniformly over all
\begin{equation}
\label{eq:diagonal-safe-set-conditions}
    \left\{D\in\mathcal D_n([1,\infty)) \mid D^{-1}(a\Id_n-\beta A+D)^2D^{-1}\succeq\gamma_Y\Id_n\right\}\, ,
\end{equation}
the following holds
\begin{equation}
\label{eq:diagonal-DQ2D-controlled-in-l2}
    \P_A\left\{\schnorm{
      E_{\mathcal D_n}\!\left[
        D(\beta^2\tr_n(D^{-1})\Id_n-\beta A+D)^{-2}D\right]
      -\frac{\Id_n}{1-\beta^2\tr_n(D^{-2})}
    }
    \le\frac{C\beta}{\gamma_Y^2\sqrt n}\right\} \ge 1-4e^{-n/4},
\end{equation}
for a universal constant $C>0$. In particular, on the same event, \eqref{eq:diagonal-DQ2D-controlled-in-l2} holds uniformly for $D=D(m)$ with $m\in\calS_A(c)$ from \pref{def:good-points}, for any fixed $c>0$ such that
\begin{equation}
\label{eq:option-c-finite-calibration}
    \gamma_Y\le\left(1+\frac{1+3\beta}{c}\right)^{-2}.
\end{equation}
\end{corollary}

\begin{proof}
We first control the deviation of the regularized resolvent. Set $C_{\mathrm{op}}=3$ and $\overline A=\Pi_3(A)$ in \pref{thm:chaining-square-no-sqrtlog}, so that $C_1=1+3\beta\le4$ and invoke \pref{prop:goe-norm-tail}.  Then, invoke \eqref{eq:chaining-highprob} with $p=2e^{-n/4}$. Since $\gamma_Y\le1$ and $\sqrt{\log(2/p)}=\sqrt n/2$, 
\[
    \P_{A}\left\{\sup_{(a,W,T)\in\mathcal J}|Z_{a,W,T}| \le\frac{C\beta}{\gamma_Y^2\sqrt n}\right\} \ge 1-2e^{-n/4}\,.
\]

Set a deterministic $W=D\succeq\Id_n$ and $a=\beta^2\tr_n(D^{-1})\in[0,1]$. \pref{lem:option-c-tap-free-endpoint} provides a spectral cut-off via the TAP subordination so that \pref{eq:option-c-free-buffer} holds, yielding \eqref{eq:option-c-scalar-expectation} as the free target. Consequently, for every diagonal $T$,
\[
    \left|\E_A\Phi_{a,W,T}(A)
      -\frac{\tr_n(T)}{1-\beta^2\tr_n(D^{-2})}\right|
    \le\frac{C}{n}
       \left(\beta^2\gamma_Y^{-2}+\beta^4\gamma_Y^{-3}\right)\schnorm{T}.
\]
Also, \pref{prop:block-regularization-cutoff} and \pref{prop:goe-norm-tail} give
\begin{equation}
\label{eq:option-c-projection-mean}
    \left|\E_A\Phi_{a,W,T}(\overline A)-\E_A\Phi_{a,W,T}(A)\right|
    \le\frac{9}{\gamma_Y}e^{-n/4}\schnorm{T},
\end{equation}
since the two observables agree when $\opnorm{A}\le3$ and each has absolute value at most $9\schnorm{T}/(4\gamma_Y)$. %
For fixed $\beta,\gamma_Y$, these two expectation errors are $o(\beta\gamma_Y^{-2}n^{-1/2})$ and so, on the chaining event, the triangle inequality yields
\[
    \left|\Phi_{a,W,T}(\overline A)
       -\frac{\tr_n(T)}{1-\beta^2\tr_n(D^{-2})}\right|
    \le\frac{C\beta}{\gamma_Y^2\sqrt n}\,,
\]
uniformly over valid $D$ and all diagonal $T$ with $\schnorm{T}=1$, for all $n \ge n_0(\beta,\gamma_Y)$. Intersecting the chaining event with $\{\opnorm{A}\le3\}$ gives probability at least $1-4e^{-n/4}$. On this intersection $\overline A=A$, and \pref{prop:block-regularization-safe} makes the reciprocal extension equal to $D(a\Id_n-\beta A+D)^{-2}D$ for every realized $D$ satisfying \eqref{eq:diagonal-safe-set-conditions}. Taking the supremum over $T$ and applying diagonal norm duality (see \cite[Proof of Lemma 4.25]{jekel2024pha}) proves
\eqref{eq:diagonal-DQ2D-controlled-in-l2}.

Finally, to restrict to the good set, choose $m\in\calS_A(c)$ and set $D=D(m)$. Then, $a=\beta^2\tr_n(D^{-1})=\beta^2(1-\norm{m}_2^2/n)$. Since $\opnorm{a\Id_n-\beta A}\le1+3\beta$ on the same event, \eqref{e:FTAP} and \pref{def:good-points} give
\begin{align*}
    a\Id_n-\beta A+D
    &=\nabla_m^2\calF_{\TAP}(m,\cdot)+\frac{2\beta^2}{n}mm^{\mathsf T}
      \succ cD\succeq c\Id_n,\\
    \opnorm{(\Id_n+D^{-1}(a\Id_n-\beta A))^{-1}}
    &=\opnorm{\Id_n-(a\Id_n-\beta A+D)^{-1}(a\Id_n-\beta A)}
      \le1+\frac{1+3\beta}{c}.
\end{align*}
Therefore, \eqref{eq:option-c-finite-calibration} implies the product condition \eqref{eq:diagonal-safe-set-conditions} uniformly over $m\in\calS_A(c)$.
\end{proof}

\section{Properties of the algorithmic stochastic process}
\label{sec:alg-properties}
To prove the desiderata \ref{d:drift-error}, \ref{d:Q-Lip}, and \ref{d:JE-lip}, we need the analog of \cite[\S6]{davies2026potential}, but restricted to the good set $\calS_A(c)$ defined in \pref{def:good-points}, as opposed to the uniform bounds over $(-1,1)^n$ that can be obtained when $\beta < 1/2$.
Recall the coefficient functions
\begin{equation}
\label{eq:DahQ}
\begin{aligned}
    D(m)&:=(\Id_n-\diag(m)^2)^{-1},
    &a(m)&:=\beta^2\tr_n(D(m)^{-1}),\\
    \bar Q(m)&:=(a(m)\Id_n-\beta A+D(m))^{-1},
    &\hat Q(m)&:=\left(a(m)\Id_n-\beta A+D(m)
                       -\frac{2\beta^2}{n}mm^{\sT}\right)^{-1}.
\end{aligned}
\end{equation}
Recall also the It\^o correction to ASL-TAP
\begin{equation}
\label{eq:alg-f-definition}
 f(m):=\left(E_{\mathcal D_n}[D(m)\hat Q(m)^2D(m)]
       -\beta^2\tr_n(\hat Q(m)^2)\Id_n
       -\frac{2\beta^2}{n}\hat Q(m)^2\right)m.
\end{equation}
Nearly all of the argument structure in \cite[\S6]{davies2026potential} is reused, so we deliberately maintain as much of the same document structure as possible: establishing basic resolvent bounds, controlling the rank-one
correction and diagonal error, and then deducing regularity of the drift
and properties of the trajectories.

The only new difficulty is that the set $\calS_A(c)$ that we work in is no longer convex.
Thus we can no longer bound Lipschitz constants by integrating bounds on their derivatives along straight line segments.
The main workaround is to use the resolvent identity to directly calculate the differences between the endpoints, and \pref{lem:resolvent-difference} collects those calculations, replacing \cite[Lemma B.4]{davies2026potential}.
To obtain the transportation inequality and therefore sub-Gaussian concentration of Lipschitz path functionals as in \cite[\S6.6]{davies2026potential}, we use the Kirszbraun extension theorem to fill in values for the SDE coefficients outside of $\calS_A(c)$.
Finally, the main new argument is in~\pref{lem:omega-safe-Lipschitz}, where we provide a Lipschitz bound and extension of the JE weight integrand, this being the only place where the resolvent identity doesn't suffice to establish the needed Lipschitz bounds.

As before, vector norms are Euclidean and $\schnorm{\cdot}$ is the normalized Schatten $2$-norm, so that
$\lVert\cdot\rVert_F=\sqrt n\schnorm{\cdot}$, and
all Lipschitz constants for matrix-valued maps in this section use
$\schnorm{\cdot}$ unless another norm is specified.

\subsection{High-probability events}
\label{sec:alg-events}
Throughout this section, we work on the event $\{\opnorm{A}\le3\}$, which we recall for GOE $A$ holds with probability at least $1-2e^{-n/4}$ by \pref{prop:goe-norm-tail}.

Fix $0<\beta<1$ and $c>0$, and choose
\begin{equation}
\label{eq:alg-cutoff-choice}
 \gamma_Y:=\min\left\{
       \frac{(1-\beta)^4}{4(1+2\beta^2)^2},\,
       \left(1+\frac{1+3\beta}{c}\right)^{-2}\right\}.
\end{equation}
Let 
\begin{equation}
\label{eq:alg-def-calSD}
    \calS_D
    :=\left\{
       D\in\mathcal D_n([1,\infty)):
       \beta^2\tr_n(D^{-1})\Id_n-\beta A+D\succeq c\Id_n
    \right\}.
\end{equation}
For all sufficiently large $n$, depending only on $\beta,c$, let
$\Omega_{\beta,c}$ be the conclusion of
\pref{cor:diagonal-Q2-controlled-in-l2}, so that there is a universal constant $C$ so that
\begin{equation}
\label{eq:alg-def-Omega}
    \Omega_{\beta,c} := 
    \left\{\forall D \in \calS_D \mathrel{.}\schnorm{
      E_{\mathcal D_n}\!\left[
        D(\beta^2\tr_n(D^{-1})\Id_n-\beta A+D)^{-2}D\right]
      -\frac{\Id_n}{1-\beta^2\tr_n(D^{-2})}
    }
    \le\frac{C\beta}{\gamma_Y^2\sqrt n}
    \right\}.
\end{equation}
Then $\Pr_A(\Omega_{\beta,c} \cap \{\opnorm{A}\le3\})\ge1-4e^{-n/4}$ by the proof of \pref{cor:diagonal-Q2-controlled-in-l2}.
The choice \eqref{eq:alg-cutoff-choice} ensures that this diagonal bound
also applies to every $m\in(-1,1)^n$ satisfying
$a(m)\Id_n-\beta A+D(m)\succeq c\,\Id_n$.
Constants denoted $C_{\beta,c}$ may increase between occurrences
and are independent of $n$ and of $A$.

\subsection{Basic resolvent and Lipschitzness bounds}
\label{sec:alg-basic-bounds-lipschitz}

The resolvent operator norm and Lipschitzness bounds exactly echo those of \cite[\S 6.2]{davies2026potential}, with the exception of an additional pair of bounds established for $\bar{Q}$ to assist in the Lipschitz extension of the JE weights.

\begin{lemma}[Basic resolvent bounds on the good set]
\label{lem:alg-resolvent-bounds}
Let $D,a,\bar Q,\hat Q$ be as in \eqref{eq:DahQ}.
Assume $\opnorm{A}\le3$.
Then for every $m\in\calS_A(c)$, both $\bar{Q}$ and $\hat{Q}$ are positive definite and
satisfy
\begin{enumerate}[label=(\alph*),ref=\thetheorem(\alph*),itemsep=0.1em]
\item \label{lem:alg-resolvent-bounds-Q}
\[ \opnorm{\bar Q(m)},\ \opnorm{\hat Q(m)}\le c^{-1}. \]
\item \label{lem:alg-resolvent-bounds-DQ}
\begin{equation}
\label{eq:alg-weighted-Q-bounds}
 \opnorm{D(m)\bar Q(m)},\ \opnorm{D(m)\hat Q(m)}
       \le1+\frac{3\beta+3\beta^2}{c}.
\end{equation}
\item \label{lem:alg-resolvent-bounds-DQ2}
\[ \opnorm{D(m)\bar Q(m)^2},\ \opnorm{D(m)\hat Q(m)^2}
       \le c^{-1}\bigl(1+(3\beta+3\beta^2)/c\bigr). \]
\item \label{lem:alg-resolvent-bounds-DQ2D}
\[ \opnorm{D(m)\bar Q(m)^2D(m)},\ \opnorm{D(m)\hat Q(m)^2D(m)}
       \le\bigl(1+(3\beta+3\beta^2)/c\bigr)^2. \]
\end{enumerate}
\end{lemma}
\begin{proof}
The proof is the same as \cite[Lemma 6.1]{davies2026potential}, except
that the gap now follows from
\[
 \hat Q(m)^{-1}\succ cD(m)\succeq c\Id_n,\qquad
 \bar Q(m)^{-1}=\hat Q(m)^{-1}+\frac{2\beta^2}{n}mm^{\sT}
                       \succ cD(m).
\]
In \cite[Lemma B.1]{davies2026potential}, use $W=D(m)^{-1}$ and,
respectively, $M=a(m)\Id_n-\beta A$ or
$M=a(m)\Id_n-\beta A-2\beta^2mm^{\sT}/n$.
Both have $\opnorm{M}\le3\beta+3\beta^2$.
The identities $DR=\Id_n-MR$ and $RD=\Id_n-RM$ give the weighted bounds.
\end{proof}

\begin{lemma}[Basic resolvent Lipschitzness on the good set]
\label{lem:alg-resolvent-lipschitz}
In the setting of \pref{lem:alg-resolvent-bounds}, the following
Lipschitz bounds hold for the following functions of $m$ on $\calS_A(c)$, with Euclidean distance on the
input $m \in \calS_A(c)$ and normalized Schatten $2$-norm $\schnorm{\cdot}$ on the output:
\allowdisplaybreaks
\begin{enumerate}[label=(\alph*),ref=\thetheorem(\alph*),itemsep=0.1em]
\item \label{lem:alg-resolvent-lipschitz-Q}
\begin{equation}
\label{eq:alg-resolvent-endpoint-Lipschitz}
 \lipnorm{\hat Q},\ \lipnorm{\bar Q}
                         \le\frac{C_{\beta,c}}{\sqrt n}.
\end{equation}
\item \label{lem:alg-resolvent-lipschitz-Q2}
\[ \lipnorm{\hat Q^2}\le\frac{C_{\beta,c}}{\sqrt n}. \]
\item \label{lem:alg-resolvent-lipschitz-DQ}
\[ \lipnorm{D\hat Q}\le\frac{C_{\beta,c}}{\sqrt n}. \]
\item \label{lem:alg-resolvent-lipschitz-DQQD}
\[ \lipnorm{D\hat Q^2D}\le\frac{C_{\beta,c}}{\sqrt n}. \]
\end{enumerate}
\end{lemma}
\begin{proof}
Follow \cite[Lemma 6.2]{davies2026potential}, replacing its derivative
bounds by parts (a)--(e) of \pref{lem:resolvent-difference} with
\[
 W\gets\Id_n-\diag(m)^2,\qquad
 M\gets a(m)\Id_n-\beta A-\frac{2\beta^2}{n}mm^{\sT},
\]
at each of the two values of $m$ required for the Lipschitzness bound, omitting the rank-one term for
$\bar Q$. The required input
bounds are
\allowdisplaybreaks
\begin{align*}
 \schnorm{W(m)-W(w)}&\le\frac2{\sqrt n}\lpnorm{m-w},\\
 \schnorm{M(m)-M(w)}&\le
      \left(\frac{2\beta^2}{\sqrt n}+\frac{4\beta^2}{n}\right)
                  \lpnorm{m-w}. \qedhere
\end{align*}
\end{proof}

We skip Lemma \thesection.3 because the argument has been moved to \pref{lem:tap-convexity}.
The next lemma restarts numbering at \thesection.4 in order to maintain the same lemma numbering as in \cite[\S 6]{davies2026potential}.
\stepcounter{proposition}

\subsection{Diagonal error from rank-1 term of TAP Hessian}
\label{sec:alg-diagonal-error-rank-1}
The definition of the event $\Omega_{\beta,c}$ from \pref{cor:diagonal-Q2-controlled-in-l2} is in terms of $\bar Q$, whereas the algorithm uses
$\hat Q$. The following analog of
\cite[Lemma 6.4]{davies2026potential} bounds the difference, substantially simplifying the earlier version of the proof.
\begin{lemma}[Small diagonal error from rank-1 terms]
\label{lem:alg-resolvent-bounds-DQ2-hQ2D}
In the setting of \pref{lem:alg-resolvent-bounds}, for all
$m\in\calS_A(c)$,
\[
 \schnorm{E_{\mathcal D_n}[D(m)(\hat Q(m)^2-\bar Q(m)^2)D(m)]}
                         \le\frac{C_{\beta,c}}{\sqrt n}.
\]
\end{lemma}
\begin{proof}
Apply part (e) of \pref{lem:resolvent-difference}, keeping $W=D(m)^{-1}$
fixed and changing $M$ by $-2\beta^2mm^{\sT}/n$.
The $\min\Spec(|M + W^{-1}|) > c$ condition is satisfied at both endpoints and $\schnorm{\Delta M} = \schnorm{-2\beta^2mm^{\sT}/n} \le 2\beta^2/\sqrt n$.
Then contractivity of $E_{\mathcal D_n}$ concludes the bound.
\end{proof}

\subsection{Control of the diagonal}
\label{sec:alg-diagonal-control}
As in \cite[\S6.4]{davies2026potential}, define
\begin{equation}
\label{eq:delta-diag}
 \delta_{\diag}(m):=E_{\mathcal D_n}[\hat Q(m)^2]D(m)^2
                    -(1+\beta^2\tr_n(\hat Q(m)^2))\Id_n.
\end{equation}
Since $D$ is diagonal, the first term is also $E_{\mathcal D_n}[D\hat Q^2D]$.

\begin{lemma}[Basic diagonal properties on the good set]
\label{lem:alg-diag-basic}
In the setting of \pref{lem:alg-resolvent-bounds}, for $m,w\in\calS_A(c)$,
\allowdisplaybreaks
\begin{enumerate}[label=(\alph*),ref=\thetheorem(\alph*),itemsep=0.1em]
\item \label{lem:alg-resolvent-bounds-delta}
\[ \opnorm{\delta_{\diag}(m)}\le C_{\beta,c}. \]
\item \label{lem:alg-resolvent-lipschitz-delta}
\begin{equation}
\label{eq:alg-diagonal-difference}
 \schnorm{\delta_{\diag}(m)-\delta_{\diag}(w)}
           \le\frac{C_{\beta,c}}{\sqrt n}\lpnorm{m-w}.
\end{equation}
\end{enumerate}
\end{lemma}
\begin{proof}
The proof is the same as \cite[Lemma 6.5]{davies2026potential}, using
\pref{lem:alg-resolvent-bounds} and
\pref{lem:alg-resolvent-lipschitz}. Both the diagonal projection and
$X\mapsto\tr_n(X)\Id_n$ are contractions in the relevant norms.
\end{proof}

\begin{lemma}[Uniform diagonal control]
\label{lem:alg-resolvent-bounds-delta-frob}
\label{lem:alg-diagonal-control-safe}
In the setting of \pref{lem:alg-resolvent-bounds}, assume also that we are in the event $\Omega_{\beta,c}$.
Then for every $m\in\calS_A(c)$,
\begin{equation}
\label{eq:alg-diagonal-small}
       \schnorm{\delta_{\diag}(m)}\le\frac{C_{\beta,c}}{\sqrt n}.
\end{equation}
The same bound, with every $\hat Q$ replaced by $\bar Q$, holds for all
$m\in(-1,1)^n$ such that $D(m)\in\calS_D$.
\end{lemma}
\begin{proof}
The proof of \cite[Lemma 6.6]{davies2026potential} is unchanged after
using the definition of $\calS_A(c)$ and
\pref{lem:alg-resolvent-bounds-DQ2-hQ2D} for the initial diagonal
estimate.
For $\bar Q$, skip the application of \pref{lem:alg-resolvent-bounds-DQ2-hQ2D}.
\end{proof}

\subsection{Error between ASL-TAP and PHD is small}
\label{sec:alg-phd-asl-tap-closeness}
Recall that ASL-TAP and PHD differ by the drift $\hat Q(m)(m-f(m))$, with $f$ given in \pref{eq:alg-f-definition}.
We obtain the analog of \cite[Lemma 6.7]{davies2026potential}, showing that this drift term is small.
\begin{lemma}[Closeness of ASL-TAP and PHD]
\label{lem:alg-safe-drift}
In the setting of \pref{lem:alg-resolvent-bounds}, assuming the event $\Omega_{\beta,c}$ defined in \pref{eq:alg-def-Omega} and with $f$ as in \eqref{eq:alg-f-definition}, the following bounds hold on $m \in \calS_A(c)$ in Euclidean norm:
\allowdisplaybreaks
\begin{enumerate}[label=(\alph*),ref=\thetheorem(\alph*),itemsep=0.1em]
\item \label{lem:alg-resolvent-lipschitz-f}
\begin{equation}\label{eq:alg-f-Lipschitz}
                 \lipnorm{f}\le C_{\beta,c}.
\end{equation}
\item \label{lem:alg-resolvent-bounds-m-f}
\begin{equation}\label{eq:alg-mf-bound}
                 \lpnorm{m-f(m)}\le C_{\beta,c}.
\end{equation}
\item \label{lem:alg-resolvent-bounds-drift}
\begin{equation}\label{eq:alg-drift-bounded}
                 \lpnorm{\hat Q(m)(m-f(m))}\le C_{\beta,c}.
\end{equation}
\item \label{lem:alg-resolvent-bounds-dual-drift-1}
\[ \lpnorm{D(m)\hat Q(m)(m-f(m))}\le C_{\beta,c}. \]
\item \label{lem:alg-resolvent-lipschitz-drift}
\begin{equation}\label{eq:alg-drift-Lipschitz}
       \lipnorm{m\mapsto\hat Q(m)(m-f(m))}\le C_{\beta,c}.
\end{equation}
\end{enumerate}
\end{lemma}
\begin{proof}
Use the proof of \cite[Lemma 6.7]{davies2026potential}, with all of the references to intermediate lemmas replaced with their analogs written above, with the same section-internal numbering. The identity expressing $f$ in terms of $\delta_{\diag}(m)$ is
\[
 f(m)=\left(\Id_n+\delta_{\diag}(m)
                         -\frac{2\beta^2}{n}\hat Q(m)^2\right)m.
\]
In particular,
$\lpnorm{\delta_{\diag}(m)m}\le\sqrt n\schnorm{\delta_{\diag}(m)}$, and
the same inequality applies to the difference of the diagonal errors
acting on any $w\in(-1,1)^n$.
All other matrix differences are bounded using
$\opnorm{\cdot}\le\sqrt n\schnorm{\cdot}$.
The remaining product estimates are unchanged.
\end{proof}

\subsection{Lipschitz extension of the JE weight integrand}
\label{sec:alg-je-extension}
The next two lemmata provide the additional ingredients to
adapt the JE-weight argument in
\cite[Corollary 6.14(d5)]{davies2026potential} to a nonconvex good set.
We start by constructing short paths preserving the absolute spectral
gap of the resolvent without the rank-one term, so that derivative
bounds hold along these paths and can be used to prove Lipschitz bounds.

\begin{lemma}[Short paths preserving the spectral gap]
\label{lem:two-segment-W-path}
Fix $0<\beta<1$, $c>0$, and $A\in\mathrm{Sym}_n(\mathbb R)$.
Let $W_0,W_1\in\mathcal D_n((0,1])$ satisfy
\[
    W_j^{-1} \in \calS_D,
    \qquad j=0,1,
\]
where $\calS_D$ was defined in \pref{eq:alg-def-calSD}.
Then there exists a piecewise $C^1$ path
$W:[0,1]\to\mathcal D_n((0,1])$ with $W(0)=W_0$ and $W(1)=W_1$
such that
\begin{equation}
\label{eq:alg-baromega-W-path-gap}
    W(t)^{-1} \in \calS_D,
    \qquad 0\le t\le1,
\end{equation}
and its normalized Schatten $2$-length satisfies
\begin{equation}
\label{eq:alg-baromega-W-path-length}
    \int_0^1\schnorm{W'(t)}\,dt
    \le\frac{2}{1-\beta^2}\schnorm{W_1-W_0}.
\end{equation}
\end{lemma}

\begin{proof}
Let $\min(W_0,W_1)$ denote their coordinatewise minimum.
The path will proceed in a line segment from $W_0$ to $(1+s)^{-1}\min(W_0,W_1)$ for some scalar $s$ and then in another line segment to $W_1$.
Accordingly, denote the two line segments as
\[
    s:=\frac{\beta^2}{1-\beta^2}\schnorm{W_1-W_0},
    \qquad
    W_j(t):=(1-t)W_j+\frac{t}{1+s}\min(W_0,W_1),
    \quad 0\le t\le1,\quad j=0,1.
\]
Both segments end at $(1+s)^{-1}\min(W_0,W_1)$ and satisfy
$0\prec W_j(t)\preceq W_j\preceq\Id_n$.
We now proceed with a few path derivative calculations to show that \pref{eq:alg-baromega-W-path-gap} holds throughout the line segments.
First,
\[
    -W_j'(t)=\frac{sW_j+W_j-\min(W_0,W_1)}{1+s}
    \succeq\frac{s}{1+s}W_j.
\]
By diagonality and $0\prec W_j(t)\preceq W_j\preceq\Id_n$,
\[
    \frac{d}{dt}W_j(t)^{-1}
    =W_j(t)^{-2}(-W_j'(t))
    \succeq\frac{s}{1+s}W_j(t)^{-2}W_j
    \succeq\frac{s}{1+s}\Id_n.
\]
Since $\tr_n(W_j)\le1$, Cauchy--Schwarz for the normalized trace gives
\begin{align*}
    \tr_n[-W_j'(t)]
    &=\frac{s\tr_n(W_j)+\tr_n[W_j-\min(W_0,W_1)]}{1+s}\\
    &\le\frac{s+\schnorm{W_j-\min(W_0,W_1)}}{1+s}
     \le\frac{s+\schnorm{W_1-W_0}}{1+s}.
\end{align*}
Combining these estimates yields
\begin{align*}
    \frac{d}{dt}\left(
        W_j(t)^{-1}+\beta^2\tr_n(W_j(t))\Id_n-\beta A
    \right)
    &\;\succeq\;
    \frac{(1-\beta^2)s-\beta^2\schnorm{W_1-W_0}}{1+s}\Id_n
    \;=\;0.
\end{align*}
Integrating from $0$ to $t$ and using the $W_j^{-1} \in \calS_D$ assumption gives
\[
    W_j(t)^{-1}+\beta^2\tr_n(W_j(t))\Id_n-\beta A
    \;\;\succeq\;\; W_j^{-1}+\beta^2\tr_n(W_j)\Id_n-\beta A
    \;\;\succeq\;\; c\Id_n,
\]
showing \eqref{eq:alg-baromega-W-path-gap} for both of the two segments.
Concatenate $W_0(t)$ with $W_1(t)$ traversed in reverse, and rescale the
parameter interval to $[0,1]$ to obtain the required piecewise $C^1$
path $W(t)$.
Since $\schnorm{W_j}\le1$ and
$\schnorm{W_j-\min(W_0,W_1)}\le\schnorm{W_1-W_0}$, its length satisfies
\begin{align*}
    \int_0^1\schnorm{W'(t)}\,dt
    &=\sum_{j=0}^1\int_0^1\schnorm{W_j'(t)}\,dt\\
    &\le\frac{2\bigl(s+\schnorm{W_1-W_0}\bigr)}{1+s}\\
    &=\frac{2}{(1-\beta^2)(1+s)}\schnorm{W_1-W_0}
     \le\frac{2}{1-\beta^2}\schnorm{W_1-W_0},
\end{align*}
which proves \eqref{eq:alg-baromega-W-path-length}.
\end{proof}

\begin{lemma}[Lipschitzness of the JE integrand]
\label{lem:omega-safe-Lipschitz}
In the setting of \pref{lem:alg-safe-drift}, let
\begin{equation}
\label{eq:alg-omega}
 \omega(m):=\frac12\bigl(\Tr[\hat{Q}(m)]+\lpnorm{m}^2\bigr),
                  \qquad (\forall m\in\calS_A(c))
\end{equation}
as in \ref{d:JE-lip}.
Then there is an $L_\omega=C_{\beta,c}<\infty$ such that
\[
       |\omega(m)-\omega(w)|\le L_\omega\lpnorm{m-w}
                       \qquad(\forall m,w\in\calS_A(c)).
\]
\end{lemma}
\begin{proof}
Fix $A$. 
We will first prove the Lipschitz bound for
\[
 \bar\omega(m):=\frac12(\Tr\bar Q(m)+\lpnorm{m}^2)
       =\frac n2+\frac12(\Tr\bar Q(m)-\Tr W),
       \qquad W=D(m)^{-1},
\]
and then add the rank-one term back in at the end.
With $A$ held fixed, we use the right-hand expression to regard
$\bar\omega$ as a function of $W\in\mathcal D_n((0,1])$ wherever
the following inverse exists:
\[
 D=W^{-1},\qquad a=\beta^2\tr_n(W),\qquad
 \bar Q=\bigl(W^{-1}+\beta^2\tr_n(W)\Id_n-\beta A\bigr)^{-1}.
\]
Differentiating with respect to $W$ in a diagonal direction $H$ gives
\begin{equation}
\label{eq:alg-baromega-W-differential}
 d_W\!\left[\bar{\omega}\right][H] = 
 d_W\!\left[\frac12(\Tr\bar Q-\Tr W)\right][H]
 =\frac12\Tr\!\left[
   \left(E_{\mathcal D_n}[D\bar Q^2D]
       -(1+\beta^2\tr_n(\bar Q^2))\Id_n\right)H\right].
\end{equation}
For every $W$ satisfying \eqref{eq:alg-baromega-W-path-gap},
\pref{lem:alg-diagonal-control-safe} gives
\[
 \left\|E_{\mathcal D_n}[D\bar Q^2D]
       -(1+\beta^2\tr_n(\bar Q^2))\Id_n\right\|_F
 \le C_{\beta,c}.
\]
Cauchy--Schwarz for the unnormalized trace in
\eqref{eq:alg-baromega-W-differential} therefore yields, for all $W$ satisfying \eqref{eq:alg-baromega-W-path-gap} and diagonal $H$,
\begin{equation}
\label{eq:alg-baromega-W-derivative-bound}
 \left|d_W\!\left[\bar{\omega}\right][H]\right|
 \le C_{\beta,c}\|H\|_F.
\end{equation}
Fix $m,w\in\calS_A(c)$ and set
$W_0=D(m)^{-1}$ and $W_1=D(w)^{-1}$, which implies $W_0^{-1},W_1^{-1} \in \calS_D$.
Then let $W(t)$ be the path supplied by \pref{lem:two-segment-W-path} between those two endpoints,
so that \eqref{eq:alg-baromega-W-derivative-bound} applies along this path.
Since $\bar\omega$ depends only on $D(m)^{-1}$, integrating that bound
and using \eqref{eq:alg-baromega-W-path-length} gives
\begin{align}
    |\bar\omega(m)-\bar\omega(w)|
    &\le C_{\beta,c}\sqrt n\int_0^1\schnorm{W'(t)}\,dt \notag\\
    &\le\frac{2C_{\beta,c}\sqrt n}{1-\beta^2}\schnorm{W_1-W_0}
     \le C_{\beta,c}\lpnorm{m-w},
\label{eq:alg-baromega-endpoint}
\end{align}
where the last inequality uses
\[
    \sqrt n\schnorm{W_1-W_0}
    =\left(\sum_{i=1}^n(m_i^2-w_i^2)^2\right)^{1/2}
    \le2\lpnorm{m-w}.
\]

Finally, each $\hat Q-\bar Q$ has rank at most one, so their difference
between two endpoints has rank at most two. Therefore
\begin{align*}
 &\left|\Tr\bigl[(\hat Q(m)-\bar Q(m))
                       -(\hat Q(w)-\bar Q(w))\bigr]\right|\\
 &\quad\le\sqrt2\left(
       \|\hat Q(m)-\hat Q(w)\|_F
       +\|\bar Q(m)-\bar Q(w)\|_F\right)
       \le C_{\beta,c}\lpnorm{m-w},
\end{align*}
by \pref{lem:alg-resolvent-lipschitz-Q}, applied to $\hat Q$ and
$\bar Q$ at $m,w\in\calS_A(c)$.
Combining this with \eqref{eq:alg-baromega-endpoint} proves the claim.
\end{proof}

\subsection{Implications for the sampling algorithm}
\label{sec:alg-desiderata}

\Needspace{9\baselineskip}
\begin{corollary}[Algorithmic desiderata]
\label{cor:alg-desiderata}
Assume we are in the events $\{\opnorm{A} \le 3\}$ and $\Omega_{\beta,c}$ as defined in \pref{eq:alg-def-Omega}.
Then the following hold with constants independent
of $n$ and $A$.
\begin{enumerate}[itemsep=0.45em,label=\textup{(d\arabic*)},
                 ref={\thetheorem\,(d\arabic*)},start=0]
\setcounter{enumi}{2}
\item \label{cor:alg-desiderata-d3}
For every $m\in\calS_A(c)$,
$\lpnorm{f(m)-m}^2\le C_{\beta,c}^2$.
\item \label{cor:alg-desiderata-d4}
\begin{enumerate}[itemsep=0.1em]
\item
There is a constant $L\le C_{\beta,c}$ such that
$\|\hat Q(m)-\hat Q(w)\|_F\le L\lpnorm{m-w}$ for $m,w\in\calS_A(c)$.
        \item There exists a constant $L_{\mathrm{drift}}$  such that for all $\mg, w\in \calS_A(c_{\TAP})$, with $f$ defined as in \pref{eq:alg-f-definition},
    \[
    \ve{\hat{Q}(m)(m-f(m)) - \hat{Q}(w)(w-f(w))}_2 \le L_{\mathrm{drift}}\ve{m - w}_2.
    \]
\end{enumerate}
\item \label{cor:alg-desiderata-d5}
 Let
        \[\omega(m):=\frac12\bigl(\Tr[\hat{Q}(m)]+\lpnorm{m}^2\bigr).\]
        Then there is an $L_\omega<\infty$ such that
\[
       |\omega(m)-\omega(w)|\le L_\omega\lpnorm{m-w}
                       \qquad(\forall m,w\in\calS_A(c)).
\]
\end{enumerate}
\end{corollary}
\begin{proof}
\ppart{Parts (d3)--(d4)}
As in \cite[Corollary 6.14]{davies2026potential}, these follow in order
from \pref{lem:alg-resolvent-bounds-m-f},
\pref{lem:alg-resolvent-lipschitz-Q},
and \pref{lem:alg-resolvent-lipschitz-drift}.
.
For (d4), convert normalized Schatten $2$-norm to Frobenius norm.

\ppart{Part (d5)} This is \pref{lem:omega-safe-Lipschitz}. \qedhere

\end{proof}

\section{Sampling from the time-\texorpdfstring{$T$}{T} localized distribution}\label{sec:sample-localized-distribution}

Using a result of Kumar, Sarkar, Tian and Zhu \cite[Theorem 4 and Corollary 5]{kumar2026high} we give a thresholding time $T(\beta) > 0$ to inform how long the ASL-TAP-JE procedure should run to obtain a warm-start for the annealing-type sampler in \cite[Algoirithm 2]{kumar2026high} run on a small wedge around a fixed tilt $y_T$. This samples from the time-$T(\beta)$ localized distribution in $\poly(n)$-time with $o_n(1)$-TVD error. This approach replaces the polarized--walk based sampler for the time-$T$ localized distribution that was used in \cite[\S 7]{davies2026potential}. The reason for this replacement is that the prior approach uses the needle-decomposition \cite{eldan2022spectral} before an application of the HS transform, which introduces a shift of $2\beta\Id_n$ to make the interactions PSD, after which a union bound over PSD sub-matrices constrained to have small-enough operator norm forces an upper bound of $\beta < 1/2$ due to the shift. The approach of Kumar, Sarkar, Tian and Zhu bypasses the obstacle in the prior approach of the authors \cite[\S 7.2]{davies2026potential} by reasoning about a Dobrushin-type condition straight away on the wedge around the tilt $y_T$ (which the localized measure will concentrate on with high probability for any $\beta > 0$, see \cite[\S 7]{davies2026potential}).

We invoke \cite[Corollary 5]{kumar2026high} with $J := A$\footnote{\,Note that since we are fixing the plant $y_T$, we don't need to invoke contiguity to the planted model. Since we will choose $T$ large enough so that the localized measure concentrates on a desirably small wedge, we will argue about sampling in that edge \emph{uniformly} across the choice of tilts $\sigma_0 = \sign(y_T)$.} and  $h := y_{T(\beta)}=T(\beta)x_0+\sqrt{T(\beta)}\,B$. We know that with probability $1 - e^{-c(\beta)n}$ over the SL process, the localized measure concentrates on a wedge $B(\sign(y_T), \ep(\beta,T)n)$ \cite[Lemma 7.43]{davies2026potential}. The following theorem then shows efficient sampling for this wedge.

\begin{theorem}%
[Sampling on small wedges via {\cite[Algorithm 2]{kumar2026high}}]
\label{thm:localized-sampler}
Fix $\be>0$, $\eps(\beta) = \frac{1}{C\bar{\beta}^2\log(e\bar{\beta})}$ with $C>0$ some large enough universal constant and $\bar{\beta} = \max\{1,\beta\}$, and $k\le \eps(\beta)n$.
Then, with probability $1-o_n(1)$, 
for any $ 0 < \beta' < \beta$ and $0<\varepsilon < 1/2$, uniformly over $y\in \R^n$, letting $\si_0=\sign(y) \in \{-1,1\}^n$, \cite[Algorithm 2]{kumar2026high} run for $2\log\left(\frac{16}{\eps^2(\beta')}\right)$ annealing levels, with $\poly(n)\log(1/\ep)$ Markov chain steps for each level,
gives a sample that is $\ep$ close in TV distance from $\mu_{\be' A, y}|_{B(\si_0, \ep n)}$.
\end{theorem}

\begin{proof}
    Define the diagonal matrix $R_{y} = \diag(\sign(y))$ and set $x = -R_{y}\sigma$ for any $\sigma \in B(\sign(y),\eps(\beta')n)$. Clearly, then, $\sigma$ can be in the wedge around $y$ if, and only if, the number of positive signs in $x$ are no more than $\eps(\beta')n$. This bijection between centering on $y$ and $0^n$ allows the measure on the wedge to be rewritten as
    \[
        \mu_{y}(x) \propto \exp\left\{\frac{\beta}{2}\left(x^\sT(R_y JR_y)x - (R_yy)^\sT x\right)\right\}\,,
    \]
    where we use the distributional equivalence $R_yJR_y \overset{d}{=} \mathsf{GOE}(n)$. Denote the new parameters $J = \frac{1}{2}R_y AR_y$ for $A\sim\mathsf{GOE}(n)$ and $h = -\frac{2R_yy}{\beta}$. Note that the \cite[Proof of Theorem 7]{kumar2026high} applies verbatim with the interaction matrix $R_y J R_y$ due to the distributional equivalence before since the Dobrushin-type bounds based on \cite[Lemma 8]{kumar2026high} remain unchanged. Note that we choose $\eps(\beta) = \frac{1}{C\bar{\beta}^2\log(e\bar{\beta})} \le \frac{1}{C\bar{\beta'}^2\log(e\bar{\beta'})} = \eps(\beta')$ so that the validity of \cite[Theorem 4]{kumar2026high} holds when invoking \cite[Corollary 5]{kumar2026high}, and we can get the mixing to hold simultaneously on the same event for $A$ for every $0<\beta'<\beta$. 
\end{proof}

\newpage

\addtocontents{toc}{\protect\setcounter{tocdepth}{-1}}

\section*{Acknowledgements}
\addtocontents{toc}{\protect\setcounter{tocdepth}{1}}

\iffocs{}{
    JSS thanks David Jekel for introducing him to the coefficient-matrix representation for Gaussian random matrices used in the work of Bandeira, Boedihardjo and van Handel~\cite{bandeira2023matrix}, the operator-valued convergence established in the work of Jekel, Lee, Nelson and Pi~\cite{jekel2025strong}, and the idea of the block matrix representation used in  \pref{eq:block-trace-functional} which led to the simplification of the free interpolation in \pref{sec:option-c-free-comparison} compared to the authors' original approach based on \cite[\S B.6]{davies2026potential}. 
}

\textbf{AI Disclosure:} We used generative AI to assist with verification. 
AI was used to help choose from different valid choices of \pref{l:je-survival}, \pref{eq:rational-reciprocal-cutoff}, and short-path constructions in \pref{lem:two-segment-W-path}, to aid the authors in selecting the most readable options.

\textbf{Declaration:} The authors of this article do not consent to having the contents of this article be used for the training or development of LLMs/AI models.

\vspace{-1mm}

\addtocontents{toc}{\protect\setcounter{tocdepth}{-1}}

{
    \small\hypersetup{urlcolor=Black}
    \bibliographystyle{alpha_beta_doi}
    \bibliography{main.bib}
}

\addtocontents{toc}{\protect\setcounter{tocdepth}{1}}

\newpage 

\appendix
\normalsize

\section{Resolvent product regularity and subordination at the TAP value}\label{app:deformed-wigner-resolvent}
This section has two goals -- the first is a series of bounds on finite differences of matrices made of resolvent products, and the second is a proof of a spectral gap with respect to a particular value for the free convolution of a semicircular operator and a diagonal operator with explicit empirical density. 
\subsection{Regularity estimates for resolvent products} 
\newcommand{\tM}{\widetilde{M}}
\newcommand{\tW}{\widetilde{W}}
The bounds here are finite-difference analogues of those presented in \cite[Lemma B.4]{davies2026potential}. They are quantitative estimates for finite differences between various matrices involving products of resolvents, when the minimum of the resolvent's spectrum is bounded to be greater than $0$.
\begin{lemma}[{Bounds on differences of resolvent products (analog of \cite[Lemma B.4]{davies2026potential})}]
    \label{lem:resolvent-difference}
Let $W_0,W_1\in\R^{n\times n}$ be invertible symmetric matrices,
and let $M_0,M_1\in\C^{n\times n}$.
Suppose that, for $j=0,1$,
\[
    \min\Spec(|M_j+W_j^{-1}|)\ge\gamma,
    \qquad
    \opnorm{M_j}\le C_1,
\]
where $\gamma>0$ and $C_1\ge0$.
Let $C_2:=(\gamma+C_1)\gamma^{-1}$.

    Let $D(W) := W^{-1}$, $X(W,M) := M + D(W)$, and $R(W,M) := X(W,M)^{-1}$ whenever the stated inverses exist.
    
    Let $\Delta$ be the difference operator operating on functions of $W$ and $M$, so that $\Delta f := f(W_1, M_1) - f(W_0,M_0)$.
    Let $T \in \C^{n \times n}$ be fixed and take all exponents $k$ below to be integers.
    Then:
    \begin{enumerate}[label=(\alph*), ref=\thetheorem(\alph*)]
        \item \[\Lpnorm{\Delta R} \le \gamma^{-2}\Lpnorm{\Delta M} + C_2^2\Lpnorm{\Delta W}.\]
        \label{lem:resolvent-difference-dR}
        \item For $k \ge 1$, \[\Lpnorm{\Delta(R^k)} \le k\gamma^{-(k+1)}\Lpnorm{\Delta M} + kC_2^2\gamma^{-(k-1)}\Lpnorm{\Delta W}.\]
        \label{lem:resolvent-difference-dRk}
        \item \[\Lpnorm{\Delta(DR)} \le C_2\gamma^{-1}\Lpnorm{\Delta M} + C_1C_2^2\Lpnorm{\Delta W}.\]
        \label{lem:resolvent-difference-dDR}
        \item For $k \ge 1$,
        \[\Lpnorm{\Delta(DR^k)} \le kC_2\gamma^{-k}\Lpnorm{\Delta M} + (kC_2-1)C_2^2\gamma^{-(k-2)}\Lpnorm{\Delta W}.\]
        \label{lem:resolvent-difference-dDRk}
        \item For $k \ge 2$, \[\Lpnorm{\Delta(DR^kD)} \le kC_2^2\gamma^{-(k-1)}\Lpnorm{\Delta M} + (kC_2-2)C_2^3\gamma^{-(k-3)}\Lpnorm{\Delta W}.\]
        \label{lem:resolvent-difference-dDRkD}
    \end{enumerate}
\end{lemma}
\begin{proof}
    For $f$ a matrix-valued function of $W$ and $M$, we define the evaluation operators $f_1 := f(W_1, M_1)$ and $f_0 := f(W_0, M_0)$ so that $\Delta f = f_1 - f_0$.
    Then we have the discrete non-commutative sum, product, power, and inverse rules.
 Explicitly,
 \begin{itemize}
     \item $\Delta(f+g) = \Delta f + \Delta g$,
     \item $\Delta(fg) = f_1(\Delta g) + (\Delta f)g_0$,
     \item $\Delta(f^k) = \sum_{i = 1}^{k}f_1^{i-1}(\Delta f)f_0^{k-i}$,
     \item $\Delta(f^{-1}) = -f_1^{-1}(\Delta f)f_0^{-1}$ (also known as the resolvent identity).
 \end{itemize}

    Note by the resolvent identity, $\Delta D = -D_1(\Delta W)D_0$ and $\Delta X = \Delta M - D_1(\Delta W)D_0$.
    By the argument of \cite[Lemma B.1]{davies2026potential}, we have $\opnorm{R^k} \le \gamma^{-k}$ and $\opnorm{DR^{k+1}} \le C_2\gamma^{-k}$, which we will use in every following part.
    We will also repeatedly use the non-commutative H\"older inequality.

    \ppart{Part \ref{lem:resolvent-difference-dR}} We apply the just-mentioned operator norm bounds after computing the derivative
    \[\Delta R = -R_1(\Delta X)R_0 = -R_1(\Delta M)R_0 + (RD)_1(\Delta W)(DR)_0.\]

    \ppart{Part \ref{lem:resolvent-difference-dRk}}
    We use the discrete non-commutative power rule
    \[\Delta(R^k) = \sum_{j=1}^{k} R_1^{j-1}(\Delta R)R_0^{k-j},\]
    then apply \pref{lem:resolvent-difference-dR} and the same operator norm bounds from before.

    \ppart{Part \ref{lem:resolvent-difference-dDR}}
    By the resolvent identity, $DR = \Id - MR$. Hence $\Delta(DR) = - M_1(\Delta R) - (\Delta M)R_0 $. then apply \pref{lem:resolvent-difference-dR} and $\opnorm{M} \le C_1$ and $C_2 = 1 + C_1\gamma^{-1}$.

    \ppart{Part \ref{lem:resolvent-difference-dDRk}}
    Use the discrete non-commutative product rule to compute 
    \[\Delta(DR^k) = (DR)_1\Delta(R^{k-1}) + \Delta(DR)R_0^{k-1}.\]
    Then apply \pref{lem:resolvent-difference-dDR} and \pref{lem:resolvent-difference-dRk} and, for the $\Lpnorm{\Delta W}$ term, $C_1C_2^2\gamma^{-(k-1)} + (k-1)C_2^3\gamma^{-(k-2)} = (kC_2-1)C_2^2\gamma^{-(k-2)}$.

    \ppart{Part \ref{lem:resolvent-difference-dDRkD}} 
    Use the discrete non-commutative product rule to compute 
    \[\Delta(DR^kD) = (DR)_1\Delta(R^{k-1}D) + \Delta(DR)(R^{k-1}D)_0,\]
    then apply \pref{lem:resolvent-difference-dDR} and \pref{lem:resolvent-difference-dDRk} and collect terms.
\end{proof}

\subsection{Free analysis at the TAP value of $a=\beta^2\tr_n[D^{-1}(m)]$} The following deterministic lemma gives the diagonal control and spectral gap in the free limit used in \pref{cor:diagonal-Q2-controlled-in-l2}. For the following lemma, $D$ and $S$ are free from each other.

\begin{lemma}[The TAP free limit for $\beta<1$]
\label{lem:option-c-tap-free-endpoint}
Let $0<\beta<1$, $D\succeq\Id_n$, and $a=\beta^2\tr_n(D^{-1})$. Then
\begin{equation}
\label{eq:option-c-tap-free-gap}
    a\Id_n-\beta S+D\succeq(1-\beta)^2\Id_n.
\end{equation}
For all
\begin{equation}
\label{eq:option-c-tap-cutoff}
    0<\gamma_Y\le\frac{(1-\beta)^4}{4(1+2\beta^2)^2},
\end{equation}
it is true uniformly in $D\succeq\Id_n$ that
\begin{equation*}
    (\Id_n+D^{-1}(a\Id_n-\beta S))(\Id_n+(a\Id_n-\beta S)D^{-1})
       \succeq4\gamma_Y\Id_n.
\end{equation*}
For every such $D$ and $a$,
\[
    E_{M_n}\!\left[D(a\Id_n-\beta S+D)^{-2}D\right]
       =\frac{\Id_n}{1-\beta^2\tr_n(D^{-2})}.
\]
\end{lemma}

\begin{proof}
We first show that the range of the subordination function allows one to conclude a gap between $a=\beta^2\tr_n[D^{-1}]$ and the supremum of the bulk spectrum of the free convolution of $\beta S$ and $-D$ when $\beta < 1$. We then use the inverse identity and analyticity of the subordination function at $z=a$ to explicitly evaluate $F$ and its derivative to compute the free target.
\ppart{The free gap} 
The real inverse-subordination map for $\beta S-D$ \cite[Lemma B.18]{davies2026potential} is
\[
    u\longmapsto u+\beta^2\tr_n[(D+u\Id_n)^{-1}].
\]
For $u>\beta-1$, it is defined to the right of $\Spec(-D)$, and its derivative is strictly positive since
$\beta^2\tr_n[(D+u\Id_n)^{-2}]<1$.
The support characterization in
\cite[Proposition 2.2]{capitaine2011free} identifies this interval with part of the right exterior branch. Letting $u\downarrow\beta-1$ gives
\[
    \sup\Spec(\beta S-D)
       \le\beta-1+\beta^2\tr_n[(D+(\beta-1)\Id_n)^{-1}].
\]
Indeed, the image is an interval connected to $+\infty$ and disjoint from the free support. Subtracting this upper bound from $a$ yields
\begin{align*}
    a-\sup\Spec(\beta S-D)
    &\ge(1-\beta)\left[
       1-\beta^2\tr_n\!\left(D^{-1}(D+(\beta-1)\Id_n)^{-1}\right)
       \right]\\
    &\ge(1-\beta)^2,
\end{align*}
because $D^{-1}(D+(\beta-1)\Id_n)^{-1}\preceq\beta^{-1}\Id_n$.
Faithfulness of the trace identifies the free support with the operator spectrum, proving \eqref{eq:option-c-tap-free-gap}.

\ppart{The product cutoff and the target}
The inverse identity and the free gap imply
\begin{align*}
    \opnorm{(\Id_n+D^{-1}(a\Id_n-\beta S))^{-1}}
       &=\opnorm{(a\Id_n-\beta S+D)^{-1}D}\\
       &\le1+\frac{a+2\beta}{(1-\beta)^2}
        \le\frac{1+2\beta^2}{(1-\beta)^2}.
\end{align*}
Thus the free product is at least
$(1-\beta)^4(1+2\beta^2)^{-2}\Id_n$, and
\eqref{eq:option-c-tap-cutoff} implies the required buffer.

Let $F=F_{\beta,-D}$ be the subordination function. The inverse-subordination map above has positive derivative at $u=0$ and takes $0$ to $a$. This implies its (exterior) inverse satisfies
\[
    F(a)=0,\qquad
    F'(a)=\frac{1}{1-\beta^2\tr_n(D^{-2})}.
\]
Subordination under a non-commutative conditional expectation \cite[Theorem 3.1]{biane1998processes} gives $E_{M_n}[(z\Id_n-\beta S+D)^{-1}]=(F(z)\Id_n+D)^{-1}$. Note that $z=a$ is outside the free spectrum by \eqref{eq:option-c-tap-free-gap}, and so $\partial_z F(z)$ exists at $z=a$. Evaluating the derivative and conjugating by $D$ proves the asserted free target. %
\end{proof}

\section{Self-stability for the cavity interpolation of \texorpdfstring{\cite[\S 5]{davies2026potential}}{} at \texorpdfstring{$\beta < 1$}{beta < 1}}\label{sec:cavity-interpolation}
\paragraph{Extension of cavity interpolation theory in \cite[\S 5]{davies2026potential} to $\beta < 1$}
To see that desideratum (1) holds for the full range $ \beta < 1$ (as opposed to the stated bound of $\beta < 1/2$ in \cite{davies2026potential}) an inspection of the arguments in \cite[\S 4,\,\S 5,\,\S A]{davies2026potential} make it clear only two statements need to be strengthened:
\begin{enumerate}[itemsep=0.2em]
    \item \textbf{Nishimori condition up to $\beta < 1$:} The fixed-point equations for $(m^*,q^*)$ stated in \cite[Theorem 5.1]{davies2026potential} have unique solutions for $\beta < 1$ which satisfy the Nishimori condition that $m^* = q^*$. 
    \item \textbf{Spectral self-stability for $\mathbf M$ up to $\beta < 1$:} The operator $\mathbf M$ in \cite[\S 5]{davies2026potential} should satisfy the spectral condition that $\beta^2\rho(\mathbf M) < 1$ for every $\beta < 1$, as opposed to $\beta < 1/2$ as shown in \cite[\S 5.5]{davies2026potential}. 
\end{enumerate}
As it turns out, both these conditions are rather elementary to show. The first condition is a direct consequence of \cite[Lemma 2.1]{li2026overlap}. The second condition follows by avoiding the convenient (but lossy) step of using the Frobenius norm to bound the operator norm in the proof of \cite[Lemma 5.31]{davies2026potential} and replacing it with the use of the Nishimori identity ($m^*=q^*$) to simplify the entries of operator $\mathbf X$ followed by conjugation with a diagonal matrix to obtain a similar matrix $\mathbf K(z)$ which is anti-symmetric and whose largest singular value is bounded by $1$. This leads to the bound $\rho(\mathbf M) < 1$ from which it straightforwardly follows that $\beta^2\rho(\mathbf M) < 1$ for $\beta < 1$. 

\begin{corollary}[Overlap and magnetization concentration for $\beta < 1$]\label{cor:overlap-mag-conc}
    For $0<\beta<1$ and any fixed time $0 \le t \le T(\beta) < \infty$, and every positive integer $k \in \Z_{+}$, 
    \[
        \E_{A,B_t}\an{(R_{12}-q^*)^{2k}}_t \le \frac{C(k)}{n^k} + C(\beta,k)e^{-c(k,t)n}\,,
    \]
    and
    \[
        \E_{A,B_t}\an{(M-m^*)^{2k}}_t \le \frac{C'(k)}{n^k} + C'(\beta,k)e^{-c(k,t)n}
    \]
\end{corollary}
\begin{prf}
    This follows straightforwardly after invoking \cite[Lemma 2.1]{li2026overlap}, which implies that there is a unique pair $(m^*,q^*) \in [0,1)^2$ that solve the fixed-point equations of \cite[Theorem 5.1]{davies2026potential} for $\beta < 1$. An immediate consequence of this, is that the proof of \cite[Theorem 5.2]{davies2026potential} implies magnetization concentration for $\beta < 1$ as well. For overlap concentration, note that the entire argument of \cite[Theorem 5.1]{davies2026potential} which verifies requirements (1) through (7) in the proof are valid provided that $q^*(\lambda) = q^*$ and $\beta^2 \le \lambda$ even when $\beta < 1$. For the first, we replace the use of \cite[\S 5.6]{davies2026potential} with \cite[Lemma 2.1]{li2026overlap}, and for the latter, observe that $F(\lambda,q)$ \cite[\S 2.3]{LM19} continues to have a unique supremum over $q$ for $\lambda \le 1$. \qedhere
\end{prf}

\begin{lemma}[$\beta^2\rho(\mathbf M) < 1$ for $\beta < 1$, {\cite[Minor generalization of \S 5.5]{davies2026potential}}]\label{lem:spectral-self-stability}
    Let $Y^*$ be a random variable that is almost-surely finite. Set $T:=\tanh(Y^*)\in[-1,1]$.  Define
    \[
        \mu_r:=\E[T^r],\qquad r\in\{1,2,3,4\},
    \]
    and
    \[
        \mathbf M:=\begin{pmatrix}
            1-4\mu_2+3\mu_4 & 2(\mu_1-\mu_3)\\
            \mu_3-\mu_1 & 1-\mu_2
        \end{pmatrix}.
    \]
    Then for every $\beta < 1$,
    \[
        \rho(\beta^2\mathbf M)\;<\;1.
    \]
\end{lemma}
\begin{prf}
The proof of \cite[Lemma 5.31]{davies2026potential} implies that, with $T = \tanh(Y^*) \in [-1,1]$,
\begin{align*}
    \mathbf{M} = \E[\mathbf{X}(T)]\,,
\end{align*}
where
\[
    \mathbf{X}(z) := (1-z^2)\begin{pmatrix}
        1-3z^2 & 2z \\
        -z & 1
    \end{pmatrix}\,.
\]
Set $\mathbf{S}:=\diag(1,\sqrt{2})$ and define 
\begin{align*}
    \mathbf{K}(z)
    &:= \mathbf{S}\mathbf{X}(z)\mathbf{S}^{-1} = (1-z^2) \begin{pmatrix}
        1-3z^2 & \sqrt{2}\,z \\
        -\sqrt{2}\,z & 1
    \end{pmatrix}\,,
\end{align*}
which is similar to $\mathbf X(z)$. We show that $\|\mathbf{K}(z)\|_{\mathrm{op}}\leq 1$ for every $z\in[-1,1]$ by proving that the slack matrix $\mathbf R(z) := \Id - \mathbf K(z)^\sT \mathbf K(z) \succeq 0$. To see this, let
$u:=z^2\in[0,1]$ and note that a direct computation gives
\begin{align*}
    \mathbf R (z)_{22} &= u^2(3-2u)\,, \\
    \det\left(\mathbf R (z)\right) &= u^4(u^2-6u+6)\,.
\end{align*}
For $0<u\leq 1$, both quantities are positive, since
$3-2u\geq 1$ and $u^2-6u+6\geq 1$. The Schur-complement criterion, therefore, gives $\mathbf{R}(z)\succeq 0$ (since $\mathbf{K}(0)=I_2$ and $\mathbf{R}(0)=0$). This yields 
\[
    \mathbf{K}(z)^{\top}\mathbf{K}(z)\preceq I_2\,.
\]
$\mathbf{S}$ is deterministic, and convexity of the operator norm gives
\begin{align*}
    \|\mathbf{S}\mathbf{M}\mathbf{S}^{-1}\|_{\mathrm{op}}
    = \|\E[\mathbf{K}(T)]\|_{\mathrm{op}} \leq_{\text{Jensen's}} \E\bigl[\|\mathbf{K}(T)\|_{\mathrm{op}}\bigr]
    \leq 1\,.
\end{align*}
Finally, similarity preserves the spectral radius, giving
\[
    \rho(\beta^2\mathbf{M}) = \beta^2\rho(\mathbf{S}\mathbf{M}\mathbf{S}^{-1}) \leq \beta^2 \|\mathbf{S}\mathbf{M}\mathbf{S}^{-1}\|_{\mathrm{op}} \leq \beta^2 <_{\beta < 1} 1\,. \qedhere
\]
\end{prf}

We now state a lemma which uses \pref{cor:overlap-mag-conc} and \pref{lem:spectral-self-stability} to replace invocations of overlap and magnetization concentration, and the validity of the cavity estimates of \cite[\S 5]{davies2026potential}, in the proof of \cite[Theorem 4.1]{davies2026potential}, to work all the way for $\beta < 1$. Consequently, the same covariance estimate as in \cite[Theorem 4.1]{davies2026potential} is implied for the range of $0 < \beta < 1$.

\begin{theorem}[Error of covariance estimate for $\beta < 1$]\label{thm:covar-estimate}
    Assume overlap and magnetization concentration as given in \pref{cor:overlap-mag-conc} and the spectral radius bound for the self-stability operator $\mathbf M$ defined in \pref{lem:spectral-self-stability}. Then, for $0 < \beta < 1$,
    \[
        \E_{A,B_t}\norm{\hat{Q}^{-1}(m_t)Q(m_t) - \Id_n}_F^2 \le \eps(\beta,t)\, ,
    \]
    for every $t \in [0,T(\beta)]$ with $\sup_{t\in[0,T(\beta)]}\eps(\beta, t) < \infty$.
\end{theorem}
\begin{prf}
    The proof follows by observing that, in the proof of \cite[Theorem 4.1]{davies2026potential}, all estimates in \cite[Lemmata 4.11--4.16]{davies2026potential} rely only on, after application of Stein's lemma and algebraic identities involving independent replicas, three things:
    \begin{enumerate}[itemsep=0.2em]
        \item Magnetization and overlap concentration (cf. \cite[Theorems 5.1--5.2]{davies2026potential}).
        \item Enhanced concentration for linear and cubic observables of bulk deviations (cf. \cite[\S 5.4]{davies2026potential}).
        \item Exact cavity estimates for the ``heavy'' cancellations (cf. \cite[\S 5.4]{davies2026potential}).
    \end{enumerate}
    The first point is directly valid for $0 < \beta < 1$ due to \pref{cor:overlap-mag-conc}. The second and third point rely on the fact that magnetization and overlaps have first moments which fluctuate on the $O(1/n)$ scale \cite[\S 5.4]{davies2026potential} and that various denominators in the second-order Taylor expansions used to compute exact cavity estimates via self-reducibility do not diverge \cite[\S 5.4]{davies2026potential}. A simple inspection of the argument in \cite[\S 5]{davies2026potential} with the strengthened spectral radius bound $\beta^2 \rho\mathbf M < 1$ and the fact that $\beta^2 \E_{Y*}[\sech^{2k}(Y^*)] < 1$ for $0 < \beta < 1$ immediately implies the validity of all estimates in \cite[\S 5]{davies2026potential}. \qedhere 
\end{prf}

\paragraph{Nishimori identity for $\E[\tanh^3(Y^*)] = \E[\tanh^4(Y^*)]$} We write a brief identity that shows the Nishimori condition for the third and fourth moments in the presence of a replica--symmetric cavity field $Y^*$ that is induced by the SL process.

\begin{lemma}[Nishimori identity for $3^{\text{rd}}$ and $4^{\text{th}}$ moments]\label{lem:nishimori-third} 
    Fix $t \ge 0$ and let $z \sim \calN(0,1)$. Denote $Y^* = \beta^2m^* + t + \sqrt{\beta^2q^* + t}\,h$ where $(m^*,q^*)$ solve the fixed-point equations in \cite[Theorem 5.1]{davies2026potential}. Then, 
    \[
        \E_z\left[\tanh^3(Y^*)\right] = \E_z\left[\tanh^4(Y^*)\right]\,.
    \]
\end{lemma}
\begin{prf}
    By \cite[Lemma~2.1]{li2026overlap}, we have $m^*=q^*$ for $0<\beta<1$. Following the same notation and strategy as in \cite[Lemma 2.1]{li2026overlap} and \cite[\S 5.6]{davies2026potential}, let $r:=t+\beta^2q^*\geq0$ so that $Y^*\sim\mathcal{N}(r,r)$. If $r=0$, then $Y^*=0$ almost surely
    and both expectations vanish. Therefore, assume that $r>0$ for the rest of the proof.
    
    Let $p_r$ denote the density of $Y^*$. As in the proof of \cite[Lemma 2.1]{li2026overlap} a direct computation gives
    \begin{align*}
        p_r(-y)&=e^{-2y}p_r(y)\,.
    \end{align*}
    Since $\tanh(y)=(1-e^{-2y})/(1+e^{-2y})$, it follows that
    \begin{align*}
        p_r(y)-p_r(-y) &=\tanh(y)\bigl(p_r(y)+p_r(-y)\bigr)\,.
    \end{align*}
    Pairing the contributions from $y$ and $-y$, and using the oddness
    of $\tanh^3$ and the evenness of $\tanh^4$, yields
    \begin{align*}
        \E[\tanh^3(Y^*)]
        &=\int_0^\infty \tanh^3(y)
          \bigl(p_r(y)-p_r(-y)\bigr)\,dy \\
        &=\int_0^\infty \tanh^4(y)
          \bigl(p_r(y)+p_r(-y)\bigr)\,dy \\
        &=\E[\tanh^4(Y^*)]\,. \qedhere
    \end{align*}
\end{prf}

\section{Algorithmic error analysis}
\label{s:alg-error}

The following is an algorithm and standard guarantee for rejection sampling with unknown normalizing constant.
\begin{algorithm}[!ht]
\caption{Approximate rejection sampler with unknown normalization %
}
\begin{algorithmic}[1]
\INPUT Oracle for sampling from $\td Q\approx Q$, (possibly randomized) function $\td R\approx R$ where $\dd{P}{Q} \propto R$, cutoff parameters $C_1$, $C_2$, failure probability $\de'$ %
\State Draw $N$ samples from $\td Q$, where $N=\Om(C_1^2\log(1/\de'))$ for appropriate constant.
\State Let $R_0$ be the $p$th quantile of the samples, where $p=1-\rc{6C_1}$. 
\Repeat{}
    \State Draw $X\sim \td Q$.
    \State Draw $U\sim \mathsf{Uniform}([0,1])$.
\Until{$U\le \fc{\td R(X)}{C_3R_0}$, where $C_3=8C_2$.}
\OUTPUT $X$.
\end{algorithmic}
\label{a:ars}
\end{algorithm}
\begin{lemma}[{Approximate rejection sampling with unknown normalization, \cite[Lemma 3.8]{davies2026potential}}]
\label{l:ars}
Suppose that \pref{a:ars} is run with $\td Q$, $\td R$, and parameters $C_1, C_2$  satisfying the following.
\begin{enumerate}
    \item \label{i:esm}
    (Sampling error) $\td Q$ is a distribution such that $\TV(\td Q, Q)\le \esm$.
    \item \label{i:ewt}
    (Tails of ratio)
    Let 
    \[
    C(\ep) = \inf_C \set{C}{\E_Q\pa{\dd PQ \wedge C}\ge 1-\ep}, 
    \]
    and suppose $C_1\ge C(1/2)\vee 1$, $C_2\ge C(\ewt)$. 
    \item \label{i:ert}
    (Ratio error) 
    Let $G_1$, $G_2$ be events such that 
    \begin{align*}
    G_1 &\subeq \bc{\fc{\td R}{R}\in [0,e^{\ert}]}\\
G:&= G_1\cap G_2 \subeq 
\bc{ \fc{\td R}{R} \in [e^{-\ert}, e^{\ert}] }
    \end{align*}
    and $\td Q(G_1^c)\le \de_1$, $P(G_2^c)\le \de_2$. (Note these probabilities can involve randomness in the algorithm.)
\end{enumerate}
Then letting $\hat P$ be the output distribution, with probability $\ge 1-\de'$ over the $N$ initial samples,
\[
\TV(\hat P, P) \le \ewt + \de_2 + 384C_1C_2(\de_1 + \ert + \esm) %
\]
and if 
$\ewt + \de_2 + 96C_1C_2(\de_1 + \ert + \esm)\le \rc2$, the acceptance probability is at least $\rc{192C_1C_2}$.
\end{lemma}

\begin{prf}[Proof of \pref{l:je-rs}]
    Define the good sets on path space (1) the ratio between the ideal and approximate distribution is bounded, (2) the ratio between the scaffolding and approximate distribution is not too small, and (3) the weights are bounded,
    \begin{align*}
        G_1 &= \bc{\dd{\bP_T^s}{\hat{\bP}_T}(\ga_T)\le L}&
        G_2 &= \bc{\fc{e^{\hw_T}}{\hat Z_T}\ge c_3}&
        G_3 &= \bc{\fc{e^{\hw_T}}{\hat Z_T}\le C_1}
    \end{align*}
    and let their complements be $B_1,B_2,B_3$, respectively. 
    Now
    \begin{multline}
        \E\ba{\fc{e^{\hw_T}}{\hat Z_T} \fc{R_2(\hx_T)}{Z_2} \one_{S_T\cap (B_1\cup B_2 \cup B_3)}}\\
        \le 
        \E\ba{\fc{e^{\hw_T}}{\hat Z_T}\fc{R_2(\hx_T)}{Z_2}\one_{S_T\cap B_1} }
        + \E\ba{\fc{e^{\hw_T}}{\hat Z_T}\fc{R_2(\hx_T)}{Z_2}\one_{S_T\cap G_1\cap B_2} }
        + \E\ba{\fc{e^{\hw_T}}{\hat Z_T}\fc{R_2(\hx_T)}{Z_2}\one_{S_T\cap G_1\cap G_2\cap B_3} }.
        \label{e:3-bad}
    \end{multline}
    We bound each term separately. 
    Note that if $\mu\ll \nu$, then 
    \begin{align}
    \nonumber
        \KL(\mu\|\nu)& = \E_\mu \pa{\log \dd{\mu}{\nu}}\one_{\dd{\mu}{\nu}\ge 1} + \E_\mu\pa{\log \dd{\mu}{\nu}} \one_{\dd{\mu}{\nu}< 1} 
        \ge  \E_\mu \pa{\log \dd{\mu}{\nu}}\one_{\dd{\mu}{\nu}\ge 1} - 1
        \ge \mu\pa{\dd{\mu}{\nu}\ge L}(\log L)-1\\
        \implies
        \mu\pa{\dd{\mu}{\nu}\ge L} & \le \fc{\KL(\mu\|\nu) + 1}{\log L}.
        \label{e:kl-markov}
    \end{align}
    To apply to our setting, we define the stopped process by
    \begin{align*}
        dx_t\stp &= \one_{t<\tau'} f_t\,dt + dB_t 
    \end{align*}
    and similarly define $\hx_t\stp$. Let their laws up to time $T$ be $\bP_t\stp$ and $\hat \bP_t\stp$, respectively. 
    Then
    \begin{align}
    \nonumber
        \KL(\bP_T\stp \|\hat \bP_T\stp )
        & = \E_{\bP_T\stp} \log \dd{\bP_T\stp}{\hat \bP_T\stp}\\
    \nonumber
        &= -\E  
    \int_0^{T\wedge \tau'} \ip{\hat f_t - f_t}{dB_t} + \rc 2 \E
    \int_0^{T\wedge \tau'} \ve{\hat f_t - f_t}^2dt\\
    &= \rc 2 
    \int_0^{T\wedge \tau'} \E\ve{\hat f_t - f_t}^2\,dt \le \rc 2 \int_0^T\ep(t)^2\,dt.
    \label{e:KL-stop}
    \end{align}
    by assumption 1 and Girsanov's theorem.
    
    For the first term of \eqref{e:3-bad}, noting that $x_{[0,T]}= x_{[0,T]}\stp$ if $T<\tau'$ so that $\bP_T^s = \bP_T\stp$ restricted to $\{T<\tau'\}$,
    \begin{align*}
        \E\ba{\fc{e^{\hw_T}}{\hat Z_T}\fc{R_2(\hx_T)}{Z_2}\one_{S_T\cap B_1} }
        &= \E \ba{\dd{\bP_T^s}{\hat{\bP}_T}(\hx_{[0,T]})\one_{S_T\cap B_1}(\hx_{[0,T]})}= \E\ba{\one_{S_T\cap B_1}(x_{[0,T]})}\\
        &\le \P\ba{\dd{\bP_T^s}{\hat\bP_T}(x_{[0,T]})>L}
        \le \P\ba{\dd{\bP_T\stp}{\hat\bP_T\stp}(x_{[0,T]}\stp)>L} + \bP_T[\tau'\le T]\\
        &\le_{\eqref{e:kl-markov}\text{, assumption 1}} \fc{\KL( \bP_T\stp\|\hat\bP_T\stp)+1}{\log L} + \de'
        \le_{\eqref{e:KL-stop}} \fc{\rc 2 \int_0^T \ep(t)^2dt + 1}{\log L} + \de'\le %
        \fc{\ep}3 + \de'
    \end{align*}
    by choosing $L=\exp\pa{\fc{3}{\ep}\pa{\fc 12 \int_0^T \ep(t)^2dt+1}}$. 
    Next, 
    note that directly by assumption 3, $\P[S_T\cap B_2]\le \ep_3$.
    Hence we can bound the second term
    \begin{align*}
        \E\ba{\fc{e^{\hw_T}}{\hat Z_T}\fc{R_2(\hx_T)}{Z_2}\one_{S_T\cap G_1\cap B_2} }
        &= \E\ba{\dd{\bP_T^s}{\hat \bP_T}(\hx_{[0,T]})\one_{S_T\cap G_1\cap B_2}} \le L \P[S_T\cap B_2] \le L\ep_3.
    \end{align*}
    Finally, note that under $S_T\cap G_1\cap G_2$, 
    $\fc{R_2(\hx_T)}{Z_2}= \dd{\bP_T^s}{\hat\bP_T}\big/ \fc{e^{\hw_T}}{\hat Z_T} \le \fc{L}{c_3}$, so 
    for the third term, by assumption 2,
    \begin{align*}
        \E\ba{\fc{e^{\hw_T}}{\hat Z_T}\fc{R_2(\hx_T)}{Z_2}\one_{S_T\cap G_1\cap G_2\cap B_3} }
        &\le 
        \fc{L}{c_3}\E\ba{\fc{e^{\hw_T}}{\hat Z_T}\one_{S_T\cap B_3}}
        \le \fc{L\ep_1}{c_3}.
    \end{align*}
    Putting everything together, 
    \begin{align*}
        \E\ba{\fc{e^{\hw_T}}{\hat Z_T} \fc{R_2(\hx_T)}{Z_2}\one_{S_T\cap (B_1\cup B_2 \cup B_3)}}
        &\le \fc{\ep}3 + \de'+ L\ep_3 + \fc{L\ep_1}{c_3}=\ep+\de'
    \end{align*}
    by choosing $\ep_3 = \fc{\ep}{3L}$ and $\ep_1 = \fc{c_3\ep}{3L}$. Hence 
    \begin{align*}
        \E\ba{\one_{S_T}
            \fc{e^{\hw_T}}{\hat Z_T} \fc{R_2(\hx_T)}{Z_2}\wedge \fc{C_1L}{c_3}
        } &\ge \E\ba{\fc{e^{\hw_T}}{\hat Z_T} \fc{R_2(\hx_T)}{Z_2}\one_{S_T\cap G_1\cap G_2 \cap G_3}}\\
        &= \E\ba{\one_{S_T}\dd{\bP_T^s}{\hat  \bP_T}(\hx_T)} - 
        \E\ba{\fc{e^{\hw_T}}{\hat Z_T} \fc{R_2(\hx_T)}{Z_2}\one_{S_T\cap (B_1\cup B_2\cup B_3)}}\\
        &= \E\ba{\one_{S_T}(x_{[0,T]})} - \E\ba{\fc{e^{\hw_T}}{\hat Z_T} \fc{R_2(\hx_T)}{Z_2}\one_{S_T\cap (B_1\cup B_2\cup B_3)}}
        \ge 1-\de-\de' - \ep.
    \end{align*}  
    The first part now follows from \Cref{l:ars} applied to path measures $Q = \hat \bP_T$ and $P=\bP_T$, with $\ewt = \ep+\de+\de'$, $\de_1 = \de_2 = \ert = \esm =0$, and then taking the marginal at time $T$. 
    Note we allow probability $\ep$ of failure (in the quantile estimate or rejection sampling taking too many trials).

    For the second part, let $\td Q = \td  \bP_{T}$ be the path measure of $(\td x_t)_{t\in [0,T]}$. 
    Let 
    $\td R = e^{\td w_T}\td R_2$, and $R = e^{\hw_T}R_2$. 
    Note that assumption \ref{i:je-path-error} gives assumption \ref{i:esm} of \Cref{l:ars} with $\esm=\ep_4$. Consider the events
    \begin{align*}
        G_1' &= \bc{\fc{\td R_2}{R_2}\in [0,e^{\ep_4/2}]} \cap 
        \bc{|\tw_T - \hw_T|\le \fc{\ep_4}2}
        \subeq \bc{\fc{\td R}{R}\in [0,e^{\ep_4}]}.
        \\
        G_2' &= \bc{\fc{\td R_2}{R_2}\in [e^{-\ep_4/2}, e^{\ep_4/2}]},
    \end{align*}
    and note $G_1'\cap G_2'\subeq \bc{\fc{\td R}{R}\in [e^{-\ep_4},e^{\ep_4}]}$. It remains to check assumption \ref{i:ert} of \Cref{l:ars}. We check
    \begin{align*}
        \td \bP_T({G_1'}^c)\le \td \bP_T^s({G_1'}^c) + \hat\de   
        &\le \P\pa{\fc{\td R_2(\tx_T)}{R_2(\td x_T)}\nin [0,e^{\ep_4/2}]}
        + \td\bP_T^s\pa{|\tw_T - \hw_T|>\fc{\ep_4}2} + \hat \de
        \le_{\eqref{i:je-ewt}, \eqref{i:je-ert}(a)} \fc{\ep_4}2+\fc{\ep_4}2 + \hat\de = \ep_4 + \hat\de\\
        \bP_T({G_1'}^c)
        &= \P \pa{\fc{\td R_2(\tx_T)}{R_2(\td x_T)}\nin [e^{-\ep_4/2},e^{\ep_4/2}]
        } \le_{\eqref{i:je-ert}(b)} \ep.
    \end{align*}
    Thus the assumption holds with $\ewt = \ep+\de+\de'$, $\de_2=\ep$,  $\de_1=\esm=\ep_4+\hat\de$, and  $\ert=\ep_4$. The final TV error of the sampled distribution $\hat P$ is
    \begin{align*}
\TV(\hat P, P)
&\le \ewt + \de_2 + 384 C_1'C_2'(\de_1+\ert+ \esm)\\
&\le \ep +\de+\de'+\ep + 384 {C_1'}^2(\ep_4+ \hat \de+\ep_4+\ep_4+\hat\de) = 5\ep + \de +\de'+  768{C_1'}^2\hat \de.
    \end{align*}
    Allowing probability $\ep$ of failure then gives the result.
\end{prf}

\end{document}